\documentclass[11pt]{article}

\usepackage[utf8]{inputenc} 

\usepackage{xr} 
\usepackage{authblk}

\usepackage{comment}
\usepackage{lipsum}
\usepackage[usenames]{xcolor}
\usepackage{bm} 
\usepackage{eufrak}
\usepackage{upgreek}
\usepackage{enumitem}
\usepackage{rsfso}
\usepackage{abraces} 
\newcommand{\revsag}{\textcolor{black}}
\usepackage[letterpaper,
			left = 0.7truein,  
			right = 0.7truein, 
			top = 0.7truein, 
			bottom = 0.7truein]{geometry} 
\usepackage{amsthm, amsfonts, amssymb, amsmath,enumitem, bbm, mathabx}
\usepackage{mathtools}
\mathtoolsset{showonlyrefs,showmanualtags}
\numberwithin{equation}{section}

\newtheorem{defn}{Definition}[section]
\newtheorem{thm}{Theorem}[section]
\newtheorem{lem}{Lemma}[section]

\newtheorem{prop}{Proposition}[section]
\theoremstyle{remark}
\newtheorem{rem}{Remark}[section] 
\newcommand{\E}{\mathbb{E}}
\newcommand{\Pro}{\mathbb{P}}
\newcommand{\Var}{\mbox{Var}}

\usepackage{subcaption}
\usepackage{hyperref}
\usepackage{array}

\newcommand{\sech}{\mbox{sech}}
\usepackage{graphicx} 
\usepackage{color, float}

\usepackage{booktabs} 
\usepackage{array} 
\usepackage{verbatim} 
\usepackage[numbers]{natbib} 

\title{Detecting Planted Partition in Sparse Multi-Layer Networks}
\author[1]{Anirban Chatterjee}
\author[2]{Sagnik Nandy}
\author[3]{Ritwik Sadhu}
\affil[1]{\emph{University of Pennsylvania}}
\affil[2]{\emph{University of Pennsylvania}}
\affil[3]{\emph{Cornell University}}
\begin{document}
	\maketitle

\abstract{Multilayer networks are used to represent the interdependence between the relational data of individuals interacting with each other via different types of relationships. To study the information-theoretic phase transitions in detecting the presence of planted partition among the nodes of a multi-layer network with additional nodewise covariate information and diverging average degree, \citet{ma_nandy} introduced \emph{Multi-Layer Contextual Stochastic Block Model}. In this paper, we consider the problem of detecting planted partitions in the Multi-Layer Contextual Stochastic Block Model, when the average node degrees for each network is greater than $1$.
We establish the sharp phase transition threshold for detecting the planted bi-partition. Above the phase-transition threshold testing the presence of a bi-partition is possible, whereas below the threshold no procedure to identify the planted bi-partition can perform better than random guessing. We further establish that the derived detection threshold coincides with the threshold for weak recovery of the partition and provide a quasi-polynomial time algorithm to estimate it.} 

\maketitle
\section{Introduction}
Relational data generated in diverse fields of scientific study such as genomics, ecology, information science, etc., are naturally represented by networks. Such relational data of individuals interacting with each other through multiple types of relationships can be represented by multiple dependent networks sharing the same set of vertices. This collection of networks is termed a \emph{multiplex network}. Here the node set shared among the networks represents the subjects involved in the interaction and the edge sets of different layers represent different types of relationships between them. For example, in their study on the diffusion of participation in a microfinance program through social networks, \citet{banerjee2013diffusion} considered interactions among individuals in villages of southern India happening through social as well as financial exchanges, and \citet{lazega2001collegial} studied co-worker, advice and friendship networks between lawyers in a corporate law firm in the Northeastern United States. Multiplex networks have also been studied in social networks literature to represent multiple channels of relationship, for example, \citet{borondo2015multiple} studies three types (follow, mention, and retweet) of Twitter interaction among individuals.

The problem of detecting latent structures planted uniformly in an Erd\H{o}s-R\'enyi random graph is an important area of research in probability and mathematical statistics. In the single network Erd\H{o}s-R\'enyi model, a vast literature has been developed to study the problems of detecting hidden cliques \cite{Deshpande2015}, planted partitions \cite{mns14,massoulie2013community, AbbeMonYash,mossel2022exact}, or planted subgraphs \cite{pmlr-v99-massoulie19a,doi:10.1287/opre.2019.1886}.

The problem of detecting planted partitions of the nodes is of greater interest to statisticians because of its alternative interpretation as the problem of clustering the nodes of the network into interpretable communities, also known as the \emph{community detection problem}. In fact, we shall use the terms community detection and recovering planted bi-partition interchangeably throughout this paper. The community detection problem has found applications in diverse fields like the study of sociological interactions \cite{Fortunato_2010}, gene expressions \cite{1339264}, and recommendation systems \cite{1167344}, among others. A number of generative models have been suggested to understand the performance of different network clustering algorithms, the most popular being the Stochastic Block Model (SBM) \cite{holland1983stochastic}. The SBM's are quite accurate in modeling several real-world networks and exhibit rich phase transition phenomena (see, \cite{Decelle_2011,massoulie2013community,mns14,mnsf}). These stylized models are used to study the interplay between the statistical and computational barriers of community detection in networks with a planted partition (see, \cite{KC,YC}). We refer the reader to \citet{abbe2017community} for a detailed survey. 

It is worthwhile to note that while several works in the literature have considered the problem of detecting planted structures in a single network, very few works have considered the same problem in multiplex networks, leaving a significant gap in the literature. It is important to note that while it is possible to study a multiplex network system by collapsing the layers into a single layer (see \cite{el2020orthonet}), such collapsing leads to the loss of important information and it is imperative that the method of such collapse must be carefully chosen (see, \cite{taylor2016enhanced,taylor2017super}). While aggregating multiple networks facilitate the loss of information, studying each network separately fails to take into account the inter-dependencies among the multiple layers \cite{zhang2020flexible}. In contrast, studying such a collection of networks as a multiplex network system helps to prevent the loss of information as well as allows for the study of interdependence between the layers.  In this paper, we study the statistical thresholds for detecting the presence of a planted partition among the shared vertices in a multiplex network system. 

In that direction, we consider the stylized model proposed by \citet{ma_nandy} called the \emph{Contextual Multi-Layer Networks}.
In particular, we consider $m$ undirected networks on $n$ vertices with a planted bi-partition $\bm \sigma$ of the shared vertices. The entries of $\bm{\sigma} = (\sigma_{i})_{i=1}^{n}$ are $+1$ or $-1$ denoting the partition to which the $i$-th node belongs. Henceforth, we shall denote the set $\{1,2,\cdots,n\}$ by $[n]$.

For $1 \le i \le n$, $\sigma_i$'s are sampled uniformly from $\{-1,+1\}$, or in other words, the subjects are randomly assigned to the partitions. This uniform sampling enforces an approximate balance in the size of each partition. Let the $m$ un-directed networks be denoted by $\{\bm{G}_{i}:i\in [m]\}$. Conditional on $\bm{\sigma}, \{\bm{G}_{1},\cdots, \bm{G}_{m}\}$ are mutually independent and for each $\bm{G}_{k}$ the edges are generated independently with probability
\begin{align*}
    \mathbb{P}\left[(i,j)\in E(\bm{G}_{k})\middle|\bm{\sigma}\right] = 
    \begin{cases}
    \frac{a_{k}}{n} & \text{ if }\sigma_{i} = \sigma_{j}\\
    \frac{b_{k}}{n} & \text{ if }\sigma_{i}\neq \sigma_{j} 
    \end{cases}
    \text{ with } a_{k}>b_{k},\text{ for all }1\leq k\leq m.
\end{align*}
Here $E(\bm{G}_{k})$ denotes the edge set of the random graph $\bm G_k$. For $1\leq k\leq n$,  we define $\bm{A}_{k} := \left(A_{ij}^{(k)}\right)\in\{0,1\}^{n\times n}$ to be the adjacency matrix of the network $\bm{G}_{k}$ and $d_{k} := \frac{a_{k} + b_{k}}{2}$ to be the average degree of the network. Consider the following re-parametrization of the connection probabilities:
\begin{align*}
    a_{k} = d_{k} + \lambda_{k}\sqrt{d_{k}},\text{ and }b_{k} = d_{k} - \lambda_{k}\sqrt{d_{k}},
\end{align*}
where $\lambda_k\geq 0$, is the signal to noise ratio for the $k$-th network is defined as,
\begin{align*}
    \lambda_k^2:=\frac{(a_k-b_k)^2}{2(a_k+b_k)}.
\end{align*}

In this paper, we shall consider the regime where $d_k > 1$ for all graphs $\bm G_k$ where $1 \le k \le m$.
In addition to the networks $\bm G_1,\ldots,\bm G_m$, we also observe a $n\times p$ dimensional data matrix $\bm{B} = \left[B_{1},B_{2},\cdots, B_{n}\right] \in \mathbb{R}^{p\times n}$, where the $i^{th}$ column represents a $p-$dimensional covariate vector  corresponding to the $i^{th}$ individual. We assume that the $i$-th column of $\bm B$ is distributed as,
\begin{align}\label{eq:def-B}
    B_{i} = \sqrt{\dfrac{\mu}{n}}\sigma_{i}\bm{u} + \bm{Z}_{i}, \ i\in [n]
\end{align}
where $\bm{u}\sim \mathrm{N}(\bm{0}, \bm{I}_{p\times p})$ is a latent Gaussian vector, and $\{\bm Z_{i}\}_{i=1}^{n}$ are generated IID from $\mathrm{N}\left(\bm{0},\bm{I}_{p\times p}\right)$. In other words, given $\bm u$ and $\bm \sigma$, the columns of $\bm B$ are generated from a two-component Gaussian Mixture model where the component means are given by $\sqrt{\frac{\mu}{n}}\bm u$ and $-\sqrt{\frac{\mu}{n}}\bm u$ and the mixing proportion is equal to $1/2$. It can be easily seen that, by definition, $\bm{B}$ is independent of $\{\bm{G}_{k}:k\in [m]\}$ given the planted partition $\bm{\sigma}$. 
The parameter $\mu$ represents the signal-to-noise ratio of $\bm B$. For $m=1$, the above model recovers the \emph{contextual stochastic block model} of ~\cite{lu2020contextual}. We assume $\mu\geq 0$, and consider the the high dimensional proportional regime, where $\dfrac{n}{p}\rightarrow\gamma\in (0,\infty)$. 

A typical example of a data set that can reasonably be modeled by the Contextual Multilayer Network is the data on the social support network of rural Indian villages studied in \cite{contisciani2020community}. Here the authors consider $L=6$ networks, each depicting a particular type of mutual support offered between individuals, and based on ethnographic observations, caste, and religion were used as the additional covariate information. Another example of this scenario was described in \cite{mm} where the nodes represent proteins, whereas the edges in one network represent physical interactions between nodes and those in another network represent co-memberships in protein complexes. 
\subsection{Survey of related literature}
The methods used in this paper are motivated by the seminal papers \cite{massoulie2013community,mns14,MNS14a,mnsf} where the authors considered the same problem with a single network. \revsag{In \citet{lu2020contextual}, the authors extended the set-up considered in \citet{mnsf} to a more general setting where the network information is complemented by a set of side information in the form of Gaussian covariates. In the nondiverging average degree regime, they established the sharp phase transition threshold above which the construction of a nontrivial estimate of the planted bipartition is possible by a polynomial time algorithm. Below the mentioned phase transition threshold, it is impossible to reconstruct the planted partition with greater accuracy than what is achieved by an algorithm that randomly guesses the labels of the nodes. In fact, the authors established that below the mentioned phase transition threshold, it is impossible to consistently test for the presence of a planted partition in the underlying model. In particular, the authors showed that if the combined signal-to-noise ratio of the network and the covariates are below the phase transition threshold, the model is contiguous to a model with no planted partition. Above the phase-transition threshold, the authors constructed a consistent test to detect the presence of a planted partition in the underlying model. They also proposed an algorithm based on counting self-avoiding walks on an appropriately designed factor graph that can weakly recover the planted partition in quasi-polynomial time when the combined signal-to-noise ratio is above the detection threshold.} 

In \cite{ma_nandy}, the authors considered a new model named \emph{Contextual Multilayer Networks} where in addition to the node covariates, one observes multiple random networks with the same community structure but different connection probabilities. They derived the limit of mutual information between the latent communities and the observed data using a version of the \emph{Approximate Message Passing} algorithm. However, their work considers the diverging average degree regime and the method described there breaks down when one considers the regime with nondiverging average degrees for each network. The sparser regime, considered in this paper, with a near-constant average degree for each of the networks, is more common in real applications where small-world networks tend to be of the predominant type. Moreover, prior works in this direction do not provide an algorithm implementable in polynomial time to weakly recover the latent community structure from the Multi-layer Contextual Network data. This leaves a significant gap in the literature. 

\revsag{We adapt the mathematical techniques of \citet{lu2020contextual} to characterize the sharp thresholds for detecting and weakly recovering a planted bi-partition in \emph{Contextual Multilayer Networks} with $m$ networks on the same set of vertices and a $p \times n$ Gaussian covariate matrix. In the special case when $m=1$ our results recover the results of \citet{lu2020contextual}. Akin to the analysis of \citet{lu2020contextual} for $m=1$ (in particular, Theorem 1 and Proposition 5), in this paper, we establish a detectability threshold for the planted bi-partition in the sparse Contextual Multilayer Stochastic Block Models and prove an ANOVA-type decomposition of the variability in log-likelihood of the model. Our analysis applies to networks with a growing number of vertices but constant average degrees even when the number of networks $m \ge 1$. }

\revsag{However, it is worthwhile to mention that the optimal way to combine information across different layers of the multi-layer network to achieve consistent detection and weak recovery of the planted partition for the weakest possible signal-to-noise ratio is not obvious. We achieve this by constructing an appropriate factor graph that aggregates all the information about the planted partition in an efficient way. Therefore, a test based on graph statistics computed using this factor graph consistently detects the presence of a planted bi-partition up to the information-theoretic threshold. While the mathematical techniques used in this paper are adapted from previous works of \citet{massoulie2013community,mns14,MNS14a,lu2020contextual}, it is the construction of this factor graph that underscores the main technical novelty of this work. Deriving information theoretic phase transition for weak recovery in the multilayer set-up using graph statistics computed from such aggregated factor graphs has not been explored in the literature before. It is also worthwhile to mention that no previous works had considered developing a Belief Propagation algorithm in a multi-layer network setup. Belief Propagation algorithms have been considered for the single layer case, i.e., $m=1$ (see, \cite{lu2020contextual,duranthon2024optimal}). The main novelty of the BP Algorithm provided in this paper is the use of the aggregated factor graph. This aggregation technique allows the BP Algorithm to achieve high accuracy in weak recovering the planted bi-partition for extremely low signal-to-ratio up to the information-theoretic threshold for weak recovery.}

\revsag{Furthermore, \citet{agterberg2022joint} considered the multi-layer network set-up in higher signal-to-noise ratio regimes and proposed a spectral algorithm to recover the latent communities. Similar results have also been explored in \citet{paul2018random} where the authors derived the minimax lower bound of the mis-clustering rate but failed to provide a polynomial time implementable algorithm to achieve the lower bound. However, the methods described in this paper are not appropriate for deriving the minimax lower bound of the mis-clustering rate. For that purpose, it is essential to consider a more delicate construction of the least favorable prior. Our results, on the other hand, delineate the information-theoretic lower bound on testing. As observed in different high-dimensional problems, consistent testing for the presence of a planted partition within this framework is possible in a wider signal-to-noise ratio regime than recovering such a partition. } 

\revsag{Various applications of multiplex network data have also been studied in \cite{NIPS,inproceedings,witten,GW_2,9173970,chen2021global,racz2021correlated, ghasemian2016detectability, xu2022covariate, kumar2019effect}. Very recently, \citet{yang2024fundamental} have studied a more generic model where the community assignments across different layers of the networks can be correlated but not identical. However, they can prove the information-theoretic threshold for the detection of the underlying partitions only in the diverging average degree regime.} 

\subsection{Summary of our results}
Our major contributions to this paper are as follows:
\begin{enumerate}
    \item In Section \ref{sec:test}, we consider the problem of testing for the presence of a planted bi-partition among the shared nodes of the multiplex network. We show that the model with the planted bi-partition is contiguous to the model without the planted structure if $\lambda^2_1+\cdots+\lambda^2_m+\mu^2/\gamma<1$. On the other hand, we construct a sequence of consistent tests for the stated hypotheses when $\lambda^2_1+\cdots+\lambda^2_m+\mu^2/\gamma>1$. It is worthwhile to note that for a single network $G_{k}$ without covariate information, the detection threshold is $\lambda^2_k>1$ while only observing the covariate information changes the detection threshold to $\frac{\mu^2}{\gamma}>1$. Thus, combining information across multiple layer helps to boost the detection threshold of each network additively. In the contiguity regime, we further derive the limit of the likelihood ratio in terms of appropriate cycle statistics providing a decomposition of the randomness present in the log-likelihood ratio.
    \item In Section \ref{subsec: weak_recovery}, we consider the problem of \emph{weak recovery} of the planted bi-partition, i.e., constructing an estimator of the partition with a positive correlation with the ground truth. We show that the signal-to-noise ratio threshold below which \emph{weak recovery} is impossible coincides with the threshold below which consistently testing for the presence of a partition is impossible. Above the threshold, \emph{weak recovery} of the latent partition is possible and we provide a quasi-polynomial time algorithm to achieve the same. 
    \item In Section \ref{Belief_Propagation}, keeping in mind the impracticality of the above proposed quasi-polynomial time algorithm, we propose a Belief Propagation-based algorithm to approximately compute the MAP estimator of the partition $\bm \sigma$ and in Section 5, we provide numerical simulations to demonstrate the performance of our Belief Propagation Algorithm. Codes for all the experiments can be found in the Github repository \href{https://github.com/anirbanc96/Sparse-MCSBM}{https://github.com/anirbanc96/Sparse-MCSBM}.
\end{enumerate}

\section{Testing the presence of Planted Partition}\label{sec:test}
Let us consider the following hypothesis testing problem.
\begin{equation}
\label{eq:hyp_1}
\bm{H}_0: (\lambda_1,\cdots,\lambda_m,\mu)=(0,\cdots,0,0) \quad \quad \mbox{vs} \quad \quad \bm{H}_{1}: (\lambda_1,\cdots,\lambda_m,\mu)\neq (0,\cdots,0,0).
\end{equation}
Consider $\mathbb{P}_{\bm{0},0}$ to be the joint distribution of the $m$ networks and the covariate matrix when $\lambda_k=0$ for all $k$ and $\mu=0$. Similarly, $\mathbb{P}_{\bm{\lambda},\mu}$ be the same joint distribution with non-trivial signal-to-noise parameters. 

Observe that the null hypothesis refers to the setup when there is no planted partition. It is imperative that for very small values of $\lambda_i$'s and $\mu$, the null will be indistinguishable from the alternative. In such a situation, when the values of $\lambda_i$'s and $\mu$ are all very small, we cannot hope to get any estimator of $\bm \sigma$ that is positively correlated with the ground truth. However, in a multiplex network system, we can afford to have some of the $\lambda_k$'s to be small, as long as the combined effect of all the $\lambda_k$'s and $\mu$ is substantial. We formalize this intuition in the following theorem, whose proof is given in Appendix \ref{appendix:proofofthm1}.
\begin{thm}
\label{thm:detection_threshold}
 If $\lambda^2_1+\cdots+\lambda^2_m+\mu^2/\gamma<1$, then $\mathbb{P}_{\bm{\lambda},\mu}$ is contiguous to $\mathbb{P}_{\bm{0},0}$. On the other hand, if $\lambda^2_1+\cdots+\lambda^2_m+\mu^2/\gamma>1$, the distributions are asymptotically mutually singular.
\end{thm}
Let us recall that the signal-to-noise ratio is $\lambda^2_k$ if we only observe the $k$-th network in isolation and is $\mu^2/\gamma$ if we only observe the covariate matrix. The effective signal-to-noise ratio of the entire multilayer network system comprising all the $m$ networks and the data matrix is $\lambda^2_1+\cdots+\lambda^2_m+\mu^2/\gamma$. 
Theorem \ref{thm:detection_threshold} implies that if the effective signal-to-noise ratio is less than $1$, then we cannot have a consistent test between the null and the alternative. We observe that the signal-to-noise ratio from each of the individual data sources is boosted additively by the combination of information across different data sources. This allows us to test for the presence of planted partition using the combined information even when it is impossible to test with non-trivial power for the presence of the planted partition utilizing any component of the data in isolation.

 When the effective signal-to-noise ratio is larger than $1$ we construct a sequence of consistent tests for the hypotheses \eqref{eq:hyp_1} which exhibits the mutual singularity of $\mathbb{P}_{\bm{\lambda},\mu}$ and $\mathbb{P}_{\bm{0},0}$. In the following paragraph, we consider a class of short cycle statistics on an appropriately defined \emph{Factor Graph}, $G_F=(V_F, E_F)$. If $\lambda^2_1+\cdots+\lambda^2_m+\mu^2/\gamma>1$, this statistic has different asymptotic distributions under $H_0$ and $H_1$. The mean of our test statistic is significantly elevated under the alternative which provides a way to construct the above-mentioned sequence of consistent tests for testing $\bm H_0$ vs $\bm H_1$.

\paragraph{Cycle Statistics and the Detecting Test.}
\label{test_section}
Let us consider a factor graph $G_F=(V_F, E_F)$. 
Consider $V_{0}$ to be a collection of $n$ nodes, \revsag{referred to as the \emph{variable nodes}}, representing the $n$ subjects under study, and $m+1$ layers of factor nodes. For each network $\bm G_k$, we have a layer of factor nodes representing each of $A^{(k)}_{ij}$, for all $1 \le i<j\le n$, collectively denoted by $V_{k1}$. Additionally, we have $p$ factor nodes for the $p$ covariates, denoted by $V_2$. Now, collecting all the above nodes construct, $V_F=V_0\cup V_{11}\cup \cdots \cup V_{m1}\cup V_2$. For all $1 \le i < j \le n$ and $k \in [m]$, the factor node corresponding to $A^{(k)}_{ij}$ is connected to the variable nodes corresponding to $i$ and $j$ where the edge weights are given by $A^{(k)}_{ij}$. Let the edges between factor nodes of layer $k$ and the variable nodes be denoted by $E_{1_{k}}$. The subgraph made of the $p$ factor nodes corresponding to the $p$ covariates and the $n$ variable nodes is a complete bipartite graph. Let these edges be denoted by $E_2$ and each edge weight for edge $e=(i,j)$, where $i\in V_0$ and $j \in V_2$, be $B_{ij}$. Thus, the edge set $E_F$ is given by $E_F=E_{1_{1}}\cup E_{1_{2}}\cup \cdots \cup E_{1_{m}} \cup E_{2}$. An illustration of this factor graph is given in Figure \ref{fig:factor-graph}.  

Constructing the factor graph $G_F$ in the above-described manner allows for the optimal combination of information across the data sources. In effect, the factor graph $G_F$ captures all the available information about $\bm \sigma$ from $\bm G_1,\ldots,\bm G_m$ and $\bm B$.

\begin{figure}
    \centering
    \includegraphics[width = 0.8\linewidth]{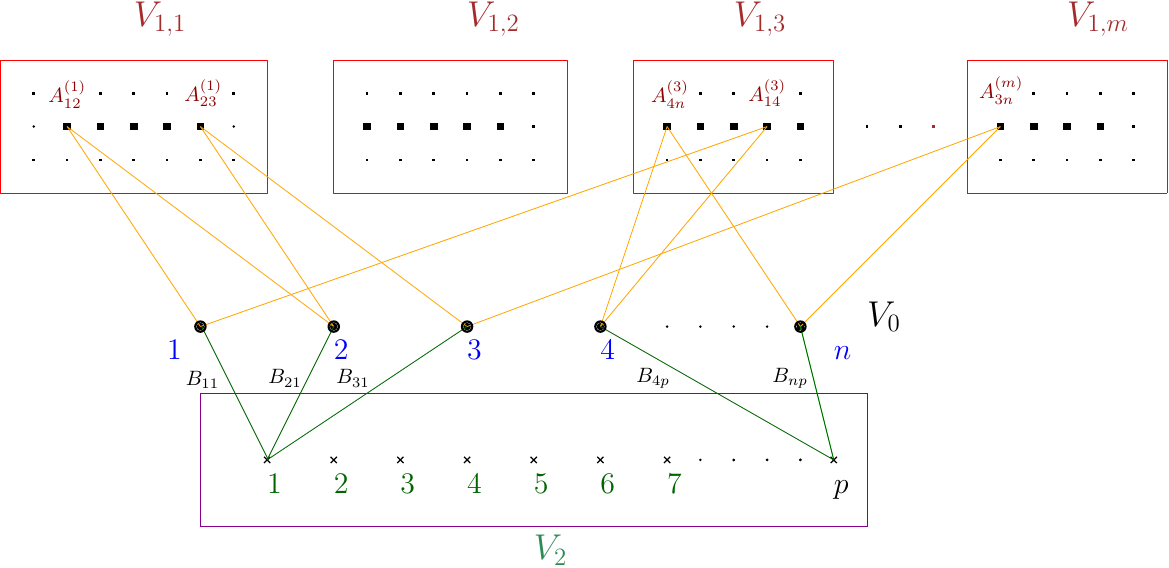}
    \caption{Example of Factor Graph where $V_{0}$ is the vertex set comprising of $n$ variable nodes, $V_{1,k},1\leq k\leq m$ are the factor nodes corresponding to networks $\bm{G}_{k},1\leq k\leq m$ and $V_{2}$ are the factor nodes corresponding to the $p$ covariates, where corresponding edge weights are indicated to the left of such edge.}
    \label{fig:factor-graph}
\end{figure}
Let us consider cycles on this factor graph. We observe that all the edges from $E_{1_{k}}$ or $E_{2}$ occur in pairs. Such pairs are referred to as wedges. If it is made of two edges of type $E_{1_{k}}$, for $1 \le k \le m$, then we call it a $\bm{A}_{k}$ type wedge and if it is made of two edges of type $E_2$, then we call it a $\bm{B}$ type wedge. We denote a cycle with $k_r$ type $\bm{A}_{r}$ wedges for $1 \le r \le m$ and $\ell$ type $\bm{B}$ wedges by $\omega_{k_{1},\cdots,k_m,\ell}$ and the subgraph induced by this cycle be denoted by $(V_{\omega}, E_{\omega})$. Let the edges of $\omega_{k_1,\cdots,k_m,\ell}$ which are of type $E_{1_{k}}$ be denoted by $E_{\omega,1_{k}}$ and those which are of type $E_2$ be denoted by $E_{\omega,2}$. Now, consider the following statistic.
\[
Y_{k_1,k_2,\cdots,k_m,\ell}:=\frac{1}{n^\ell}\sum\limits_{\omega_{k_1,\cdots,k_m,\ell}}\prod\limits_{e_1 \in E_{\omega,1_{1}}}\cdots\prod\limits_{e_m \in E_{\omega,1_{m}}}\prod\limits_{e_\ell \in E_{\omega,2}}A^{(1)}_{e_1}\cdots A^{(m)}_{e_m}B_{e_\ell}.
\]
This statistic counts the number of cycles in the observed networks with $k_r$ type $\bm{A}_{r}$ wedge for $1 \le r \le m$ and $\ell$ type $\bm{B}$ wedges; weighted by the product of the edge weights of type $E_2$ edges and scaled by $n^\ell$. The following theorem characterizes the distribution of $Y_{k_1,k_2,\cdots,k_m,\ell}$ under the null and the alternative generative models from \eqref{eq:hyp_1}.

\begin{thm}\label{thm:cycle-stat-asymp-dist}
Let $k=k_1+\cdots+k_m+\ell$. Then as $n \rightarrow \infty$,
\begin{enumerate}
    \item Under $H_0$,
    \[
    Y_{k_1,k_2,\cdots,k_m,0}\overset{d}{\rightarrow}\mathrm{Poi}\left(\frac{1}{2k}\;\frac{k!}{k_1!\,k_2!\cdots k_m!}\left\{d^{k_1}_1\,d^{k_2}_{2}\,\cdots d^{k_m}_m\right\}\right),
    \]
    and,
    \[
    \frac{Y_{k_1,k_2,\cdots,k_m,\ell}}{\sqrt{\dfrac{1}{2k\gamma^{\ell}}\;\dfrac{k!}{\ell!k_1!\,k_2!\cdots k_m!}\prod\limits_{j=1}^{m}d_{j}^{k_j}}}\overset{d}{\rightarrow} \mathrm{N}(0,1).
    \]
    \item Under $H_1$,
    \[
    Y_{k_1,k_2,\cdots,k_m,0}\overset{d}{\rightarrow}\mathrm{Poi}\left(\frac{1}{2k}\;\frac{k!}{\prod_{j=1}^{m}k_{j}!}\left\{\prod_{j=1}^{m}d_{j}^{k_{j}} +\prod_{j=1}^{m}(\lambda_j\sqrt{d_j})^{k_j}\right\}\right),
    \]
    and,
    \[
    \frac{Y_{k_1,k_2,\cdots,k_m,\ell}-\dfrac{1}{2k}\;\dfrac{k!}{\ell!\prod_{j=1}^{m}k_{j}!}\prod\limits_{j=1}^{m}(\lambda_j\sqrt{d_j})^{k_j}\left(\dfrac{\mu}{\gamma}\right)^{\ell}}{\sqrt{\dfrac{1}{2k\gamma^{\ell}}\;\dfrac{k!}{\ell!k_1!\,k_2!\cdots k_m!}\prod\limits_{j=1}^{m}d_{j}^{k_j}}}\overset{d}{\rightarrow} \mathrm{N}(0,1).
    \]
\end{enumerate}
Further $\{Y_{k_1,\cdots,k_m,\ell}\}$ are asymptotically independent and the asymptotic distribution continues to hold for $k=O(\log^{1/4} n)$.
\end{thm}
\revsag{From the above result, we conclude that as $n \rightarrow \infty$, compared to $H_0$, in the alternative $H_1$, we gain additional signal characterized by 
\[
\dfrac{1}{2k}\;\dfrac{k!}{\ell!\prod_{j=1}^{m}k_{j}!}\prod\limits_{j=1}^{m}(\lambda_j\sqrt{d_j})^{k_j}\left(\dfrac{\mu}{\gamma}\right)^{\ell}.
\]
It shall be shown in Section \ref{appendix:proofofthm1} that this signal is significant compared to the inherent noise in the statistics $Y_{k_1,k_2,\cdots,k_m,\ell}$, if and only if, $\lambda^2+\mu^2/\gamma >1$.
Hence, these statistics can be used to construct consistent tests for \eqref{eq:hyp_1}, if and only if, $\lambda^2+\mu^2/\gamma >1$.
}

To prove Theorem \ref{thm:cycle-stat-asymp-dist}, in the case when $\ell=0$ we use the Method of Moments and the techniques described in \cite{mnsf}. For $\ell \neq 0$ we rely on the method of moments and Wick's Formula. We defer the elaborate proofs to Appendix \ref{appendix:proofofthm2}.

\paragraph{Contiguity Regime and the Asymptotic Expansion of the Likelihood Ratio.}\label{def} In the contiguity regime, i.e., when $\sum_{i=1}^{m}\lambda^2_i+\mu^2/\gamma<1$, one can find an asymptotic expansion of the log Likelihood ratio as a weighted sum of the graph statistics $\{Y_{k_1,k_2,\cdots,k_m,\ell}\}$ for different $k_1,\cdots,k_\ell$ and $m$. Since these statistics are asymptotically independent by Theorem \ref{thm:cycle-stat-asymp-dist}; this provides an Analysis of Variance style decomposition of the randomness in the log-likelihood ratio. To obtain the expansion, we use the small subgraphs conditioning techniques described by \cite{janson_1995} and later extended by \cite{banerjee2018contiguity,Banerjee2018AsymptoticNA,lu2020contextual}. 

Let us define $\delta_{k_1,\cdots,k_m}$, $\lambda_{k_1,\cdots,k_m}$, $\mu_{k_1,\cdots,k_m,\ell}$ and $\sigma^2_{k_1,\cdots,k_m,\ell}$ for $k_1,\cdots,k_m,\ell \in \mathbb{N}\cup\{0\}$ as follows, 
\begin{equation}
\label{eq:first_set_parameters}
\delta_{k_1,\cdots,k_m}:=\prod\limits_{j=1}^{m}\left(\frac{a_i-b_i}{a_i+b_i}\right)^{k_j},\quad 
\mu_{k_1,\cdots,k_m,\ell}:=\frac{1}{2k}\;\frac{k!}{\ell!\prod_{j=1}^{m}k_{j}!}\prod\limits_{j=1}^{m}(\lambda_j\sqrt{d_j})^{k_j}\left(\frac{\mu}{\gamma}\right)^{\ell},
\end{equation}
and
\begin{equation}
\label{eq:second_set_parameters}
\lambda_{k_1,\cdots,k_m}:=\frac{1}{2k}\;\frac{k!}{\prod_{j=1}^{m}k_{j}!}\prod\limits_{j=1}^{m}d_j^{k_j}, \quad \sigma^2_{k_1,\cdots,k_m,\ell}:=\frac{1}{2k\gamma^\ell}\;\frac{k!}{\ell!\prod_{j=1}^{m}k_{j}!}\prod\limits_{j=1}^{m}d_j^{k_j}.
\end{equation}
Let $\nu_{(k_1,\cdots,k_m)}$ be a sequence of independent Poisson$(\lambda_{k_1,\cdots,k_m})$ random variables, and $Z_{(k_1,\cdots,k_m,\ell)}$ be another sequence of independent $N(0,\sigma^2_{k_1,\cdots,k_m,\ell})$ random variables. Then the following theorem describes the asymptotic expansion of the log-likelihood ratio.
\begin{thm}
\label{thm:contiguity}
Consider $\mathbb{P}_n$ to be the sequence of distributions under $\mathrm{H}_0$, and $\mathbb Q_n$ to be the sequence of distributions under $\mathrm{H}_1$. If $\sum_{i=1}^{m}\lambda^2_i+\mu^2/\gamma<1$, then under $\mathbb P_n$,
\begin{align}
\log\frac{d\mathbb Q_n}{d\mathbb{P}_n}&\overset{d}{\rightarrow}\sum\limits_{K=1}^{\infty}\Bigg\{\sum_{k_1+\cdots+k_m=K}\nu_{(k_1,\cdots,k_m)}\log(1+\delta_{k_1,\cdots,k_m})-\sum_{k_1+\cdots+k_m=r}\lambda_{k_1,\cdots,k_m}\delta_{k_1,\cdots,k_m}\\
&\hspace{1.3in}+\sum_{\substack{k_1+\cdots+k_m+\ell=K\\\ell \neq 0}}\frac{2\mu_{k_1,\cdots,k_m,\ell}Z_{(k_1,\cdots,k_m,\ell)}-\mu^2_{k_1,\cdots,k_m,\ell}}{2\sigma^2_{k_1,\cdots,k_m,\ell}}\Bigg\}.
\end{align}
Further,there exists a constant $K(\varepsilon,\delta)>0$, such that for all sequence $\{n_k\}$, there exists a further subsequence $\{n_{k_\ell}\}$ such that,
\begin{align}
\label{eq:approx_lik_ratio}
\limsup\limits_{\ell \rightarrow \infty}\mathbb P_{n_{k_\ell}}\Bigg[\Bigg|\log\frac{d\mathbb Q_n}{d\mathbb{P}_n}-\sum\limits_{r=1}^{K}\Bigg\{\sum_{k_1+\cdots+k_m=r}Y_{n_{k_\ell},k_1,\cdots,k_m,0}\log(1+\delta_{k_1,\cdots,k_m})-\sum_{k_1+\cdots+k_m=r}\lambda_{k_1,\cdots,k_m}\delta_{k_1,\cdots,k_m}\nonumber\\
\hspace{1.6in}+\sum_{\substack{k_1+\cdots+k_m+\ell=r\\\ell \neq 0}}\frac{2\mu_{k_1,\cdots,k_m,\ell}Y_{n_{k_\ell},k_1,\cdots,k_m,\ell}-\mu^2_{k_1,\cdots,k_m,\ell}}{2\sigma^2_{k_1,\cdots,k_m,\ell}}\Bigg\}\Bigg|\ge \varepsilon\Bigg]\le \delta.
\end{align}
\end{thm}
\revsag{This theorem decomposes the asymptotic distribution of the log-likelihood ratio in terms of the distribution of simpler cycle statistics. An alternative way to interpret this decomposition is as follows: The distribution of the log-likelihood ratio under the null approximately belongs to an exponential family with the sufficient statistics $\{Y_{n_{k_\ell},k_1,\cdots,k_m,\ell}\}$ and the natural parameters $\{\log(1+\delta_{k_1,\cdots,k_m})\}$ and $\{\mu_{k_1,\cdots,k_m,\ell}/\sigma^2_{k_1,\cdots,k_m,\ell}\}$. This decomposition can serve as a first step towards designing computationally efficient tests achieving the same optimal power as the likelihood ratio test for the hypotheses defined in \eqref{eq:hyp_1}. Since,  Theorem \ref{thm:contiguity} characterizes the asymptotic distribution of the log-likelihood ratio; to construct asymptotically level $\alpha$ test for the hypothesis \eqref{eq:hyp_1} it suffices to approximate the quantiles of the distribution outlined in \eqref{eq:approx_lik_ratio}. Typically, such quantiles are estimated by Monte-Carlo simulation of the statistic. However, counting the number of cycles of all possible sizes in a graph is an NP-hard problem and hence no polynomial time algorithm can be used to approximate these quantiles. It remains an interesting problem to study if there are good approximation algorithms that can be combined with Monte Carlo simulations to estimate such quantiles. We defer the proof of the theorem to Appendix \ref{appendix:proofofthm3}.}

\section{Weak Recovery of the Planted Partition}
\label{subsec: weak_recovery}

Next, we turn to the problem of the weak recovery of the partition $\bm{\sigma}$. 
\begin{defn}[\cite{lu2020contextual, ma_nandy}]
An estimator $\hat{\bm\sigma} := \hat{\bm\sigma}(\bm A_1, \dots, \bm A_m, \bm B)$ of $\bm \sigma$  achieves weak recovery under $\mathbb{P}_{\bm{\lambda}, \mu}$ if
\[
\liminf_{n\to\infty}\frac{1}{n}\E_{\bm\lambda, \mu}[|\langle \hat{\bm{\sigma}}, \bm{\sigma} \rangle | ] > 0.
\]
Weak recovery is said to be possible under $\mathbb{P}_{\bm{\lambda}, \mu}$ if there exists an estimator $\hat{\mathbf\sigma}$ that achieves weak recovery under $\mathbb{P}_{\bm{\lambda}, \mu}$.
\end{defn}

In \cite{ma_nandy}, it is shown that the weak recovery threshold coincides with the detection threshold mentioned in Theorem~\ref{thm:detection_threshold} for contextual multilayer networks with diverging average degree parameters $ d_k$s. The result below generalizes this to the almost constant average degree setup of the current paper.

\begin{thm}
\label{thm:weak_recovery_threshold}
When $\sum_{k=1}^m \lambda_k^2 + \mu^2/\gamma < 1$ and $d_k \ge 1$ for $1 \le k \le m$, for any estimator $\hat{\bm\sigma}$ of $\bm \sigma$, we have
\[
\frac{1}{n}\E [ |\langle \hat{\bm\sigma}, \bm\sigma \rangle| ] \to 0,
\]
i.e., weak recovery is not possible. On the other hand, when $\sum_{k=1}^m \lambda_k^2 + \mu^2/\gamma > 1$, weak recovery is possible.
\end{thm}

Under the contiguity regime $\sum_{k=1}^m \lambda_k^2 + \mu^2/\gamma < 1$, the impossibility of weak recovery follows by analyzing the posterior distribution of the components $\bm\sigma_u$ of $\bm\sigma$ given the data $\bm A_1, \dots, \bm A_m, \bm B$, and a disjoint component $\bm\sigma_S$ (where $u \not\in S \subset [n]$) of $\bm \sigma$. 

When $\sum_{k=1}^m \lambda_k^2 + \mu^2/\gamma > 1$, we describe in the following, a quasi-polynomial time algorithm using self-avoiding walks on the factor graph described before to compute an estimator that achieves weak recovery. The proof of Theorem \ref{thm:weak_recovery_threshold} is deferred to Appendix \ref{appendix:proofofthm4}.
\subsection{Weak recovery via self-avoiding walks.} Let 
\[
\hat A_{i_1, i_2}^{(k)} = \frac{2n}{a_k - b_k} \left ( A_{i_1, i_2}^{(k)} - \frac{a_k+b_k}{2n} \right )
\]
for each $k = 1, \dots, m$, where $i_1\neq i_2 \in [n]$, and similarly define for all $j \in [p]$
\[
\hat B^j_{i_1, i_2} = \frac{n}{\mu} B_{i_1, j} B_{i_2, j} = \frac{n}{\mu} \left ( \sqrt{\frac{\mu}{n}}\sigma_{i_1} u_j + Z_{i_1,j} \right ) \left ( \sqrt{\frac{\mu}{n}}\sigma_{i_2} u_j + Z_{i_2,j} \right ).
\]
The expectation and variance of the above terms are
\begin{align*}
    \E[\hat A_{i_1, i_2}^{(k)} | \sigma] = \sigma_{i_1} \sigma_{i_2}, \quad & \quad \Var \left ( \hat A_{i_1, i_2}^{(k)} \right ) = \frac{n}{\lambda_k^2}, \\
    \E[\hat B_{i_1, i_2}^j | \sigma] = \sigma_{i_1} \sigma_{i_2}, \quad & \quad \Var \left ( \hat B_{i_1, i_2}^j \right ) = \frac{np}{\mu^2/\gamma}(1+o(1)).
\end{align*}

We associate the weight $A_{i_1, i_2}^{(k)}$ to the type $\bm{A}_{k}$ wedge between variable nodes $i_1$ and $i_2$, and the weight $\hat B_{i_1, i_2}^j$ to the type $\bm{B}$ wedge between variable nodes $i_1$ and $i_2$ including the $j^{th}$ $\bm{B}$-Type factor node. We call a path $\alpha$ starting at the variable node $i_1$ and ending at the variable node $i_2$ to be \emph{self-avoiding} if it visits no factor (type $\bm{B}$) node twice and does not have two types $\bm{A}$ wedges between the same pair of variable nodes. The total weight corresponding to $\alpha$ is defined as $p_\alpha := \prod_{e \in \alpha} w(e)$, where $e$ are wedges in the path $\alpha$. 

Consider a self-avoiding walk $\alpha$ with $k_i$ type $\bm{A}_{i}$ wedges for $1 \le i \le m$ and $\ell$ type $\bm{B}$ wedges connecting the variable nodes $i_1$ and $i_2$. Then, it can be shown that
\begin{equation}
\E[p_\alpha|\bm{\sigma}] = \sigma_{i_1}\sigma_{i_2},\quad \Var(p_\alpha) = \prod_{j=1}^m \left ( \frac{n}{\lambda_j^2} \right)^{k_j} \left ( \frac{np}{\mu^2/\gamma} \right)^\ell (1 + o(1)). 
\end{equation}

Our estimator of the matrix $\Sigma = \bm{\sigma}\bm{\sigma}^\top $ is then given by
\begin{equation}
\label{eq:widehat_sigma}
\widehat\Sigma_{i_1 i_2} := \frac{1}{|\mathcal{W}(i_1, i_2, k_1, \dots, k_m, \ell)|} \sum_{\alpha \in \mathcal{W}(i_1, i_2, k_1, \dots, k_m, \ell)} p_{\alpha},
\end{equation}
where $\mathcal{W}(i_1, i_2, k_1, \dots, k_m, \ell)$ is the set of all self-avoiding walks between factor nodes $i_1$ and $i_2$ on the factor graph that have $k_j$ type $\bm{A}_{j}$ edges for $1 \leq j \leq m$, and $\ell$ type $\bm{B}$ wedges. Here, $|\mathcal{W}(i_1, i_2, k_1, \dots, k_m, l)|$ denotes the cardinality of this set. The matrix $\widehat{\Sigma}$ satisfies the following reverse Cauchy Schwartz type inequality.
\begin{lem}
\label{lem:psd}
For any $\varepsilon>0$, if $\sum_{k=1}^{m}\lambda^2_k+\mu^2/\gamma \ge (1+\varepsilon)$, then there exists a constant $\delta=\delta(\lambda_1,\cdots,\lambda_m,\mu,\gamma,\varepsilon)>0$, such that,
\[
\mathbb{E}_{\bm \lambda,\mu}\left[\langle\widehat{\Sigma},\bm\sigma\bm\sigma^\top\rangle\right] \ge \delta \mathbb{E}_{\bm \lambda,\mu}\left[\|\widehat{\Sigma}\|^2_F\right]^{1/2}n,
\]
where $\bm\sigma$ is the vector of true partition assignments.
\end{lem}
Now, let us consider the $n\times n$ matrix $\widehat{\Psi}$ which is the solution to the following convex program.
\begin{gather}
\min_{\Psi}\left\|\Psi\right\|^2_{F}\nonumber\\
\mbox{s.t. } \mbox{diag}(\Psi)=1, \quad \nonumber \\ 
\quad \ \langle\widehat{\Sigma},\Psi\rangle \ge \delta n\,\|\widehat{\Sigma}\|_F,\ \mbox{and}\ \Psi\succeq 0.\label{eq:opt_problem}
\end{gather}
We can obtain $\widehat{\Psi}$ by solving \eqref{eq:opt_problem} using $\delta$'s constructed in Lemma \ref{lem:psd} and $\widehat{\Sigma}$ defined in \eqref{eq:widehat_sigma}. The estimator for $\bm \sigma$, denoted by $\widehat{\bm \sigma}$, is given by co-ordinate wise signs of $\bm Z$, a $n$ dimensional Gaussian vector with mean $\bm 0$ and variance $\widehat{\Psi}$. Then by Lemma \ref{lem:psd}, Markov Inequality, and Lemma 3.5 of \cite{8104074} we can deduce that if $k=O(\log n)$ then with high probability $\widehat{\bm \sigma}$ weakly recovers $\bm \sigma$.

To analyze the time complexity of the algorithm,  observe that
\[
|\mathcal{W}(i_1, i_2, k_1, \dots, k_m, \ell)|=\sum_{i=i_0\neq f_{i_0i_1}\neq i_1\neq f_{i_1i_2} \neq i_2\cdots\neq i_{2k}=j}\widetilde{G}_{i_0f_{i_0i_1}}\widetilde{G}_{f_{i_0i_1}i_1}\cdots\widetilde{G}_{i_{2k-1}f_{i_{2k-1}i_{2k}}i_{2k}},
\]
where $\widetilde{G}$ is the adjacency matrix of the entire factor graph $G_F$, $i_0,\cdots,i_{2k}$'s are the variable nodes and $f_{i_0i_1},\cdots,f_{i_
{2k-1}i_{2k}}$'s are the factor nodes. Thus, computing $|\mathcal{W}(i_1, i_2, k_1, \dots, k_m, \ell)|$ is equivalent to computing the matrix power $\widetilde{G}^{2k}$ plus checking in each step to ensure the resulting path is self-avoiding. This is polynomial time in the number of total nodes, which in turn is of order $n^{k+\ell}$ in the proportional asymptotic regime. One can similarly compute $\sum_{\alpha \in \mathcal{W}(i_1, i_2, k_1, \dots, k_m, l)} p_{\alpha}$ by considering appropriately weighted adjacency graphs. Hence, for $k$ and $\ell$ fixed, the matrix $\widehat\Sigma$ can be computed in polynomial time. Further, solving the convex program to find $\widehat{\Psi}$ is also a polynomial time exercise, and hence we can compute $\widehat{\bm \sigma}$ in $O(n^{\log{n}})$ time. As remarked in \cite{8104074}, one can probably improve this to  a polynomial time algorithm by considering non-backtracking walks and color coding, but this is beyond the scope of this paper.
\section{Belief Propagation and Approximate MAP Estimation}
\label{Belief_Propagation}
Let us observe that the quasi-polynomial time algorithm described in the previous section is hard to implement in practice. A commonly used approach in these situations is to consider the \emph{maximum a posteriori} (MAP) estimate of $\bm \sigma$ given the networks $\bm{G}_{1},\cdots, \bm{G}_{m}$ and $\bm B$. However, this requires marginalization over $\bm \sigma \in \{\pm 1\}^n$ and $\bm u \in \mathbb R^p$, which is intractable in polynomial time. Hence, we require some type of approximate method like variational inference via mean field or Bethe approximation. For approximate tree-like graphs, a local algorithm to estimate such planted partition is the Belief Propagation Algorithm (see, \cite{clique}). 

\revsag{In the extremely sparse regime, each of the component networks of the multilayer SBM is approximately tree-like. Hence we design a Belief Propagation algorithm based on the factor graph $G_F$ defined in Section \ref{sec:test} to recover the planted partition. The major idea behind the algorithm is to iteratively compute \emph{wedge messages} $\{\eta^{t}_{i \rightarrow j; \ell}: i,j \in [n], \ell \in [m]\}$ between the vertices of the factor graph $G_F$ along the wedges corresponding to the different layers of the network, \emph{variable node messages} $\{\eta^{t}_{i}: i \in [n]\}$ and \emph{factor node messages} $\{(m^{t}_{q},\tau^t_q): q \in [p]\}$. It is worthwhile to mention that the wedge messages can be decomposed into \emph{edge messages} for the component edges using the techniques specified in \cite{duranthon2024optimal}. But for the sake of simplicity, we avoid doing that in this paper.}

\revsag{
We begin by computing messages $\nu^{t}_{i \rightarrow j;\ell}(\cdot)$ which are the marginal distributions of the variable $\sigma_i$ under $\mathbb P_{i \rightarrow j;\ell}(\bm \sigma,\bm u|\bm G_1,\ldots,\bm G_m,\bm B / \sqrt{p})$ when the posterior distribution is computed after removing all the connections to variable node $j$ in the $\ell$-th network from the factor graph. Similarly, we compute the messages $\nu^{t}_{i \rightarrow q}(\cdot)$ which are the marginal distributions of $\sigma_i$ under $\mathbb P_{i \rightarrow q}(\bm \sigma,\bm u|\bm G_1,\ldots,\bm G_m,\bm B / \sqrt{p})$ when the vertex $u_q$ is removed from the factor graph with all associated connections. Finally, the messages  $\nu^{t}_{q \rightarrow i}(\cdot)$ which are the marginal distributions of $u_q$ under $\mathbb P_{q \rightarrow i}(\bm \sigma,\bm u|\bm G_1,\ldots,\bm G_m,\bm B / \sqrt{p})$ when the $i^{th}$ variable node is removed from the factor graph with all associated connections. Let us observe that both $\nu^{t}_{i \rightarrow j;\ell}(\cdot)$ and $\nu^{t}_{i \rightarrow q}(\cdot)$ for $i,j 
\in [n]$, $\ell \in [m]$ and $q \in [p]$ are discrete distributions supported on $\{\pm 1\}$. Hence to track such distributions it is enough to track the log-odds ratio. Therefore we focus on tracking
\begin{equation}
\label{eq:f}
\eta^t_{i \rightarrow j;\ell}:=\frac{1}{2}\log\frac{\nu^{t}_{i \rightarrow j;\ell}(+1)}{\nu^{t}_{i \rightarrow j;\ell}(-1)},
\end{equation}
and
\begin{equation}
\label{eq:rho}
\eta^t_{i \rightarrow q}:=\frac{1}{2}\log\frac{\nu^{t}_{i \rightarrow q}(+1)}{\nu^{t}_{i \rightarrow q}(-1)}.
\end{equation}
Further, for the messages $\nu^t_{q \rightarrow i}$ we use the \emph{Gaussian ansatz},  
\[
\nu^t_{q \rightarrow i}:=\mathsf{N}\left(\frac{m^t_{q \rightarrow i}}{\sqrt{p}},\frac{\tau^t_{q \rightarrow i}}{p}\right).
\] 
Therefore, for these messages, it is enough to track the one-dimensional parameters $(m^t_{q \rightarrow i},\tau^t_{q \rightarrow i})$.}

\revsag{The update equations of the parameters $\eta^t_{i \rightarrow q}, m^t_{q \rightarrow i}$ and $\tau^t_{q \rightarrow i}$ are further simplified by linearizing their expressions around certain `zero information' fixed points given by $\eta^t_i, m^t_q$ and $\tau^t_q$ for $i \in [n]$ and $q \in [p]$. Therefore it is enough to track the evolution of the parameters $\eta^t_{i \rightarrow j;\ell}, \eta^t_{i}, m^t_{q}$ and $\tau^t_{q}$. These parameters are iteratively updated as follows.
\begin{align}
\label{eq:bp_updates}
    \eta_i^{t+1} =
    &\sqrt{\frac{\mu}{p\gamma}}\sum_{r = 1}^{p}B_{ri}m_r^{t} -\frac{\mu}{\gamma}\left(\sum_{r=1}^{p}\frac{B_{ri}^2\tau_r^t}{p}\right)\tanh\left(\eta_i^{t-1}\right) + \sum_{r=1}^{m}\left\{\sum_{k\in \partial_{r}i}f(\eta^t_{k \rightarrow i; r};\rho_r)-\sum\limits_{k=1}^{n}f(\eta^t_k;\rho_{n,r})\right\},\\
    \eta_{i\rightarrow j,\ell}^{t+1} = 
    & \sqrt{\frac{\mu}{p\gamma}}\sum_{r = 1}^{p}B_{ri}m_r^{t} -\frac{\mu}{\gamma}\left(\sum_{r=1}^{p}\frac{B_{ri}^2\tau_r^t}{p}\right)\tanh\left(\eta_i^{t-1}\right) + \sum_{\substack{r=1\\ r\neq \ell}}^{m}\left\{\sum_{k\in \partial_{r}i}f(\eta^t_{k \rightarrow i; r};\rho_r)-\sum\limits_{k=1}^{n}f(\eta^t_k;\rho_{n,r})\right\}\\
    & + \sum_{k\in\partial_{\ell}i\setminus\{j\}}f(\eta^t_{k \rightarrow i; \ell};\rho_\ell) - \sum\limits_{k=1}^{n}f(\eta^t_k;\rho_{n,\ell}),\\
    \label{eq:bp_updates_end}
    \tau_q^{t+1} = 
    &\left(1 + \mu - \frac{\mu}{p\gamma}\sum_{j=1}^{n}B_{qj}^2\sech^2(\eta^t_j)\right)^{-1}\\
    m_q^{t+1} &= \tau_q^{t+1}\sqrt{\dfrac{\mu}{p\gamma}}\sum_{j=1}^{n}B_{qj}\tanh(\eta_j^{t}) - \frac{\tau_q^{t+1}\mu}{\gamma}\left(\sum_{j=1}^{n}\frac{B_{qj}^2}{p}\sech^2(\eta^t_j)\right)m_q^{t-1}.
\end{align}}

Here, 
\[
f(z;\rho)=\frac{1}{2}\log\frac{\cosh(z+\rho)}{\cosh(z-\rho)};
\]
\[
\rho_\ell=\tanh^{-1}(\lambda_\ell/\sqrt{d_\ell}) \quad \mbox{and} \quad \rho_{n;\ell}=\tanh^{-1}(\lambda_\ell\sqrt{d_\ell}/(n-d_\ell)).
\]
Further, for all $i \in [n]$, $\partial_r i$ is the neighborhood of $i$ in the $r$-th network. We consider the following estimate of $\bm \sigma$ denoted by $\widehat{\bm \sigma}$ which closely approximates the MAP estimate. We run the algorithm for $T_{\max}$ many iterations and estimate the $i$-th node label by,
\begin{align}
\label{eq:wh_sig}
\widehat{\sigma}_i=\mbox{sgn}(\eta^{T_{\max}}_i).
\end{align}
\revsag{The intuitive explanation behind using this estimate is as follows: The messages $\eta^{t}_{i}$ approximates the posterior log-odds of $\sigma_i$ given the networks and $\bm B$. In other words,
\[
\eta^{t}_{i} \approx \log\frac{\mathbb P(\sigma_i = 1|\bm G_1,\ldots,\bm G_m,\bm B/\sqrt{p})}{\mathbb P(\sigma_i = -1|\bm G_1,\ldots,\bm G_m,\bm B/\sqrt{p})}.
\]
Therefore, if $\mathbb P(\sigma_i = 1|\bm G_1,\ldots,\bm G_m,\bm B/\sqrt{p}) > \mathbb P(\sigma_i = -1|\bm G_1,\ldots,\bm G_m,\bm B/\sqrt{p})$, then $\mbox{sgn}(\eta^{t}_i)>0$ and vice-versa. Hence, $\mbox{sgn}(\eta^t_i)$ is a reasonable proxy of $\sigma_i$ after $t$ iterations if $\mathbb P(\sigma_i = 1|\bm G_1,\ldots,\bm G_m,\bm B/\sqrt{p})$ is bounded away from $1/2$. If $\mathbb P(\sigma_i = 1|\bm G_1,\ldots,\bm G_m,\bm B/\sqrt{p}) \approx 1/2$, i.e., the data $\bm G_1,\ldots,\bm G_m,\bm B$ are non-informative about $\sigma_i$, then $\bm \eta^t_i \approx 0$.}
\begin{rem}
Using the techniques used in \cite{ContBlockMod}, this BP algorithm can be further approximated to obtain an approximate linear message passing algorithm as in (13)-(15) of \cite{ContBlockMod}. It is interesting in itself to study the spectrum of the version of the non-backtracking walk on the multilayer network obtained by setting $\mu=0$ in these approximate update equations. However, that is beyond the scope of our current work and we relegate it to future work.
\end{rem}
\section{Numerical Simulations}
\label{experiments}
\revsag{Recall the hypothesis outlined in \eqref{eq:hyp_1}. Towards testing the same, we need to estimate the parameters $(\lambda_1,\ldots,\lambda_m,\mu)$ from the data. However, following the results of \citet{mnsf}, if $(a_k-b_k)^2 \le 2(a_k+b_k)$, it becomes infeasible to consistently test $a_k=b_k$ versus $a_k \neq b_k$ for each individual $k$. Consequently, no method can consistently estimate these parameters. In such situations, any data-dependent estimators of the parameters $a_k,b_k$ would not be good approximations of the true parameters. Therefore, using them to construct the BP updates might lead to low overlap and misleading inference, and it becomes important to rely on prior approximate knowledge of the parameters. But if $(a_k-b_k)^2 > 2(a_k+b_k)$ for all $k \in [m]$, then it is possible to estimate the parameters $a_k$ and $b_k$. One possible technique to estimate these parameters is to consider the sample eigenvalues of the adjacency matrices of the observed networks. The largest eigenvalue of the adjacency matrix of $\bm G_k$ serves as a proxy of $(a_k+b_k)/2$ and the second largest eigenvalue serves as a proxy of $(a_k-b_k)/2$. Hence, from these eigenvalues, it is possible to construct reasonable approximations of the model parameters $a_k,b_k$ for all $k \in [m]$. Furthermore, the square of the largest singular value of $\bm B/\sqrt{p}$ can be used to estimate the parameter $\mu$. Another way to estimate the parameters $a_k,b_k$ is outlined in \citet{mnsf}. These estimated parameters can be used in the algorithm \eqref{eq:bp_updates} and the resulting estimates $\bm{\eta}^{T_{\max}}$ can be used to test the composite null hypothesis \eqref{eq:test}.}
\revsag{However, in this paper, to demonstrate the properties of our Belief Propagation algorithm, we consider following the simpler testing problem:
\begin{equation}
\label{eq:simp_test}
\bm{H}_0: (\lambda_1,\cdots,\lambda_m,\mu)=(0,\cdots,0,0) \quad \quad \mbox{vs} \quad \quad \bm{H}_{1}: (\lambda_1,\cdots,\lambda_m,\mu)= (\lambda^*_{1},\cdots,\lambda^*_{m},\mu^*).
\end{equation}
The same setting has also been considered in \citet{ContBlockMod} and \citet{lu2020contextual}. The above theory on contiguity continues to hold forthe testing problem in \eqref{eq:simp_test}. In particular, the information-theoretic threshold for detectability is still attained at $\sum_{i=1}^{m}(\lambda_i^*)^2+(\mu^*)^2/\gamma=1$.}

Under the null hypothesis $\bm H_0$, \revsag{the data does not provide informative evidence about the planted partition $\bm \sigma$, leading to the norm of $\bm\eta^{T_{\max}}$ being close to zero, as discussed in the previous section.} Consequently, as the combined signal strength $\sum_{i=1}^{m}(\lambda_i^*)^2+(\mu^*)^2/\gamma$, increases we expect an increase in $\|\bm \eta^{T_{\max}}\|_2$. Hence, the test rejects the null when 
\begin{align}
\label{eq:test}
\|\bm \eta^{T_{\max}}\|_2>\|\bm \eta^{0}\|_2,
\end{align}
where $\bm \eta^0$ is an initializer generated independently from a $\mathsf{N}(0,0.01)$.

\revsag{We study the empirical power of the simple versus simple hypothesis test outlined in \eqref{eq:simp_test}. In the same settings, we also analyze the empirical overlap of the BP estimate $\widehat{\bm \sigma}$ with the true $\bm \sigma$. These estimates are in some sense, an oracle estimate since they are constructed using the knowledge of $a_k,b_k$ and $\mu$. Again in relatively higher signal-to-noise ratio regimes, we can use the plug-in estimates of these parameters constructed from the data to provide data-adaptive estimates.}
\begin{figure}[h]
  \centering
  \begin{subfigure}[b]{0.45\textwidth}
    \centering
    \includegraphics[width=\textwidth]{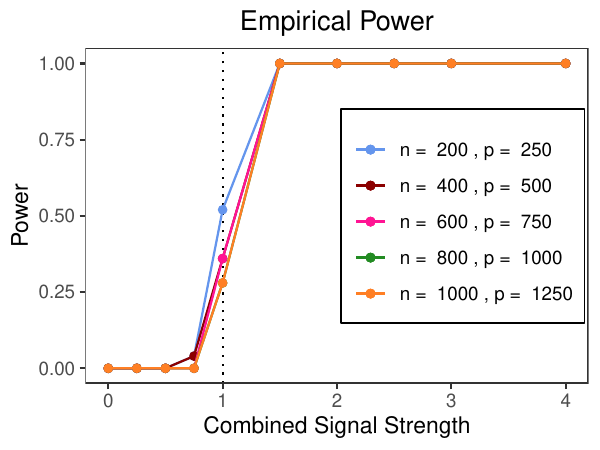}
    \caption{Average power as a function of combined signal strength}
    \label{fig:plot_n_power}
  \end{subfigure}
  \begin{subfigure}[b]{0.45\textwidth}
    \centering
    \includegraphics[width=\textwidth]{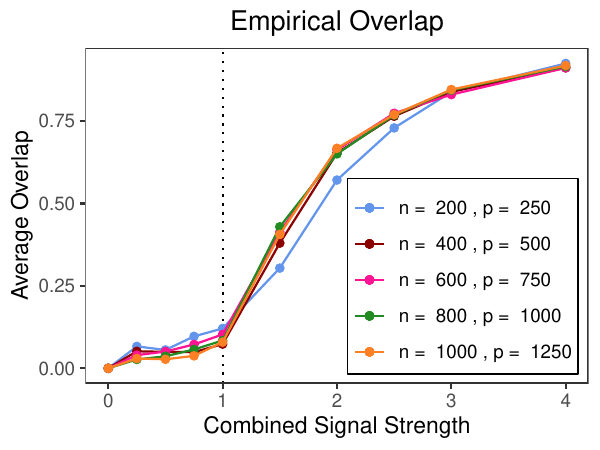}
    \caption{Average overlap as a function of combined signal strength}
    \label{fig:plot_n_overlap}
  \end{subfigure}
  \caption{Variation of power and empirical overlap with $n$ and $p$}
\end{figure}
\subsection{Dependence on the sample size}
\label{exp_var_samp_size}
In this experiment, we study the variation in the power of the test given by \eqref{eq:test} and the empirical overlap of $\widehat{\bm \sigma}$ with $\bm \sigma$ as a function of the combined signal strength $$t=\sum_{i=1}^{m}(\lambda_i^*)^2+(\mu^*)^2/\gamma,$$ for various values of $n$ and $p$. We fix the number of networks $m=3$ and the average degrees are taken to be $d_1=3,~ d_2=2,~ d_3=2$. We vary the combined signal strength on the set $t \in \{0, 0.25, 0.5, 0.75, 1, 1.5, 2, 2.5, 3, 4\}$. Moreover, we set individual signals as 
\[
(\lambda^*_i)^2=\frac{t}{5} \quad \mbox{for $i=1,2,3$, ~and,} \quad (\mu^*)^2 = \frac{2\gamma~t}{5}.
\]
The Belief Propagation updates are iterated for $50$ iterations to compute $\bm \eta_{50}$, which serves as a test statistic. The probability of rejection for different values of $t$ is empirically estimated by averaging over 25 independent replications of the experiment. 

In Figure \ref{fig:plot_n_power}, we illustrate the empirical power as a function of $t$ for various values of $n$ and $p$. As indicated by our theoretical results, the test remains powerless below $1$, achieves non-trivial power at $1$, and achieves power equal to $1$ when the combined signal strength exceeds 1. An interesting finding from this plot is that as $n$ increases, the power at $1$ is closer to zero. This indicates that a sharp phase transition may be possible at $1$ when $n \rightarrow \infty$. Furthermore, we plot the estimated average overlap between $\widehat{\bm \sigma}$ and $\bm \sigma$ as a function of $t$ in Figure \ref{fig:plot_n_overlap}. Here, we observe that the average overlap remains close to zero below $1$ and shows a sharp increase when the combined signal strength exceeds $1$. This corroborates our theoretical claims.

\begin{figure}[!h]
  \centering
  \begin{subfigure}[b]{0.45\textwidth}
    \centering
    \includegraphics[width=\textwidth]{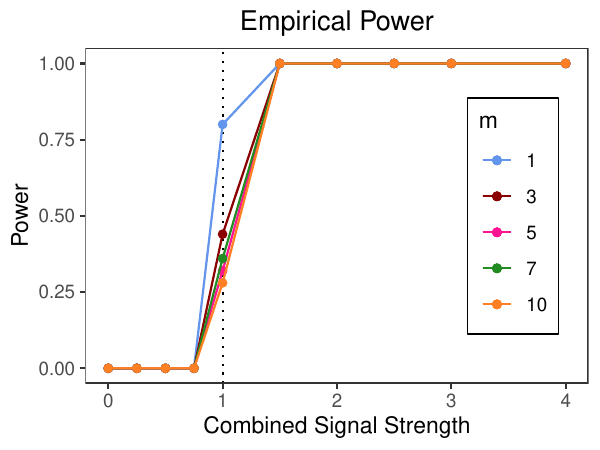}
    \caption{Average power as a function of combined signal strength}
    \label{fig:plot_m_power}
  \end{subfigure}
  \begin{subfigure}[b]{0.45\textwidth}
    \centering
    \includegraphics[width=\textwidth]{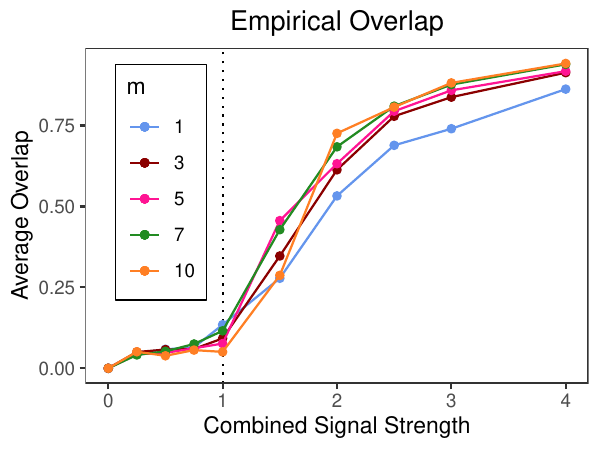}
    \caption{Average overlap as a function of combined signal strength}
    \label{fig:plot_m_overlap}
  \end{subfigure}
   \caption{Variation of power and empirical overlap with the number of networks $m$}
\end{figure}
\subsection{Dependence on the number of networks}
\label{exp_var_num_networks}
In this subsection, we explore the dependence of the power of the test given by \eqref{eq:test} and the average overlap between the BP estimate of $\bm \sigma$ with the ground truth as a function of the number of component networks $m$. Here, we fix $n=400$ and $p=500$. The signal strengths of the component networks are chosen as follows:
\[
(\lambda^*_i)^2=\frac{t}{(m+2)} \quad \mbox{for $i=1,\ldots,m$, ~and,} \quad (\mu^*)^2 = \frac{2\gamma~t}{(m+2)},
\]
where $t$ is the combined signal strength. We vary $t \in \{0, 0.25, 0.5, 0.75, 1, 1.5, 2, 2.5, 3, 4\}$ and study the power and empirical overlap as a function of $t$ for the number of networks $m$ varying in $\{1, 3, 5, 7, 10\}$. The average degrees of the component networks namely $d_1,\ldots,d_m$ are sampled uniformly from the set $\{2,3,4,5\}$ in the beginning and fixed for the rest of the experiment. The empirical estimates of the power and the average overlap are computed by averaging over $25$ independent iterations of the experiment.

In Figures \ref{fig:plot_m_power} and \ref{fig:plot_m_overlap}, we illustrate how the power of the test given by \eqref{eq:test} varies for various values of $m$. We observe that for all values of $m$, both the power and the average overlap are close to zero when $t$ is below $1$. Both of them sharply rise when $t$ exceeds $1$. This behaviour aligns with our theoretical findings. 

Furthermore, Figure \ref{fig:plot_m_overlap} also illustrates a general upward trend in the empirical overlap with an increase in the number of networks. In the sparse regime considered in this paper, the stochastic block model generates networks that are locally tree-like. The Belief propagation mechanism is known to perform better in such networks in terms of extracting information about the latent parameter. In contrast, the factor graph corresponding to the covariates is a fully connected network. Hence, as the contribution of the networks increase in the total signal strength, we observe better performance in terms of empirical overlap. 

\begin{figure}[h]
  \centering
  \begin{subfigure}[b]{0.45\textwidth}
    \centering
    \includegraphics[width=\textwidth]{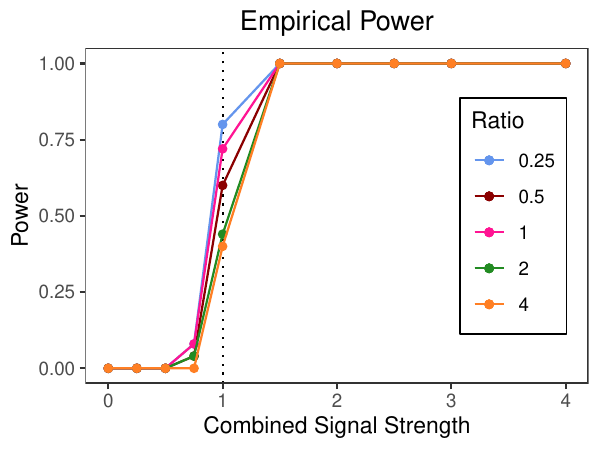}
    \caption{Average power as a function of combined signal strength}
    \label{fig:plot_r_power}
  \end{subfigure}
  \begin{subfigure}[b]{0.45\textwidth}
    \centering
    \includegraphics[width=\textwidth]{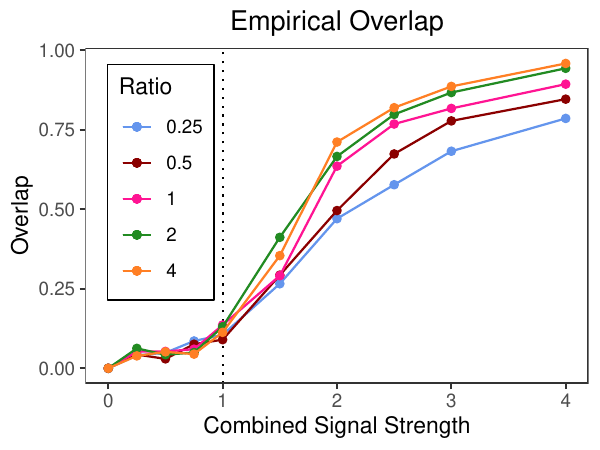}
    \caption{Average overlap as a function of combined signal strength}
    \label{fig:plot_r_overlap}
  \end{subfigure}
  \caption{Variation of power and empirical overlap with the ratio $\frac{\lambda_{1}^{2}+\ldots+\lambda_{m}^{2}}{\mu^{2} / \gamma}$}
\end{figure}

\subsection{Dependence on the ratio of the signal from covariates to the signal from the networks}
\label{exp_ratio}
In this experiment, we study the power of the test \eqref{eq:test} for the hypothesis \eqref{eq:simp_test} as a function of the ratio $r=\frac{(\lambda^*_1)^2+\ldots+(\lambda^*_m)^2}{(\mu^*)^2/\gamma}$. We fix $m=2$ and take the average degrees to be $d_1=3,~d_2=2$. Furthermore, we vary $r \in \{0.25,0.5,1,2,4\}$ and the combined signal strength $t \in \{0, 0.25, 0.5, 0.75, 1, 1.5, 2, 2.5, 3, 4\}$. The individual signals are taken as 
\[(\lambda^*_1)^2=(\lambda^*_2)^2= \frac{r~t}{2(r+1)}, \quad \mbox{and,} \quad (\mu^*)^2=\frac{\gamma~t}{r+1}.
\]
Fixing $n=400$ and $p=500$, we study the empirical power, and the average overlap. The experimental values are computed by averaging over 25 independent iterations of the experiment. 

The empirical power as a function of the combined signal strength is plotted in Figure \ref{fig:plot_r_power}. The behavior of the power is similar to the other two experiments described before. For the average overlap, we observe an increase in the overlap value with the increase in the contribution of the network components in the combined signal strength. Like before, this phenomenon is explained by the fact that the BP iterates are more accurate in estimating the parameters on sparser networks which are approximately cycle-free and not very efficient in achieving the same goal in a dense network. Since the part of the factor graph encoding the information about the covariates is represented by a fully connected network, the performance of the BP in extracting information from that part is not very good.

\subsection{Effect of data integration}
In this experiment, we demonstrate that in relatively sparse networks and low combined signal strength, the Belief Propagation-based estimator $\widehat{\bm \sigma}$ has better performance in terms of empirical overlap compared to existing methods in the literature that use multilayer networks or covariate information in isolation. In other words, when the combined signal strength is above the phase transition threshold mentioned in Theorem \ref{thm:weak_recovery_threshold} but the signals from the individual components, namely, the multilayer network and the covariate matrix $\bm B$ are small, Belief Propagation can utilize the power of data integration to provide an estimate with better overlap with the true $\bm \sigma$. 

To illustrate this, we consider the DC-MASE algorithm from \citet{agterberg2022joint} which constructs a spectral estimate of $\bm \sigma$ from the multilayer networks. Furthermore, we also consider the Approximate Message Passing (AMP) algorithm from \citet{montanari2021} which estimates $\bm \sigma$ using the matrix $\bm B$. 

We fix $n=400, p=500$ and $m=2$. The average degrees of the component networks are fixed at $d_1=3, d_2=2$ and the combined signal strength $t$ is varied across 10 equidistant points in $[0,4.5]$. Further, we vary the ratio $r=\frac{(\lambda^*_1)^2+\ldots+(\lambda^*_m)^2}{(\mu^*)^2/\gamma}$ in $\{0.5,1,2\}$. The individual signals are chosen to be
\[
(\lambda^*_1)^2 = \frac{3r~t}{4(1+r)}, \quad (\lambda^*_2)^2 = \frac{r~t}{4(1+r)}, \quad \mbox{and,} \quad (\mu^*)^2=\frac{\gamma~t}{(1+r)}.
\]
For each value of $r$ and $t$, we estimate $\bm \sigma$ using DC-MASE which only uses the networks $\bm G_1$ and $\bm G_2$, AMP using only $\bm B$ and Belief propagation using $\bm G_1, \bm G_2$ and $\bm B$. The average overlap of the estimated $\widehat{\bm \sigma}$ with the true $\bm \sigma$ is computed for each of the three estimators by averaging over 25 independent replications of the experiments. We visually compare the performance of these estimators in terms of average overlap in Figure \ref{fig:method_comparison}. We observe that the performance of DC-MASE, a state-of-the-art spectral algorithm (to the best of our knowledge) for detecting communities in multilayer networks, is quite subpar compared to the Belief Propagation estimate in sparse networks with low signals considered in this paper. It improves slightly as the contribution of the networks increases in the total signal strength. The Approximate Message Passing algorithm performs better than DC-MASE when the contribution of the $\bm B$ is more than the networks. However, neither method can compete with the Belief Propagation which provides an estimate with greater overlap with the true $\bm \sigma$ by integrating information across different components of the data. 

\begin{figure}[h]
  \centering
  \begin{subfigure}[b]{0.32\textwidth}
    \centering
    \includegraphics[width=\textwidth]{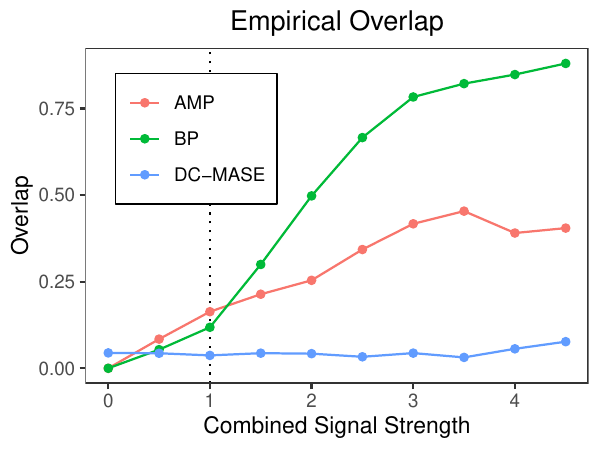}
    \caption{$r=0.5$}
    \label{fig:plot_mt_power}
  \end{subfigure}
  \begin{subfigure}[b]{0.32\textwidth}
    \centering
    \includegraphics[width=\textwidth]{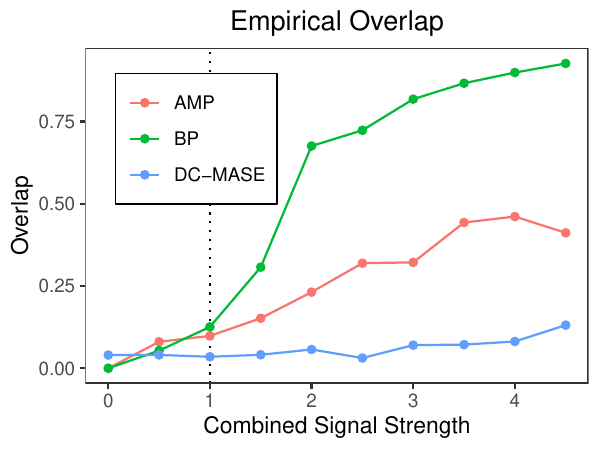}
    \caption{$r=1$}
  \end{subfigure}
  \begin{subfigure}[b]{0.32\textwidth}
    \centering
    \includegraphics[width=\textwidth]{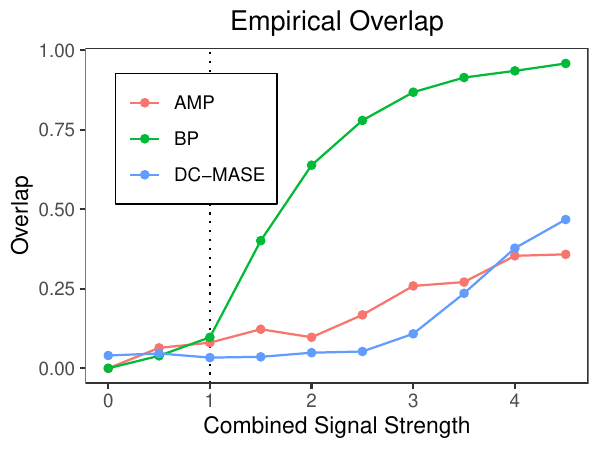}
    \caption{$r=2$}
  \end{subfigure}
  \caption{Empirical average overlap as a function of $(\lambda_1^*)^2+(\lambda_2^*)^2+(\mu^*)^2/\gamma$.}
  \label{fig:method_comparison}
\end{figure}

\section{Conclusion and Limitations}
In this paper, we have found the sharp information-theoretic threshold for testing the presence of a planted partition of the nodes in a multiplex network system where the average node degrees for each of the networks are greater than $1$ but non-diverging. We have further shown that the detection threshold coincides with the threshold for weakly reconstructing the planted partition and have provided a quasi-polynomial algorithm to achieve the weak recovery of the partition. Finally, we have given an approximate version of the Belief Propagation Algorithm to compute the approximate MAP estimate of $\bm \sigma$ and shown that this estimator is also quite good in weakly recovering the true $\bm \sigma$.

The techniques described in this work are limited to the balanced two community networks and are not easily extendable for multiple communities. Various interesting scenarios occur in a multi-community setup, for example, each layer of the network may be informative about a single community, and all the communities are identifiable only when all the layers are observed \cite{mercado18a}. The detection threshold in such problems is an interesting direction for future research.

It is also interesting to analyze the situation when the conditional independence assumption between the layers and the covariates is relaxed. However, extending our methods to
such a setup might not be straightforward and is beyond the scope of our current work.

\section{Acknowledgement}
The authors would like to thank Bhaswar B. Bhattacharya and Subhabrata Sen for their many helpful comments and suggestions.

\appendix

\section{Proof of Theorem \ref{thm:detection_threshold}}\label{appendix:proofofthm1}
\subsection{Proof of the lower bound}
In this section we show that $\mathbb{P}_{\bm{\lambda},\mu}$ is contiguous to $\mathbb{P}_{\bm{0},0}$ when $\frac{\mu^{2}}{\gamma} + \sum_{i=1}^{m}\lambda_{i}^{2}< 1$. Consider $\widetilde{\mathbb{P}}_{\bm{\lambda},\mu}$, the distribution of the contextual SBM with the parameters $\bm\lambda$ and $\mu$ conditioned on $\bm{\sigma}$ and $\bm{u}$. Given $\delta>0$ define,
\begin{align}\label{eq:def-set-S}
    \mathcal{S} = \left\{\bm{u}:\left\|\bm{u}\right\|_{2}\leq (1+\delta)\sqrt{p}\right\}.
\end{align}
Observe that the likelihood ratio between $\mathbb{P}_{\bm{\lambda},\mu}$ and $\mathbb{P}_{0,0}$ can be written as,
\begin{align*}
    L := \dfrac{\mathbb{P}_{\bm{\lambda},\mu}(\mathbb{A}_{m},\bm{B})}{\mathbb{P}_{\bm 0,0}(\mathbb{A}_{m},\bm{B})} = \dfrac{\mathbb{E}_{\bm{\sigma},\bm{u}}\left[\widetilde{\mathbb{P}}_{\bm{\lambda},\mu}\left(\mathbb{A}_{m},\bm{B}\middle|\bm{\sigma},\bm{u}\right)\right]}{\mathbb{P}_{\bm 0,0}(\mathbb{A}_{m},\bm{B})}
\end{align*}
where $\mathbb{A}_{m} = \{\bm{A}_{i}:1\leq i\leq m\}$. Now consider the truncated likelihood ratio
\begin{align*}
    L_{t}:= \dfrac{\mathbb{E}_{\bm{\sigma},\bm{u}}\left[\widetilde{\mathbb{P}}_{\bm{\lambda},\mu}\left(\mathbb{A}_{m},\bm{B}\middle|\bm{\sigma},\bm{u}\right)\bm{1}\{\bm{u}\in\mathcal{S}\}\right]}{\mathbb{P}_{\bm 0,0}(\mathbb{A}_{m},\bm{B})}.
\end{align*}
\begin{lem}\label{lemma:Lt-bdd}
If there exists $C>0$ such that, $\mathbb{E}_{\bm{H}_{0}}L_{t}^{2}\leq C$, then $\mathbb{P}_{\bm{\lambda},\mu}$ is contiguous to $\mathbb{P}_{\bm{0},0}$.
\end{lem}
\begin{proof}
Let $\{E_{n}:n\geq 1\}$ be any sequence of events such that $\mathbb{P}_{\bm{0},0}(E_{n})\rightarrow 0$, as $n\rightarrow\infty$. Observe that,
\begin{align}\label{eq:equation-1}
    \mathbb{P}_{\bm{\lambda},\mu}(E_{n}) = \mathbb{E}_{\bm{0},0}\left[L\bm{1}\{E_{n}\}\right] = \mathbb{E}_{\bm{0},0}\left[L_{t}\bm{1}\{E_{n}\}\right] + \mathbb{E}_{\bm{0},0}\left[(L-L_{t})\bm{1}\{E_{n}\}\right].
\end{align}
By definition of $L_t$, it then follows:
\begin{align}\label{eq:equation-2}
    \mathbb{E}_{\bm{0},0}\left[(L-L_{t})\bm{1}\{E_{n}\}\right]\leq \mathbb{E}_{\bm{0},0}\left[L-L_{t}\right]\leq \mathbb{E}_{\bm{\sigma},\bm{u}}\left[\bm{1}\{\bm{u}\not\in\mathcal{S}\}\right]\overset{n\rightarrow\infty}{\longrightarrow}0.
\end{align}
Using Cauchy-Schwarz Inequality along with the bound on $\mathbb{E}_{\bm{H}_{0}}L_{t}^{2}$ we have,
\begin{align}\label{eq:equation-3}
    \mathbb{E}_{\bm{0},0}\left[L_{t}\bm{1}\{E_{n}\}\right]\leq \sqrt{\mathbb{E}_{\bm{0},0}\left[L_{t}^2\right]\mathbb{P}_{\bm{0},0}(E_{n})}\overset{n\rightarrow\infty}{\longrightarrow}0.
\end{align}
Combining \eqref{eq:equation-1},\eqref{eq:equation-2} and \eqref{eq:equation-3} completes the proof.
\end{proof}
\begin{prop}\label{prop:Elt^2-bd}
Under $\bm{H}_{0}$, $\mathbb{E}L_{t}^{2}\leq C_{0}<\infty$ for some universal constant $C_{0}>0$.
\end{prop}
\begin{proof}
Observe that, by Fubini's Theorem,
\begin{align}\label{eq:ELt-expression}
    \mathbb{E}_{\bm{0},0}L_{t}^{2} = \mathbb{E}_{(\bm{\sigma},\bm{u}),(\bm{\tau},\bm{v})}\left[\mathbb{E}_{\bm{0},0}\left[\dfrac{\widetilde{\mathbb{P}}_{\bm{\lambda},\mu}\left(\mathbb{A}_{m},\bm{B}\middle|\bm{\sigma},\bm{u}\right)}{\mathbb{P}_{\bm 0,0}(\mathbb{A}_{m},\bm{B})}\dfrac{\widetilde{\mathbb{P}}_{\bm{\lambda},\mu}\left(\mathbb{A}_{m},\bm{B}\middle|\bm{\tau},\bm{v}\right)}{\mathbb{P}_{\bm 0,0}(\mathbb{A}_{m},\bm{B})}\bm{1}\left\{\bm{u},\bm{v}\in\mathcal{S}\right\}\right]\right]
\end{align}
where $\bm{\tau}$ and $\bm{v}$ are i.i.d copies of $\bm{\sigma}$ and $\bm{u}$ respectively. Recall that, by construction, given $\{\bm{\sigma},\bm{u}\}$, $\bm{A}_{i}$'s for $1\leq i\leq m$ are independent and $\bm A_{i}$'s are mutually independent and independent of $\bm{B}$. Then,
\begin{align*}
    \dfrac{\widetilde{\mathbb{P}}_{\bm{\lambda},\mu}\left(\mathbb{A}_{m},\bm{B}\middle|\bm{\sigma},\bm{u}\right)}{\mathbb{P}_{\bm 0,0}(\mathbb{A}_{m},\bm{B})} = \left\{\prod_{i=1}^{m}\dfrac{\widetilde{\mathbb{P}}_{\bm{\lambda},\mu}\left(\bm{A}_{i}\middle|\bm{\sigma}\right)}{\mathbb{P}_{\bm{0},0}(\bm{A}_{i})}\right\}\dfrac{\widetilde{\mathbb{P}}_{\bm{\lambda},\mu}\left(\bm{B}\middle|\bm{\sigma},\bm{u}\right)}{\mathbb{P}_{\bm{0},0}(\bm{B})}.
\end{align*}
Define,
\begin{align*}
    W_{ij}^{(k)}:=W_{ij}^{(k)}\left(\bm{A}_{k},\bm{\sigma}\right)=
    \begin{cases}
    \dfrac{2a_{k}}{a_{k} + b_{k}}\text{ if }\sigma_{i}=\sigma_{j},\ A_{ij}^{(k)} = 1\\
    \dfrac{2b_{k}}{a_{k} + b_{k}}\text{ if }\sigma_{i}\neq \sigma_{j},\ A_{ij}^{(k)} = 1\\
    \dfrac{n-a_{k}}{n-\frac{a_{k}+b_{k}}{2}}\text{ if }\sigma_{i}=\sigma_{j},\ A_{ij}^{(k)} = 0\\
    \dfrac{n-b_{k}}{n-\frac{a_{k}+b_{k}}{2}}\text{ if }\sigma_{i}\neq \sigma_{j},\ A_{ij}^{(k)} = 0
    \end{cases}
\end{align*}
and hence by definition,
\begin{align*}
    \dfrac{\widetilde{\mathbb{P}}_{\bm{\lambda},\mu}\left(\bm{A}_{k}\middle|\bm{\sigma}\right)}{\mathbb{P}_{\bm{0},0}(\bm{A}_{k})} = \prod_{i<j}W_{ij}^{(k)}.
\end{align*}
By definition of $\bm{B}$ from \eqref{eq:def-B},
\begin{align*}
    \dfrac{\widetilde{\mathbb{P}}_{\bm{\lambda},\mu}\left(\bm{B}\middle|\bm{\sigma},\bm{u}\right)}{\mathbb{P}_{\bm{0},0}(\bm{B})} = \exp\left(\sqrt{\dfrac{\mu}{n}}\sum_{i=1}^{n}\sigma_{i}\bm{R}_{i}^{T}\bm{u} - \dfrac{\mu}{2}\|\bm{u}\|_{2}^{2}\right).
\end{align*}
For all $1\leq k\leq m$ consider $V_{ij}^{(k)} = V_{ij}^{(k)}\left(\bm{A}_{k},\bm{\tau}\right), 1\leq i<j\leq n$ to be defined similarly to $W_{ij}, 1\leq i<j\leq n$. Let $V_{ij}^{(k)}$'s and $W_{ij}$'s be independent. By Lemma 5.4 of \cite{mnsf},
\begin{align}\label{eq:WV-expectation}
    \mathbb{E}\left[\prod_{i<j}W_{ij}^{(k)}V_{ij}^{(k)}\middle| \bm{\sigma},\bm{\tau}\right] = \left(1+o(1)\right)\exp\left( - \frac{\lambda_{k}^{2}}{2} - \frac{\lambda_{k}^{4}}{4}\right)\exp\left(\dfrac{\rho^{2}\lambda_{k}^{2}}{2}(d_{k} + n)\right)
\end{align}
where $\rho:=\rho(\bm{\sigma},\bm{\tau}) = \frac{1}{n}\langle\bm{\sigma},\bm{\tau}\rangle.$ Using MGF of multivariate Gaussian distribution,
\begin{align}\label{eq:prod-B-expectation}
    \mathbb{E}_{\bm{0},0}\left[\exp\left(\sqrt{\frac{\mu}{n}}\sum_{i=1}^{n}\bm{R}_{i}^{T}\left(\sigma_{i}\bm{u} + \tau_{i}\bm{v}\right) - \frac{\mu}{2}\left(\|\bm{u}\|_{2}^{2} + \|\bm{v}\|_{2}^{2}\right)\right)\middle|\bm{u},\bm{v};\bm{\sigma},\bm{\tau}\right] = \exp\left(\mu\langle\bm{u},\bm{v}\rangle\rho\right).
\end{align}
Observe that under $\bm{H}_{0}$, $\mathbb{A}_{m}$ and $\bm{B}$ are independent and hence
\begin{align*}
    \mathbb{E}_{\bm{0},0}
    &\left[\dfrac{\widetilde{\mathbb{P}}_{\bm{\lambda},\mu}\left(\mathbb{A}_{m},\bm{B}\middle|\bm{\sigma},\bm{u}\right)}{\mathbb{P}_{\bm 0,0}(\mathbb{A}_{m},\bm{B})}\dfrac{\widetilde{\mathbb{P}}_{\bm{\lambda},\mu}\left(\mathbb{A}_{m},\bm{B}\middle|\bm{\tau},\bm{v}\right)}{\mathbb{P}_{\bm 0,0}(\mathbb{A}_{m},\bm{B})}\bm{1}\left\{\bm{u},\bm{v}\in\mathcal{S}\right\}\middle|\bm{u},\bm{v};\bm{\sigma},\bm{\tau}\right]\\
    = & \bm{1}\left\{\bm{u},\bm{v}\in\mathcal{S}\right\}\mathbb{E}_{\bm{0},0}\left[\prod_{k=1}^{m}\dfrac{\widetilde{\mathbb{P}}_{\bm{\lambda},\mu}\left(\bm{A}_{k}\middle|\bm{\sigma}\right)}{\mathbb{P}_{\bm{0},0}(\bm{A}_{k})}\dfrac{\widetilde{\mathbb{P}}_{\bm{\lambda},\mu}\left(\bm{A}_{k}\middle|\bm{\tau}\right)}{\mathbb{P}_{\bm{0},0}(\bm{A}_{k})}\middle|\bm{u},\bm{v};\bm{\sigma},\bm{\tau}\right]\\
    &\mathbb{E}_{\bm{0},0}\left[\dfrac{\widetilde{\mathbb{P}}_{\bm{\lambda},\mu}\left(\bm{B}\middle|\bm{\sigma},\bm{u}\right)}{\mathbb{P}_{\bm{0},0}(\bm{B})}\dfrac{\widetilde{\mathbb{P}}_{\bm{\lambda},\mu}\left(\bm{B}\middle|\bm{\tau},\bm{v}\right)}{\mathbb{P}_{\bm{0},0}(\bm{B})}\middle|\bm{u},\bm{v};\bm{\sigma},\bm{\tau}\right]\\
    = & \bm{1}\left\{\bm{u},\bm{v}\in\mathcal{S}\right\}\prod_{k=1}^{m}\mathbb{E}_{\bm{0},0}\left[\prod_{i<j}W_{ij}^{(k)}V_{ij}^{(k)}\middle| \bm{\sigma},\bm{\tau}\right]\\
    &\mathbb{E}_{\bm{0},0}\left[\exp\left(\sqrt{\frac{\mu}{n}}\sum_{i=1}^{n}\bm{R}_{i}^{T}\left(\sigma_{i}\bm{u} + \tau_{i}\bm{v}\right) - \frac{\mu}{2}\left(\|\bm{u}\|_{2}^{2} + \|\bm{v}\|_{2}^{2}\right)\right)\middle|\bm{u},\bm{v};\bm{\sigma},\bm{\tau}\right]\\
    \leq & (1+o(1))\bm{1}\left\{\bm{u},\bm{v}\in\mathcal{S}\right\}\exp\left(-\sum_{k=1}^{m}\frac{\lambda_{k}^{2}}{2} - \sum_{k=1}^{m}\frac{\lambda_{k}^{4}}{4} + \sum_{k=1}^{m}\frac{\lambda_{k}^{2}d_{k}^{2}}{2}\right)\\
    &\exp\left(n\left(\frac{\rho^{2}\sum_{k=1}^{m}\lambda_{k}^{2}}{2}+\frac{\mu}{\gamma}\rho\frac{\langle\bm{u},\bm{v}\rangle}{p}\right)\right)
\end{align*}
where the last inequality follows by \eqref{eq:prod-B-expectation} and \eqref{eq:WV-expectation} and noting that $\rho\leq 1$. Plugging the above bound into \eqref{eq:ELt-expression}, we obtain,
\begin{align}\label{eq:ELt^2-bd}
    \mathbb{E}_{\bm{0},0}L_{t}^{2}\leq
    &(1+o(1))\exp\left(-\sum_{k=1}^{m}\frac{\lambda_{k}^{2}}{2} - \sum_{k=1}^{m}\frac{\lambda_{k}^{4}}{4} + \sum_{k=1}^{m}\frac{\lambda_{k}^{2}d_{k}^{2}}{2}\right)\nonumber\\
    &\mathbb{E}_{(\bm{\sigma},\bm{u}),(\bm{\tau},\bm{v})}\left[\exp\left(n\left(\frac{\rho^{2}\sum_{k=1}^{m}\lambda_{k}^{2}}{2}+\frac{\mu}{\gamma}\rho\frac{\langle\bm{u},\bm{v}\rangle}{p}\right)\right)\bm{1}\left\{\bm{u},\bm{v}\in\mathcal{S}\right\}\right]
\end{align}
Recall that $\frac{\mu^{2}}{\gamma} + \sum_{i=1}^{m}\lambda_{i}^{2}<1$, then we can choose $\delta>0$, defined in \eqref{eq:def-set-S}, small enough such that $\frac{\mu^{2}}{\gamma}(1+\delta)^2 + \sum_{i=1}^{m}\lambda_{i}^{2}<1$. Choosing such $\delta>0$ and following the proof of \cite[Theorem 1]{lu2020contextual} we have,
\begin{align}\label{eq:lusen-exp-bd}
    \mathbb{E}_{(\bm{\sigma},\bm{u}),(\bm{\tau},\bm{v})}\left[\exp\left(n\left(\frac{\rho^{2}\sum_{k=1}^{m}\lambda_{k}^{2}}{2}+\frac{\mu}{\gamma}\rho\frac{\langle\bm{u},\bm{v}\rangle}{p}\right)\right)\bm{1}\left\{\bm{u},\bm{v}\in\mathcal{S}\right\}\right]< C_{1}
\end{align}
for some constant $C_{1}>0$. The proof is completed by substituting the bound from \eqref{eq:lusen-exp-bd} into \eqref{eq:ELt^2-bd}. 
\end{proof}
Finally observe that applying Lemma \ref{lemma:Lt-bdd} along with Proposition \ref{prop:Elt^2-bd} shows that $\mathbb{P}_{\bm{\lambda},\mu}$ is contiguous to $\mathbb{P}_{\bm{0},0}$.\qed
\subsection{Proof of the upper bound}
In this section, for $\frac{\mu^{2}}{\gamma} + \sum_{i=1}^{m}\lambda_{i}^{2}>1$, we devise a consistent test using the cycle statistics $Y_{n,k_{1},\cdots,k_{m},\ell}$. This implies the asymptotic singularity of $\mathbb{P}_{\bm{\lambda},\mu}$ and $\mathbb{P}_{\bm{0},0}$. Recall that $k = \sum_{j=1}^{m}k_{j}+\ell$. By Theorem \ref{thm:cycle-stat-asymp-dist} for $k=O(\log^{1/4}n)$, under $\bm{H}_{0}$, 
\begin{align*}
    \dfrac{Y_{n,k_{1},\cdots,k_{m},\ell}}{\sigma_{k_{1},\cdots,k_{m},\ell}}\overset{d}{\rightarrow}\mathrm{N}(0,1),
\end{align*}
where,
\begin{align*}
    \sigma^2_{k_{1},\cdots,k_{m},\ell} = \dfrac{1}{2k\gamma^{\ell}}\;\dfrac{k!}{\ell!k_1!\,k_2!\cdots k_m!}\prod\limits_{j=1}^{m}d_{j}^{k_j}.
\end{align*}
and under $\bm{H}_{1}$,
\begin{align*}
    \dfrac{Y_{n,k_{1},\cdots,k_{m},\ell}}{\sigma_{k_{1},\cdots,k_{m},\ell}} - \widetilde{\mu}_{k_{1},\cdots,k_{m},\ell}\overset{d}{\rightarrow}\mathrm{N}(0,1)
\end{align*}
where,
\begin{align*}
    \widetilde{\mu}_{k_{1},\cdots,k_{m},\ell} = \sqrt{\dfrac{1}{2k}\dfrac{k!}{\ell!\prod_{j=1}^{m}k_{j}!}\prod_{j=1}^{m}(\lambda_{j}^{2})^{k_{j}}\left(\dfrac{\mu^{2}}{\gamma}\right)^{\ell}}
\end{align*}
Now choose,
\begin{align*}
     k_{j} = \left\lfloor \dfrac{\lambda_{j}^{2}k}{\mu^{2}/\gamma + \sum_{j=1}^{m}\lambda_{j}^{2}}\right\rfloor\text{ for all }1\leq j\leq m,\text{ and }\ell = k - \sum_{j=1}^{m}k_{j}.
\end{align*}
Then using Stirling's approximation,
\begin{align*}
    \widetilde{\mu}_{k_{1},\cdots,k_{m},\ell}^{2} 
    &\approx \dfrac{C_{1}}{2k}\sqrt{\frac{k}{\ell\prod_{j=1}^{m}k_{j}}}\left(\dfrac{k}{\ell}\right)^{\ell}\prod_{j=1}^{m}\left(\dfrac{k}{k_{j}}\right)^{k_{j}}\prod_{j=1}^{m}(\lambda_{j}^{2})^{k_{j}}\left(\dfrac{\mu^{2}}{\gamma}\right)^{\ell}\\
    &\geq C_{2}\sqrt{\dfrac{1}{k\ell\prod_{j=1}^{m}k_{j}}}\left(\dfrac{\mu^{2}}{\gamma} + \sum_{j=1}^{m}\lambda_{j}^{2}\right)^{k}\\
    &\geq C_{2}\dfrac{1}{k^{\frac{m+1}{2}}}\left(\dfrac{\mu^{2}}{\gamma} + \sum_{j=1}^{m}\lambda_{j}^{2}\right)^{k}\overset{k\rightarrow\infty}{\longrightarrow}\infty
\end{align*}
where $C_{1},C_{2}>0$ are universal constants. Thus for $k$ slowly growing in $n$ such that, $k=O(\log^{1/4}n)$ we get a sequence of consistent tests, completing the proof.\qed

\section{Proof of Theorem \ref{thm:cycle-stat-asymp-dist}}\label{appendix:proofofthm2}
Fix $r>0$. Consider $\{(k_{j_{1}},\cdots, k_{j_{m}},\ell_{j}):1\leq j\leq r\}$ such that $k_{j} = \sum_{p=1}^{m}k_{j_{p}} + \ell_{j}$ for all $1\leq j\leq r$ and $m_{1},\cdots, m_{r}\geq 1$. Without loss of generality suppose there exists $r_{1}\leq r$ such that $\ell_{p} = 0, 1\leq p\leq r_{1}$ and $\ell_{p}>0$ for $p>r_{1}$. Further suppose that $k_{1}<\cdots<k_{r_{1}}$ and $k_{r_{1}+1}<\cdots<k_{r}$. 
We shall show that,
\begin{align*}
    \mathbb{E}_{\bm 0,0}\left[\prod\limits_{j=1}^{r}Y^{m_j}_{n,k_{j_{1}},\cdots,k_{j_{m}},l_{j}}\right]
     \rightarrow \prod\limits_{j=1}^{r_1}\mathbb{E}\left[\nu^{m_j}_{(k_{j_1},\cdots,k_{j_m})}\right]\prod\limits_{j=r_1+1}^{r}\mathbb{E}\left[Z^{m_j}_{(k_1,\cdots,k_m,\ell)}\right],
\end{align*}
where $\nu^{m_j}_{(k_{j_1},\cdots,k_{j_m})}$ are Poisson$(\lambda_{k_{j_1},\cdots,k_{j_m}})$ and $Z_{(k_1,\cdots,k_m,\ell)}$ are $N(0,\sigma^2_{k_1,\cdots,k_m,\ell})$. Similarly, we shall also show
\begin{align*}
    \mathbb{E}_{\bm \lambda,\mu}\left[\prod\limits_{j=1}^{r}Y^{m_j}_{n,k_{j_{1}},\cdots,k_{j_{m}},l_{j}}\right]
     \rightarrow \prod\limits_{j=1}^{r_1}\mathbb{E}\left[\nu^{m_j}_{(k_{j_1},\cdots,k_{j_m})}\right]\prod\limits_{j=r_1+1}^{r}\mathbb{E}\left[\widetilde{Z}^{m_j}_{(k_1,\cdots,k_m,\ell)}\right],
\end{align*}
where $\widetilde{Z}_{(k_1,\cdots,k_m,\ell)}$ are $N(\mu_{k_1,\cdots,k_m,\ell},\sigma^2_{k_1,\cdots,k_m,\ell})$. To show that we need the following lemma.
\begin{lem}
\label{eq:asymp_ind}
As $n \rightarrow \infty$ we have the following,
\begin{enumerate}
    \item \hspace{-0.04in}$\left|\mathbb{E}_{\bm 0,0}\left[\prod\limits_{j=1}^{r}Y^{m_j}_{n,k_{j_{1}},\cdots,k_{j_{m}},l_{j}}\right]\hspace{-0.04in}-\mathbb{E}_{\bm 0,0}\left[\prod\limits_{j=1}^{r_1}Y^{m_j}_{n,k_{j_{1}},\cdots,k_{j_{m}},0}\right]\hspace{-0.01in}\mathbb{E}_{\bm 0,0}\left[\prod\limits_{j=r_1+1}^{r}Y^{m_j}_{n,k_{j_{1}},\cdots,k_{j_{m}},l_{j}}\right]\right|\rightarrow 0$,
    \item \hspace{-0.04in}$\left|\mathbb{E}_{\bm \lambda,\mu}\left[\prod\limits_{j=1}^{r}Y^{m_j}_{n,k_{j_{1}},\cdots,k_{j_{m}},l_{j}}\right]\hspace{-0.04in}-\mathbb{E}_{\bm \lambda,\mu}\left[\prod\limits_{j=1}^{r_1}Y^{m_j}_{n,k_{j_{1}},\cdots,k_{j_{m}},0}\right]\hspace{-0.01in}\mathbb{E}_{\bm \lambda,\mu}\left[\prod\limits_{j=r_1+1}^{r}Y^{m_j}_{n,k_{j_{1}},\cdots,k_{j_{m}},l_{j}}\right]\right|\rightarrow 0$.
\end{enumerate}
\end{lem}
The proof of Lemma \ref{eq:asymp_ind} is omitted here and given in section \ref{eq:proof_of_decouple_lemma}. By the decoupling shown in Lemma \ref{eq:asymp_ind}, under both $\bm{H}_{0}$ and $\bm{H}_{1}$, it is enough to analyze the terms with $\ell=0$ and $\ell>0$ separately. 
\subsection{Proof of Theorem \ref{thm:cycle-stat-asymp-dist}(1)} 
Let us first consider the $\ell=0$ case. \revsag{Observe that for a Poisson random variable $Y \sim \mbox{Poi}(\lambda)$, the factorial moments satisfy
\[
\mathbb E[Y(Y-1)\cdots(Y-M+1)] = \lambda^m.
\]
Since all the cycle statistics and the Poisson random variable have finite moment generating functions, to show the Poisson convergence in Theorem \ref{thm:cycle-stat-asymp-dist}~(1), it is enough to show,
\begin{align}\label{eq:H0-Poisson-Convg}
    \mathbb{E}_{\bm{0},0}\left[\left(Y_{n,k_{1},\cdots,k_{m},0}\right)_{[M]}\right] \overset{n\rightarrow\infty}{\longrightarrow}\left(\frac{1}{2k}\dfrac{k!}{\prod_{j=1}^{m}k_{j}!}\prod_{j=1}^{m}d_{j}^{k_{j}}\right)^{M}, \text{ for all }M\geq 1
\end{align}
where $(T)_{[M]} = T(T-1)\cdots(T-M+1)$ for any random variable $T$. This technique of using method of moments to prove Poisson convergence has been used before in \cite{mnsf,lu2020contextual}.}

Let $Y_{C}$ be the indicator that $C$ is a cycle in the factor graph having $k_{j}$ many type $\bm{A}_{j}$ wedges. Consider,
\begin{align}
    \mathcal{C}_{k_{1},\cdots,k_{m},0} = \{C:C \text{ is a cycle in the factor graph having $k_{j}$ many type $\bm{A}_{j}$ wedges}\}.
\end{align}
Then it is easy to see that $Y_{n,k_{1},\cdots,k_{m},0}$ can be rewritten as,
\begin{align}\label{eq:rewrite-Y-1}
    Y_{n,k_{1},\cdots,k_{m},0} = \sum_{C\in \mathcal{C}_{k_{1},\cdots,k_{m},0}}Y_{C}.
\end{align}
By a simple counting argument,
\begin{align}\label{eq:count-M-tuple-factor-graph}
    \left|\mathcal{C}_{k_{1},\cdots,k_{m},0}\right| = {n\choose k}\dfrac{(k-1)!}{2}\dfrac{k!}{\prod_{j=1}^{m}k_{j}!}.
\end{align}
For any $C\in\mathcal{C}_{k_{1},\cdots,k_{m},0}$ by definition, $\mathbb{E}_{\bm{0},0}Y_{C} = \prod_{j=1}^{m}(d_{j}/n)^{k_{j}}$ and hence,
\begin{align*}
    \mathbb{E}_{\bm{0},0}Y_{n,k_{1},\cdots,k_{m},0} = {n\choose k}\dfrac{(k-1)!}{2}\dfrac{k!}{\prod_{j=1}^{m}k_{j}!}\prod_{j=1}^{m}\left(\dfrac{d_{j}}{n}\right)^{k_{j}}.
\end{align*}
It is well known that,
\begin{align*}
    \dfrac{n(n-1)\cdots(n-k+1)}{n^{k}}\rightarrow 1\text{ whenever }k = o(\sqrt{n}).
\end{align*}
Taking $n\rightarrow\infty$ we get,
\begin{align}\label{eq:H0-Poisson-M1}
    \mathbb{E}_{\bm{0},0}Y_{n,k_{1},\cdots,k_{m},0} = (1+o(1))\frac{1}{2k}\dfrac{k!}{\prod_{j=1}^{m}k_{j}!}\prod_{j=1}^{m}d_{j}^{k_{j}}\text{ whenever }k=o(\sqrt{n}).
\end{align}
Then \eqref{eq:H0-Poisson-M1} shows \eqref{eq:H0-Poisson-Convg} for $M=1$. Define,
\begin{align*}
    \mu_{\text{Poi},\bm{H}_{0}} :=  \frac{1}{2k}\dfrac{k!}{\prod_{j=1}^{m}k_{j}!}\prod_{j=1}^{m}d_{j}^{k_{j}}.
\end{align*}
We now show that 
\begin{align*}
    \mathbb{E}_{\bm{0},0}\left(Y_{n,k_{1},\cdots,k_{m},0}\right)_{[M]}\rightarrow \mu_{\text{Poi},\bm{H}_{0}}^{M}.
\end{align*}
By definition, one can see that $\left(Y_{n,k_{1},\cdots,k_{m},0}\right)_{[M]}$ counts the number of $M$ tuples of cycles $(C_{1},C_{2},\cdots,C_{M})$, where all $C_{i}\in \mathcal{C}_{k_{1},\cdots,k_{m},0}$ are distinct for all $1\leq i\leq M$. First, suppose $(C_{1},C_{2},\cdots,C_{M})$ are all vertex disjoint and hence $Y_{C_{i}}$ are independent for $1\leq i\leq M$, then,
\begin{align}\label{eq:vertex-disjoint-M-prob}
    \mathbb{P}_{\bm{0},0}\left(C_{i}\in G_{F}, 1\leq i\leq M\right) = \mathbb{E}_{\bm{0},0}\prod_{i=1}^{M}Y_{C_{i}} = \prod_{i=1}^{M}\mathbb{E}Y_{C_{i}} = \left(\prod_{j=1}^{m}(d_{j}/n)^{k_{j}}\right)^{M} = n^{-kM}\left(\prod_{j=1}^{m}d_{j}^{k_{j}}\right)^{M}
\end{align}
where $G_{F}$ is the factor graph. Observe that number of ways to choose such vertex disjoint cycles is given by,
\begin{align}\label{eq:count-vertex-disjoint}
    {n\choose kM}\prod_{i=0}^{M-1}\left[{k(M-i)\choose k}\dfrac{(k-1)!}{2}\dfrac{k!}{\prod_{j=1}^{m}k_{j}!}\right] = \dfrac{n!}{(n-kM)!}\dfrac{1}{(2k)^{M}}\left(\dfrac{k!}{\prod_{j=1}^{m}k_{j}!}\right)^{M}
\end{align}
Then as long as $k = o(\sqrt{n})$, the contribution of vertex disjoint cycles in  $\mathbb{E}_{\bm{0},0}\left(Y_{n,k_{1},\cdots,k_{m},0}\right)_{[M]}$ as $n\rightarrow\infty$ is given by,
\begin{align*}
    \left(\frac{1}{2k}\dfrac{k!}{\prod_{j=1}^{m}k_{j}!}\prod_{j=1}^{m}d_{j}^{k_{j}}\right)^{M} = \mu_{\text{Poi},\bm{H}_{0}}^{M}
\end{align*}
Finally, it remains to show that the contribution of $M$ tuples of cycles $(C_{1},\cdots, C_{M})$ such that at least one pair is not vertex disjoint is asymptotically negligible. Consider $\mathcal{C}_{k_{1},\cdots,k_{m},0}^{(M,2)}$ to be the collection of such $M$ tuples. Then the contribution of $\mathcal{C}_{k_{1},\cdots,k_{m},0}^{(M,2)}$ in $\left(Y_{n,k_{1},\cdots,k_{m},0}\right)_{[M]}$ is given by,
\begin{align*}
    \left(Y_{n,k_{1},\cdots,k_{m},0}\right)_{[M]}^{(2)}:= \sum_{(C_{1},\cdots,C_{M})\in \mathcal{C}_{k_{1},\cdots,k_{m},0}^{(M,2)}}\prod_{i=1}^{M}Y_{C_{i}}
\end{align*}
Now, observe that $\left(Y_{n,k_{1},\cdots,k_{m},0}\right)_{[M]}^{(2)}$ is stochastically dominated by the same random variable for a Erd\H{o}s-R\'enyi random graph with connection probabilities $\max_{j=1}^{m}d_{j}$. Then, by \cite{bollobas_2001} Chapter $4$,
\begin{align}\label{eq:stochastic-dominance}
    \mathbb{E}_{\bm{0},0}\left(Y_{n,k_{1},\cdots,k_{m},0}\right)_{[M]}^{(2)} = o(1)\text{ whenever }d=O(\log^{1/4}n)
\end{align}
which completes the proof for \eqref{eq:H0-Poisson-Convg}.\\

Next, consider $0<\ell<k = \sum_{j=1}^{m}k_{j}+\ell$. Observe that under $\bm{H}_{0}$, $\bm{A}_{s}'$s are independent of $\bm{B}$. Hence,
\begin{align*}
    \mathbb{E}_{\bm{0},0}\left[Y_{n,k_{1},\cdots,k_{m},\ell}\right] = 0.
\end{align*}
Moving on to the variance calculations we observe that,
\begin{align}
    \mathbb{E}_{\bm{0},0}
    &\left[Y_{n,k_{1},\cdots,k_{m},\ell}^2\right]\nonumber\\ =&\frac{1}{n^{2\ell}}\sum_{\omega_{1:2}}\mathbb{E}_{\bm{0},0}\left[\left(\prod_{j=1}^{m}\prod_{e_{j}\in E_{\omega_{1},1_{j}}}A_{e_{j}}^{(j)}\prod_{e_{\ell}\in E_{\omega_{1},2}}B_{e_{\ell}}\right)\left(\prod_{j=1}^{m}\prod_{e_{j}\in E_{\omega_{2},1_{j}}}A_{e_{j}}^{(j)}\prod_{e_{\ell}\in E_{\omega_{2},2}}B_{e_{\ell}}\right)\right]\label{eq:exp-Y-sq}
\end{align}
where $\omega_{1:2}$ is a collection of cycles $\omega_{1},\omega_{2}$ having $k_{r}$ type $E_{1_{r}}$ wedges for $1\leq r\leq m$ and $\ell$ type $E_{2}$ wedges. We decompose \eqref{eq:exp-Y-sq} as follows,
\begin{align*}
    \mathbb{E}_{\bm{0},0}
    &\left[Y_{n,k_{1},\cdots,k_{m},\ell}^2\right]
    = T_{1} + T_{2}
\end{align*}
where
\begin{align*}
    T_{1} = \frac{1}{n^{2\ell}}\sum_{\omega}\mathbb{E}_{\bm{0},0}\left[\prod_{j=1}^{m}\prod_{e_{j}\in E_{\omega,1_{j}}}A_{e_{j}}^{(j)}\prod_{e_{\ell}\in E_{\omega,2}}B_{e_{\ell}}^2\right]
\end{align*}
and
\begin{align}\label{eq:def-T2}
    T_{2} = \frac{1}{n^{2\ell}}\sum_{\omega_{1}\neq \omega_{2}}\mathbb{E}_{\bm{0},0}\left[\left(\prod_{j=1}^{m}\prod_{e_{j}\in E_{\omega_{1},1_{j}}}A_{e_{j}}^{(j)}\prod_{e_{\ell}\in E_{\omega_{1},2}}B_{e_{\ell}}\right)\left(\prod_{j=1}^{m}\prod_{e_{j}\in E_{\omega_{2},1_{j}}}A_{e_{j}}^{(j)}\prod_{e_{\ell}\in E_{\omega_{2},2}}B_{e_{\ell}}\right)\right]
\end{align}
Now fix a cycle $\omega$, then by definition of B-type edges,
\begin{align*}
    \mathbb{E}_{\bm{0},0}\left[\prod_{j=1}^{m}\prod_{e_{j}\in E_{\omega,1_{j}}}A_{e_{j}}^{(j)}\prod_{e_{\ell}\in E_{\omega,2}}B_{e_{\ell}}^2\right]
    & = \mathbb{E}_{\bm{0},0}\left[\prod_{j=1}^{m}\prod_{e_{j}\in E_{\omega,1_{j}}}A_{e_{j}}^{(j)}\right] = \prod_{j=1}^{m}\mathbb{E}_{\bm{0},0}\left[\prod_{e_{j}\in E_{\omega,1_{j}}}A_{e_{j}}^{(j)}\right]\\
    & = \prod_{j=1}^{m}\left(\dfrac{d_{j}}{n}\right)^{k_{j}}
\end{align*}
implying that,
\begin{align*}
    T_{1} = \frac{1}{n^{2\ell}}\sum_{\omega}\prod_{j=1}^{m}\left(\dfrac{d_{j}}{n}\right)^{k_{j}}
\end{align*}
To complete the expression of $T_{1}$ we need to compute number of cycles $\omega$ having $k_{r}$ many type $E_{1_{r}}$ wedges for $1\leq r\leq m$ and $\ell$ many type $E_{2}$ wedges. Recall $k = \ell + \sum_{j=1}^{m}k_{j}$ then the number of such cycles can be easily computed to be,
\begin{align}\label{eq:count-cycle}
    \frac{1}{2}{n\choose k}(k-1)!{k\choose \ell}{k-\ell \choose k_{1}}{k-\ell-k_{1}\choose k_{2}}\cdots p^{\ell}. 
\end{align}
Then,
\begin{align}\label{eq:T1_order_1}
    T_{1} 
    & = \frac{1}{n^{2\ell}}\prod_{j=1}^{m}\left(\dfrac{d_{j}}{n}\right)^{k_{j}}\frac{1}{2}{n\choose k}(k-1)!{k\choose \ell}{k-\ell \choose k_{1}}{k-\ell-k_{1}\choose k_{2}}\cdots p^{\ell}\\
    & = \frac{1}{2k}\prod_{j=1}^{m}d_{j}^{k_{j}}\frac{1}{\gamma^{\ell}}\frac{k!}{\ell!\prod_{j=1}^{m}k_{j}!}(1+o(1)).
\end{align}
as long as $k=o(\sqrt{n})$. Now recalling the definition of $T_{2}$ from \eqref{eq:def-T2} it is easy to see that the expectation would be $0$ unless $\omega_{1}\neq \omega_{2}$ have exactly the same $B$ wedges. Using independence under $\bm{H}_{0}$ along with the above observation,
\begin{align*}
    T_{2} 
    & = \frac{1}{n^{2\ell}}\sum_{\substack{\omega_{1}\neq \omega_{2}\\E_{\omega_{1},2} = E_{\omega_{2},2}}}\mathbb{E}_{\bm{0},0}\left[\left(\prod_{j=1}^{m}\prod_{e_{j}\in E_{\omega_{1},1_{j}}}A_{e_{j}}^{(j)}\prod_{e_{\ell}\in E_{\omega_{1},2}}B_{e_{\ell}}\right)\left(\prod_{j=1}^{m}\prod_{e_{j}\in E_{\omega_{2},1_{j}}}A_{e_{j}}^{(j)}\prod_{e_{\ell}\in E_{\omega_{2},2}}B_{e_{\ell}}\right)\right]\\
    & = \frac{1}{n^{2\ell}}\sum_{\substack{\omega_{1}\neq \omega_{2}\\E_{\omega_{1},2} = E_{\omega_{2},2}}}\mathbb{E}_{\bm{0},0}\left[\prod_{j=1}^{m}\prod_{e_{j}\in E_{\omega_{1},1_{j}}}A_{e_{j}}^{(j)}\prod_{e_{j}\in E_{\omega_{2},1_{j}}}A_{e_{j}}^{(j)}\right].
\end{align*}
Fix $\omega_{1}\neq \omega_{2}$ such that $E_{\omega_{1},2} = E_{\omega_{2},2}$ then suppose $\omega_{1}$ and $\omega_{2}$ share $a$ many $\mathbb{A}$ type wedges, that is,
\begin{align*}
    \left|\bigcup_{i=1}^{2}\bigcup_{j=1}^{m}E_{\omega_{i},1_{j}}\right| = 2a.
\end{align*}
Since $\omega_{1}\neq \omega_{2}$, and they cannot differ in one edge, hence $0\leq a\leq k-\ell-2$. Define $b = k-\ell-a$. Define,
\begin{align*}
    \mathcal{X}_{b} = \left\{\omega_{1}\neq \omega_{2}:E_{\omega_{1},2} = E_{\omega_{2},2}\text{ and } \left|\bigcup_{i=1}^{2}\bigcup_{j=1}^{m}E_{\omega_{i},1_{j}}\right| = 2(k-\ell-b)\right\}.
\end{align*}
Then,
\begin{align}\label{eq:expression_T2}
    T_{2} = \frac{1}{n^{2\ell}}\sum_{b=2}^{k-\ell}\sum_{\mathcal{X}_{b}}\mathbb{E}_{\bm{0},0}\left[\prod_{j=1}^{m}\prod_{e_{j}\in E_{\omega_{1},1_{j}}}A_{e_{j}}^{(j)}\prod_{e_{j}\in E_{\omega_{2},1_{j}}}A_{e_{j}}^{(j)}\right].
\end{align}
Fix $b$ and consider $\omega_{1},\omega_{2}\in \mathcal{X}_{b}$. Suppose that there are $\delta_{r}$ many common wedges contributed by $\bm{A}_{r}$ for all $1\leq r\leq m$. Then we must have,
\begin{align*}
    \sum_{j=1}^{m}\delta_{j} = k-\ell-b.
\end{align*}
Recall that under $\bm{H}_{0},\bm{A}_{r}, 1\leq r\leq m$ are independent and hence for above choice of $\omega_{1}$ and $\omega_{2}$ we have,
\begin{align*}
    \mathbb{E}_{\bm{0},0}\left[\prod_{j=1}^{m}\prod_{e_{j}\in E_{\omega_{1},1_{j}}}A_{e_{j}}^{(j)}\prod_{e_{j}\in E_{\omega_{2},1_{j}}}A_{e_{j}}^{(j)}\right]
    & = \prod_{j=1}^{m}\mathbb{E}\left[\prod_{e_{j}\in E_{\omega_{1},1_{j}}}A_{e_{j}}^{(j)}\prod_{e_{j}\in E_{\omega_{2},1_{j}}}A_{e_{j}}^{(j)}\right]\\
    & = \prod_{j=1}^{m}\left(\dfrac{d^{(j)}}{n}\right)^{\delta_{j} + 2(k_{j}-\delta_{j})} = O\left(\frac{d}{n}\right)^{k-\ell+b}.
\end{align*}
Hence by \eqref{eq:expression_T2} we have,
\begin{align}\label{eq:T2_order_1}
    T_{2} = O\left(\frac{1}{n^{2\ell}}\sum_{b=2}^{k-\ell}\sum_{\mathcal{X}_{b}}\left(\frac{d}{n}\right)^{k-\ell+b}\right).
\end{align}
Now we need an upper bound on $\left|\mathcal{X}_{b}\right|$ for $2\leq b\leq k-\ell$. Observe that similar to \eqref{eq:count-cycle} we can choose the first cycle $\omega_{1}$ in $O\left(n^{k}p^{\ell}\right)$ many ways whenever $k=o(\sqrt{n})$. By definition of $\mathcal{X}_{b}$, the second cycle $\omega_{2}$ can be chosen in $O\left(n^{b-1}\right)$ ways. Hence,
\begin{align*}
    \left|\mathcal{X}_{b}\right| = O\left(n^{k+\ell + b-1}\right)
\end{align*}
Recalling $k = o(\sqrt{\log n})$, then by \eqref{eq:T2_order_1} we conclude that $T_{2} = o\left(1\right)$. Finally combining with \eqref{eq:T1_order_1} we have,
\begin{align}\label{eq:asymp-gaussian-var}
    \mathbb{E}_{\bm{0},0}
    &\left[Y_{n,k_{1},\cdots,k_{m},\ell}^2\right] \overset{n\rightarrow\infty}{\longrightarrow} \frac{1}{2k}\prod_{j=1}^{m}d_{j}^{k_{j}}\frac{1}{\gamma^{\ell}}\frac{k!}{\ell!\prod_{j=1}^{m}k_{j}!}
\end{align}
Now to show asymptotic Gaussianity of $Y_{n,k_{j_{1}},\cdots,k_{j_{m}},\ell_{j}}, r_{1}+1\leq j\leq r$ we show that the limits of the moments satisfy Wick's formula, that is we will show that for all $\zeta\in\mathbb{N}$, $T_{n,i}\in \{Y_{n,k_{j_{1}},\cdots,k_{j_{m}},\ell_{j}}:r_{1}+1\leq j\leq r\}, i\in [\zeta]$,
\begin{align}\label{eq:asymp-Wick}
    \mathbb{E}\left[\prod_{\nu = 1}^{\zeta}T_{n,\nu}\right] = 
    \begin{cases}
    \sum_{\eta}\prod_{i=1}^{\zeta/2}\mathbb{E}\left[T_{n,\eta(i,1)}T_{n,\eta(i,2)}\right] + o(1)& \text{ if }\zeta\text{ is even}\\
    o(1) & \text{ otherwise }
    \end{cases}
\end{align}
where $\eta$ is a partition of $[\zeta]$ into $\frac{\zeta}{2}$ blocks of size two and for $j\in \{1,2\}$, $\eta(i,j)$ denotes the $j^{th}$ element of the $i^{th}$ block of $\eta$.\\

Fix $\zeta\in\mathbb{N}$, and consider a choice $\in \{Y_{n,k_{j_{1}},\cdots,k_{j_{m}},\ell_{j}}:r_{1}+1\leq j\leq r\}$ for $1\leq i\leq \zeta$. For notational convenience in the following we will consider $T_{n,i}$ to have cycles with $k_{i_{s}}$ many $\bm{A}_{s}$ many wedges for $1\leq s\leq m$ and $\ell_{i}$ many $\bm{B}$ type wedges. Then by definition,
\begin{align}\label{eq:T_n-nu-prod-moment-eq1}
    \mathbb{E}_{\bm{0},0}\left[\prod_{\nu=1}^{\zeta}T_{n,\nu}\right] 
    = \dfrac{1}{n^{\sum_{i=1}^{\zeta}l_{i}}}\sum_{\omega_{1:\zeta}}\mathbb{E}_{\bm{0},0}\left[\prod_{i=1 }^{\zeta}\left(\prod_{j=1}^{m}\prod_{e_{j}\in E_{\omega_{i},1_{j}}}A_{e_{j}}^{(j)}\right)\right]\mathbb{E}_{\bm{0},0}\left[\prod_{i=1}^{\zeta}\left(\prod_{e_{2}\in E_{\omega_{i},2}}B_{e_{2}}\right)\right]
\end{align}
where $\omega_{1:\zeta}$ is a collection of cycles $\omega_{1},\cdots,\omega_{\zeta}$ on the factor graph such that $\omega_{i}$ has $k_{i_{s}}$ many wedges coming from $\bm{A}_{s}$ for all $1\leq s\leq m$ and $\ell_{i}$ many $\bm{B}$ type wedges with $x_{i}$ contiguous block of $\mathbb{A}$ type wedges and $\bm{B}$ type wedges where in a cycle $\omega$ we call a block of wedges to be $\mathbb{A}$ type wedges if there are no $\bm{B}$ type wedge in that block. Further consider $\Omega$ to be the collection of all such $\omega_{1:\zeta}$. 

\revsag{Given a cycle $\omega$ consider $\mathcal{G}(\omega)$ to be the graph corresponding to it. Suppose that $\omega$ has $x$ contiguous blocks of $\mathbb{A}$ type wedges. For all $1\leq j\leq x$ consider $\mathcal{G}_{\mathbb{A}}(\omega,j)$ to be the $j^{th}$ block of $\mathbb{A}$ type wedges. Now, we construct a quotient graph $\mathcal{G}_{Q}(\omega)$ by identifying the block $\mathcal{G}_{\mathbb{A}}(\omega,j)$ to be a single vertex for all $1\leq j\leq x$. In other words, for each block of contiguous $\mathbb{A}$ type wedges, we collapse all associated edges and vertices into one single vertex called a quotient operation vertex. This operation is illustrated in Figure \ref{fig:quotient-graph} and is called a quotient operation. The $\bm B$ type edges and vertices are kept intact in the operation.} 

\begin{figure}
    \centering
    \includegraphics[scale = 0.8]{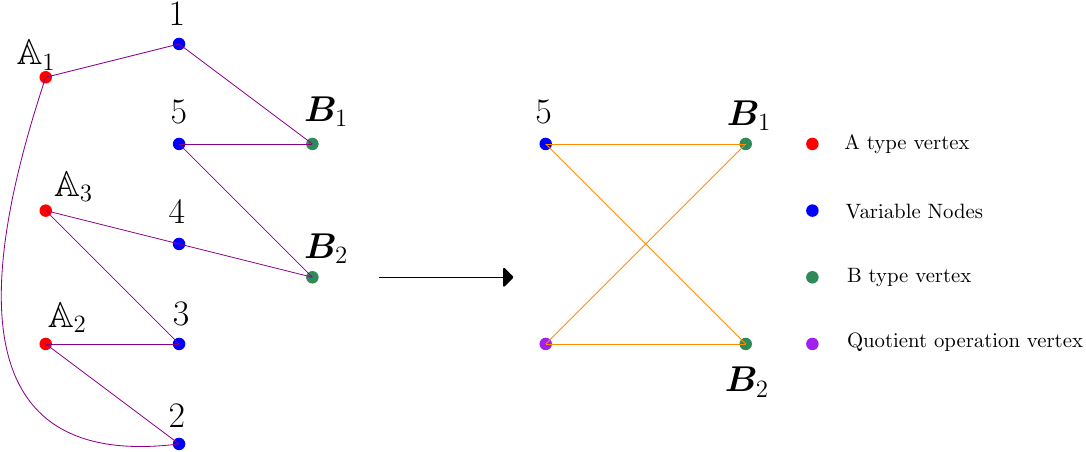}
    \caption{An example of construction of quotient graph by identifying the $\mathbb{A}$ type wedges $\{(1, \mathbb{A}_1,2), (2,\mathbb{A}_2,3), (3,\mathbb{A}_3, 4)\}$ to be a single vertex (quotient operation vertices) in the newly constructed graph.}
    \label{fig:quotient-graph}
\end{figure}

Since, the $\bm{B}$ type vertices, as well as edges, remain unchanged in the quotient graph $\mathcal{G}_{Q}(\omega)$, hence \eqref{eq:T_n-nu-prod-moment-eq1} becomes,
\begin{align}\label{eq:T_n-nu-prod-moment-eq2}
    \mathbb{E}_{\bm{0},0}\left[\prod_{\nu=1}^{\zeta}T_{n,\nu}\right] = \dfrac{1}{n^{\sum_{i=1}^{\zeta}l_{i}}}\sum_{\omega_{1:\zeta}}\mathbb{E}_{\bm{0},0}\left[\prod_{i=1 }^{\zeta}\left(\prod_{j=1}^{m}\prod_{e_{j}\in E_{\omega_{i},1_{j}}}A_{e_{j}}^{(j)}\right)\right]\mathbb{E}_{\bm{0},0}\left[\prod_{i=1}^{t}\left(\prod_{e\in E\left(\mathcal{G}_{Q}(\omega_{i})\right)}B_{e}\right)\right]
\end{align}
where $E\left(\mathcal{G}_{Q}(\omega_{i})\right)$ denotes the edges of the factor graph $\mathcal{G}_{Q}(\omega_{i})$ constructed from $\mathcal{G}(\omega_{i})$ for all $1\leq i\leq \zeta$. Now consider the following equivalence relation among the vertices formed by the quotient operation. For two cycles $\omega_{i_{1}}$ and $\omega_{i_{2}}$ if the first and the last vertices of $\mathcal{G}_{\mathbb{A}}(\omega_{i_{1}},j_{1})$ and $\mathcal{G}_{\mathbb{A}}(\omega_{i_{2}},j_{2})$ are the same for some $1\leq j_{1}\leq x_{i_{1}}$ and $1\leq j_{2}\leq x_{i_{2}}$, then we consider the corresponding vertices of $\mathcal{G}_{Q}(\omega_{i_{1}})$ and $\mathcal{G}_{Q}(\omega_{i_{2}})$ to be the same. For each $\omega_{1:\zeta}$ consider $\mathcal{G}(\omega_{1:\zeta},\mathbb{A})$ to be the collection of $\mathbb{A}$ type wedges coming from $\omega_{1:\zeta}$ and $\mathcal{G}_{Q}(\omega_{1:\zeta})$ to be the collection of quotient graph constructed from $\omega_{i}, 1\leq i\leq \zeta$. Further consider,
\begin{align*}
    \mathbb{G}(\mathbb{A}) = \left\{\mathcal{G}(\omega_{1:\zeta},\mathbb{A}):\omega_{1:\zeta}\in \Omega\right\}
\end{align*}
and given $g\in \mathbb{G}(\mathbb{A})$ let $\omega_{1:\zeta}^{(g)}:=(\omega_{1}^{(g)},\cdots,\omega_{\zeta}^{(g)})\in \Omega$ be such that $\mathcal{G}(\omega_{1:\zeta}^{(g)},\mathbb{A}) = g$. Observe that by construction $\omega_{g}$ is not unique. Then from \eqref{eq:T_n-nu-prod-moment-eq2} we have,
\begin{align}
    \mathbb{E}_{\bm{0},0}\left[\prod_{\nu=1}^{\zeta}T_{n,\nu}\right] =
    \dfrac{1}{n^{\sum_{i=1}^{\zeta}l_{i}}}
    &\left(\sum_{g\in\mathbb{G}(\mathbb{A})}\mathbb{E}_{\bm{0},0}\left[\prod_{i=1 }^{\zeta}\left(\prod_{j=1}^{m}\prod_{e_{j}\in E_{\omega_{i}^{(g)},1_{j}}}A_{e_{j}}^{(j)}\right)\right]\right)\nonumber\\
    &\left(\sum_{\mathcal{G}_{Q}(\omega_{1:\zeta})}\mathbb{E}_{\bm{0},0}\left[\prod_{i=1}^{t}\left(\prod_{e\in E\left(\mathcal{G}_{Q}(\omega_{i})\right)}B_{e}\right)\right]\right)\label{eq:T-moment-decomp}
\end{align}
Now consider,
\begin{align*}
    \mathbb{G}_{1}(\mathbb{A}) = \left\{g\in \mathbb{G}(\mathbb{A}): \text{ no overlap among $\mathbb{A}$ type edges in }\omega_{i}^{(g)}, 1\leq i\leq \zeta\right\}\text{ and }\mathbb{G}_{2}(\mathbb{A}) = \mathbb{G}(\mathbb{A})\setminus\mathbb{G}_{1}(\mathbb{A})
\end{align*}
It is easy to observe that for $\omega_{1:\zeta}\in \mathbb{G}_{1}(\mathbb{A})$
\begin{align*}
    \mathbb{E}_{\bm{0},0}\left[\prod_{i=1 }^{\zeta}\left(\prod_{j=1}^{m}\prod_{e_{j}\in E_{\omega_{i},1_{j}}}A_{e_{j}}^{(j)}\right)\right] = O\left(\dfrac{d}{n}\right)^{\sum_{i=1}^{\zeta}\sum_{j\leq x_{i}}\beta_{j}^{(i)}}
\end{align*}
where $\beta_{j}^{(i)}$ is the number of wedges in the $j^{th}$ block of $\omega_{i}$ and $d = \max_{j=1}^{m}d_{j}$. Further,
\begin{align*}
    \left|\mathbb{G}_{1}(\mathbb{A})\right| = O(n)^{\sum_{i=1}^{\zeta}\sum_{j\leq x_{i}}(\beta_{j}^{(i)}+1)}
\end{align*}
which follows since in a $\mathbb{A}$ type block having $\beta$ wedges, there must be $\beta+1$ variable nodes of which there are $O(n)$ many options (Note that for every wedge we have $m$ many options which is constant in $n$). Thus the total contribution coming from non-overlapping cycles is,
\begin{align}\label{eq:A-contribution}
    O\left(\dfrac{d}{n}\right)^{\sum_{i=1}^{\zeta}\sum_{j\leq x_{i}}\beta_{j}^{(i)}}O(n)^{\sum_{i=1}^{\zeta}\sum_{j\leq x_{i}}(\beta_{j}^{(i)}+1)}
\end{align}
Notice that the dominant contribution comes from $\mathbb{G}_{1}(\mathbb{A})$, because if we consider cycles having $\xi$ many overlapping wedges then, we gain a term of $O(n^{\xi})$ in computing the expectation, while we lose a term of $O(n^{\xi+1})$ while counting above number of cycles as overlap in $\xi$ many wedges implies we must have overlap in $\xi+1$ many vertices. Hence,
\begin{align*}
    \sum_{g\in\mathbb{G}(\mathbb{A})}\mathbb{E}_{\bm{0},0}\left[\prod_{i=1 }^{\zeta}\left(\prod_{j=1}^{m}\prod_{e_{j}\in E_{\omega_{i}^{(g)},1_{j}}}A_{e_{j}}^{(j)}\right)\right] = O\left(\dfrac{d}{n}\right)^{\sum_{i=1}^{\zeta}\sum_{j\leq x_{i}}\beta_{j}^{(i)}}O(n)^{\sum_{i=1}^{\zeta}\sum_{j\leq x_{i}}(\beta_{j}^{(i)}+1)}.
\end{align*}
Next, we investigate the term contributed by $\bm{B}$ type wedges. Define,
\begin{align*}
    \mathbb{G}_{W}(\bm{B}) = \left\{\mathcal{G}_{Q}(\omega_{1:\zeta}): \forall e\in \bigcup_{i=1}^{\zeta} E(\mathcal{G}_{Q}(\omega_{i})),\ \exists i_{1}\neq i_{2}\text{ such that }e\in E(\mathcal{G}_{Q}(\omega_{i_{c}})), c=1,2\right\}.
\end{align*}
Then it is easy to observe that only the contribution of $\mathbb{G}_{w}(\bm{B})$ is non-zero. Now, for every $\omega_{1:\zeta}\in \mathbb{G}_{W}(\bm{B})$ we consider a partition $\eta(\omega_{1:\zeta})$ of $[\zeta]$ as follows, if $a$ and $b$ are in the same partition of $\eta(\omega_{1:\zeta})$ then $\mathcal{G}_{Q}(\omega_{a})$ and $\mathcal{G}_{Q}(\omega_{b})$ share at least one edge. By considering the decomposition from \eqref{eq:T-moment-decomp} along with the bounds in \eqref{eq:A-contribution} and the collection $\mathbb{G}(\bm{B})$, it is easy to see that \eqref{eq:asymp-Wick} follows from the proof of Proposition 2 in \cite{lu2020contextual}. The variance of the asymptotic Gaussian distribution was identified in \eqref{eq:asymp-gaussian-var}. Finally, we are left with verifying asymptotic independence. Observe that it is enough to show,
\begin{align*}
    \mathbb{E}_{\bm{0},0}\left[Y_{n,k_{i_{1},\cdots,i_{m},\ell_{i}}}Y_{n,k_{j_{1},\cdots,j_{m},\ell_{j}}}\right]\rightarrow 0
\end{align*}
whenever $r_{1}+1\leq i\neq j\leq r$. Note that the above expectation is $0$ is $\ell_{i}\neq \ell_{j}$. Then it suffices to consider $l_{i} = l_{j}$. Without loss of generality suppose $k_{i}< k_{j}$, then
\begin{align*}
    \mathbb{E}_{\bm{0},0}\left[Y_{n,k_{i_{1},\cdots,i_{m},\ell_{i}}}Y_{n,k_{j_{1},\cdots,j_{m},\ell_{j}}}\right]
    = \dfrac{1}{n^{2\ell_{i}}}\sum_{\omega_{i},\omega_{j}}\mathbb{E}_{\bm{0},0}\left[\prod_{s=1}^{m}\left(\prod_{e\in E_{\omega_{i},1_{s}}}A_{e}^{(s)}\right)\left(\prod_{e\in E_{\omega_{j},1_{s}}}A_{e}^{(s)}\right)\right]
\end{align*}
where for $s=i,j$, $\omega_{s}$ is a cycle having $k_{s_{a}}$ many type $\bm{A}_{s}$ wedges for all $1\leq a\leq m$ and $l_{s}$ many type $\bm{B}$ wedges. Further, the sum is over all those cycles $\omega_{i},\omega_{j}$ who intersect on all the $l_{i}=l_{j}$ type $\bm{B}$ wedges. As before one can show that the dominant term is provided by cycles having no overlapping type $\mathbb{A}$ edges and hence,
\begin{align*}
    \mathbb{E}_{\bm{0},0}\left[Y_{n,k_{i_{1},\cdots,i_{m},\ell_{i}}}Y_{n,k_{j_{1},\cdots,j_{m},\ell_{j}}}\right]
    = n^{-2\ell_{i}}O\left(\dfrac{d}{n}\right)^{k_{i}+k_{j}-2\ell_{i}}O(n)^{k_{i}+k_{j}-\ell_{i}-1} = o(1)
\end{align*}
which holds whenever $k_{i},k_{j}=o(\sqrt{\log n})$ and hence the proof is completed under $\bm{H}_{0}$.\qed
\subsection{Proof of Theorem \ref{thm:cycle-stat-asymp-dist}(2)}
Once again let us first consider the case when $\ell=0$. We shall show that as $n \rightarrow \infty$,
\begin{equation}
\label{eq:fact_mom}
    \mathbb{E}_{\bm \lambda,\mu}(Y_{n,k_1,\cdots,k_m,0})_{[M]}\rightarrow \left(\frac{1}{2k}\frac{k!}{\prod_{j=1}^mk_j!}\left\{\prod\limits_{j=1}^{m}d^{k_j}_j+\prod\limits_{j=1}^{m}(\lambda_j\sqrt{d_j})^{k_j}\right\}\right)^M,
\end{equation}
This, as was under $\bm{H}_{0}$, will imply the Poisson convergence. Recalling notations from \eqref{eq:rewrite-Y-1} we have,
\begin{align*}
    Y_{n,k_{1},\cdots,k_{m},0} = \sum_{C\in \mathcal{C}_{k_{1},\cdots,k_{m},0}}Y_{C}.
\end{align*}
Let $\rho_j$ be the number of type $\bm{A}_j$ wedge such that the two variable nodes on either side have different community labels. Let us observe that since the subgraph is a cycle, $\rho_1+\cdots+\rho_m$ is even. By the argument of Lemma 3.3 of \cite{mnsf},
\begin{align}\label{eq:exp-YC-H1}
\mathbb{E}_{\bm \lambda,\mu}[Y_C]&=n^{-k}2^{-k}\sum\limits_{\mbox{\small{$\rho_1+\cdots+\rho_m$ even}}}\;\prod_{j=1}^{m}{k_j \choose \rho_j}a^{k_j-\rho_j}_jb^{\rho_j}_j\nonumber\\
&=n^{-k}2^{-k}\left(\prod_{j=1}^{m}(a_j+b_j)^{k_j}+\prod_{j=1}^{m}(a_j-b_j)^{k_j}\right)\nonumber\\
&=n^{-k}\left(\prod_{j=1}^{m}d^{k_j}_j+\prod_{j=1}^{m}(\lambda_j\sqrt{d_j})^{k_j}\right)
\end{align}
By \eqref{eq:count-M-tuple-factor-graph} and \eqref{eq:exp-YC-H1} we get,
\begin{align*}
    \mathbb{E}_{\bm{\lambda},\mu}\left[Y_{n,k_{1},\cdots,k_{m},0}\right] = (1+o(1))\frac{1}{2k}\frac{k!}{k_1!\cdots k_m!}\left(\prod_{j=1}^{m}d^{k_j}_j+\prod_{j=1}^{m}(\lambda_j\sqrt{d_j})^{k_j}\right)\text{ whenever }k=o(\sqrt{n})
\end{align*}
which proves \eqref{eq:fact_mom} for $M=1$. Define,
\begin{align*}
    \mu_{\text{Poi},\bm{H}_{1}} = \frac{1}{2k}\frac{k!}{k_1!\cdots k_m!}\left(\prod_{j=1}^{m}d^{k_j}_j+\prod_{j=1}^{m}(\lambda_j\sqrt{d_j})^{k_j}\right)
\end{align*}
For $M>1$, we now show that,
\begin{align*}
    \mathbb{E}_{\bm \lambda,\mu}(Y_{n,k_1,\cdots,k_m,0})_{[M]}\rightarrow \mu_{\text{Poi},\bm{H_{1}}}^M
\end{align*}
Considering the $M$ tuple of vertex disjoint cycles $(C_{1},\cdots,C_{M})$ similar to \eqref{eq:vertex-disjoint-M-prob} and using \eqref{eq:exp-YC-H1} we get,
\begin{align*}
    \mathbb{P}_{\bm{\lambda},\mu}\left(C_{i}\in G_{F}, 1\leq i\leq M\right) = n^{-kM}\left[\left(\prod_{j=1}^{m}d^{k_j}_j+\prod_{j=1}^{m}(\lambda_j\sqrt{d_j})^{k_j}\right)\right]^{M}.
\end{align*}
Recalling the number of possible choices for such $M$ tuple of vertex disjoint cycles from \eqref{eq:count-vertex-disjoint} shows that contribution of collection of $M$ tuple of vertex disjoint cycles in $\mathbb{E}_{\bm \lambda,\mu}(Y_{n,k_1,\cdots,k_m,0})_{[M]}$ as $n\rightarrow\infty$ is given by,
\begin{align*}
    \left(\frac{1}{2k}\frac{k!}{k_1!\cdots k_m!}\left(\prod_{j=1}^{m}d^{k_j}_j+\prod_{j=1}^{m}(\lambda_j\sqrt{d_j})^{k_j}\right)\right)^{M} = \mu_{\text{Poi},\bm{H}_{1}}^{M}\text{ whenever }k = o(\sqrt{n})
\end{align*}
Considering the Erd\H{o}s-R\'enyi random graph with edge probability $\max_{j=1}^{m}\{a_{i}/n,b_{i}/n\}$ and repeating the stochastic dominance argument from \eqref{eq:stochastic-dominance} shows that the contribution of $M$ tuple of vertex-overlapping cycles is asymptotically negligible completing the proof of \eqref{eq:fact_mom}.\\

Next, we consider the case when $\ell \neq 0$. First, we try to compute the mean of the cycle statistics. Let us observe that,
\[
\mathbb{E}_{\bm \lambda,\mu}[B_{i_1j_1}B_{i_2j_1}]=\frac{\mu}{n}\sigma_{i_1}\sigma_{i_2},
\]
for any $\bm{B}$ wedge $(i_1,j_1,i_2)$. Now let us fix a cycle $\omega$ with $k_1$ type $1$ wedge, $k_2$ type $2$ wedge up to $k_m$ type $m$ wedge and $\ell$ type $\bm{B}$ wedges.
For this cycle,
\begin{align}
    \mathbb{E}_{\bm \lambda,\mu}\left[\prod\limits_{k=1}^{m}\prod\limits_{e_k \in E_{\omega,1_k}}\prod\limits_{e_\ell \in E_{\omega,2}}A^{(k)}_{e_k}B_{e_\ell}|\bm \sigma\right]&=\left[
   \prod\limits_{k=1}^{m}\prod\limits_{(e^{+}_k,e^{-}_k) \in E_{\omega,1_k}}\left(\frac{d_k+\lambda_k\sqrt{d_k}\sigma_{e^{+}_k}\sigma_{e^{-}_k}}{n}\right)\prod\limits_{(e^{+}_\ell,e^{-}_\ell) \in E_{\omega,2}}\frac{\mu}{n}\sigma_{e^{+}_\ell}\sigma_{e^{-}_\ell}\right].
\end{align}
Hence,
\begin{align}
    \mathbb{E}_{\bm \lambda,\mu}[Y_{n,k_1,\cdots,k_m,\ell}]&=\frac{1}{n^{\ell}}\sum_{\omega}\mathbb E_{\bm \sigma}\left[
   \prod\limits_{k=1}^{m}\prod\limits_{(e^{+}_k,e^{-}_k) \in E_{\omega,1_k}}\left(\frac{d_k+\lambda_k\sqrt{d_k}\sigma_{e^{+}_k}\sigma_{e^{-}_k}}{n}\right)\prod\limits_{(e^{+}_\ell,e^{-}_\ell) \in E_{\omega,2}}\frac{\mu}{n}\sigma_{e^{+}_\ell}\sigma_{e^{-}_\ell}\right]\\
   &=\frac{\mu^\ell}{n^{k+\ell}}\sum_\omega\mathbb{E}_{\bm \sigma}\left[\sum_{\substack{z_{e_k}\in\{d_k,\lambda_k\sqrt{d_k}\sigma_{e^+_{k}e^-_{k}}\}\\k \in \{1,\cdots,M\}}}\prod\limits_{k=1}^{M}\prod\limits_{e_k \in E_{\omega,1_k}}z_{e_k}\prod_{e_\ell \in E_{\omega,2}}\sigma_{e^{+}_\ell}\sigma_{e^{-}_\ell}\right]\\
   &\overset{(1)}{=}\frac{\mu^\ell}{n^{k+\ell}}{n \choose k}\frac{k!}{k_1!\cdots k_m!\ell!}\frac{(k-1)!}{2}p^\ell\prod\limits_{j=1}^{m}(\lambda_j\sqrt{d_j})^{k_j}\\
   &=\left(\frac{\mu}{\gamma}\right)^{\ell}\frac{1}{2k}\frac{k!}{k_1!\cdots k_m!\ell!}\prod\limits_{j=1}^{m}(\lambda_j\sqrt{d_j})^{k_j}(1+o(1)),
\end{align}
where equality $(1)$ follows because the only contribution comes from the terms where all $z_{e_{k}}$'s are $\lambda_k\sqrt{d_k}\sigma_{e^+_{k}e^-_{k}}$. Next, we observe that under $\bm H_1$, for all $\omega$ with $k_j$ type $j$ wedges for $k=1,\cdots,M$ and $\ell$ type $\bm{B}$ wedges and $(i,j) \in E_{\omega,2}$,
\[
B_{ij}=X_{ij}+Z_{ij}, \quad \mbox{where $X_{ij}=\sqrt{\mu/n}\sigma_iu_j$ and $Z_{ij}$ are i.i.d Gaussians.}
\]
Thus we have,
\begin{align}
    Y_{n,k_1,\cdots,k_m,\ell}&=\frac{1}{n^\ell}\sum_{\omega}\prod\limits_{k=1}^{M}\prod\limits_{e_k \in E_{\omega,1_k}}A^{(k)}_{e_k}\prod\limits_{e_\ell \in E_{\omega,2}}B_{e_\ell}\\
    &=T_1+T_2+T_3,
\end{align}
where
\[
T_1=\frac{1}{n^{\ell}}\sum\limits_{\omega}\prod\limits_{k=1}^{M}\prod\limits_{e_k \in E_{\omega,1_k}}A^{(k)}_{e_k}\prod\limits_{e_\ell \in E_{\omega,2}}Z_{e_\ell},
\]
\[
T_2=\frac{1}{n^{\ell}}\sum\limits_{\omega}\prod\limits_{k=1}^{M}\prod\limits_{e_k \in E_{\omega,1_k}}A^{(k)}_{e_k}\prod\limits_{e_\ell \in E_{\omega,2}}X_{e_\ell},
\]
and
\[
T_3=Y_{n,k_1,\cdots,k_m,\ell}-T_1-T_2.
\]
Let us observe that $T_1$ can be analyzed in the same way as in $\bm H_0$. Hence,
\[
T_1 \overset{d}{\rightarrow}N\left(0,\frac{1}{2k}\;\frac{k!}{k_1!\cdots k_{j}!\ell!}\prod\limits_{j=1}^{m}d_{j}^{k_j}\right).
\]
Next, we shall show that,
\[
T_2 \overset{P}{\rightarrow}\frac{1}{2k}\;\frac{k!}{k_1!\cdots k_{j}!\ell!}\prod\limits_{j=1}^{m}(\lambda_j\sqrt{d_j})^{k_j}\left(\frac{\mu}{\gamma}\right)^{\ell}.
\]
Finally, we shall show $T_3 \overset{P}{\rightarrow}0$. This will complete the proof of the theorem. Let us begin with $T_3$. We observe that $E_{\bm \lambda,\mu}[T_3]=0$. So it suffices to show that $\mathbb E_{\bm \lambda,\mu}[T^2_3]=o(1)$ as $n \rightarrow \infty$. Let us express $T_3=\sum_\omega V_{n,k_1,\cdots,k_m,\ell,\omega}$, where
\[
V_{n,k_1,\cdots,k_m,\ell,\omega}=\frac{1}{n^\ell}\prod\limits_{k=1}^{m}\prod\limits_{e_k \in E_{\omega,1_k}}A^{(k)}_{e_k}\sum\limits_{E_{\omega,f}\subseteq E_{\omega,2}}\prod\limits_{e_\ell \in E_{\omega,f}}X_{e_\ell}\prod\limits_{e_\ell \in E_{\omega,2}\setminus E_{\omega,f}}Z_{e_\ell}.
\]
Now by argument similar to proof of Proposition 1, Calculations under $\bm H_1$ of \cite{lu2020contextual} we can show $T_3 \overset{P}{\rightarrow}0$. Now we try to analyze $T_2$. Let us observe that, since,
\[
\mathbb{E}_{\bm \lambda,\mu}[Y_{n,k_1,\cdots,k_m,\ell}]=\left(\frac{\mu}{\gamma}\right)^{\ell}\frac{1}{2k}\frac{k!}{k_1!\cdots k_m!\ell!}\prod\limits_{j=1}^{m}(\lambda_j\sqrt{d_j})^{k_j}(1+o(1)),
\]
\[
\mathbb{E}_{\bm \lambda,\mu}[T_1]=o(1) \quad \mbox{and} \quad \mathbb{E}_{\bm \lambda,\mu}[T_3]=o(1),
\]
It is enough to show,
\[
\mbox{Var}\left(T_2\right)=o(1).
\]
For a fixed set of type $B$ factor nodes $j_1,\cdots,j_\ell$, let us denote all cycles with $k_j$ type $j$ wedges and $\ell$ type B wedges involving the above-mentioned factor nodes by $\mathcal{C}_{j_1:j_\ell}$. Let us observe that,
\[
T_2=\frac{1}{n^\ell}\left(\sum\limits_{\omega \in \mathcal{C}_{j_1:j_\ell}}\prod\limits_{k=1}^{m}\prod\limits_{e_k \in E_{\omega,1_k}}A^{(k)}_{e_k}\prod\limits_{e_\ell \in E_{\omega,2}}\sigma^{+}_{e_\ell}\sigma^{-}_{e_\ell}\right)\left(\left(\frac{\mu}{n}\right)^\ell\sum_{j_1:j_\ell}\prod\limits_{h=1}^{\ell}u^2_{j_h}\right).
\]
As $u_{j_h}$'s are independent, standard Gaussians; by the law of large numbers,
\[
\frac{1}{n^\ell}\sum_{j_1:j_\ell}\prod\limits_{h=1}^{\ell}u^2_{j_h}\overset{P}{\rightarrow}1.
\]
So, it is enough to show,
\[
\mbox{Var}_{\bm \lambda,\mu}\left(\frac{1}{n^\ell}\left(\sum\limits_{\omega \in \mathcal{C}_{j_1:j_\ell}}\prod\limits_{k=1}^{m}\prod\limits_{e_k \in E_{\omega,1_k}}A^{(k)}_{e_k}\prod\limits_{e_\ell \in E_{\omega,2}}\sigma^{+}_{e_\ell}\sigma^{-}_{e_\ell}\right)\right)=o(1),
\]
or equivalently,
\begin{multline}
\mathbb{E}_{\bm \lambda,\mu}\left[\frac{1}{n^{2\ell}}\left(\sum\limits_{\omega \in \mathcal{C}_{j_1:j_\ell}}\prod\limits_{k=1}^{m}\prod\limits_{e_k \in E_{\omega,1_k}}A^{(k)}_{e_k}\prod\limits_{e_\ell \in E_{\omega,2}}\sigma^{+}_{e_\ell}\sigma^{-}_{e_\ell}\right)^2\right]\\
=\mathbb{E}_{\bm \lambda,\mu}\left[\frac{1}{n^\ell}\left(\sum\limits_{\omega \in \mathcal{C}_{j_1:j_\ell}}\prod\limits_{k=1}^{m}\prod\limits_{e_k \in E_{\omega,1_k}}A^{(k)}_{e_k}\prod\limits_{e_\ell \in E_{\omega,2}}\sigma^{+}_{e_\ell}\sigma^{-}_{e_\ell}\right)\right]^2+o(1).
\end{multline}
This follows by argument similar to the one used in the proof of Lemma \ref{eq:asymp_ind}.\qed
\subsection{Proof of Lemma \ref{eq:asymp_ind}}\label{eq:proof_of_decouple_lemma}
Let us consider the following quantities for $\omega_j=\omega_{k_{j1},\cdots,k_{jm},\ell_j}$ for $1 \le j \le r$.
\[
T_1:=\prod\limits_{j=1}^{r_1}\left(\sum\limits_{\omega_{j}}\prod\limits_{e_1 \in E_{\omega_j,1_{1}}}\cdots\prod\limits_{e_m \in E_{\omega_j,1_m}}A^{(1)}_{e_1}\cdots A^{(m)}_{e_m}\right)^{m_j},
\]
and
\[
T_2:=\prod\limits_{j=r_1+1}^{r}\left(\frac{1}{n^{\ell_j}}\sum\limits_{\omega_j}\prod\limits_{e_1 \in E_{\omega_j,1_{1}}}\cdots\prod\limits_{e_m \in E_{\omega_j,1_m}}\prod\limits_{e_\ell \in E_{\omega_j,2}}A^{(1)}_{e_1}\cdots A^{(m)}_{e_m}B_{e_\ell}\right)^{m_j}.
\]
Expanding these terms we get,
\[
T_1:=\prod\limits_{j=1}^{r_1}\left(\sum\limits_{\omega_{j1},\cdots,\omega_{jm_j}}\prod\limits_{q=1}^{m_j}\prod\limits_{e_1 \in E_{\omega_{jq},1_{1}}}\cdots\prod\limits_{e_m \in E_{\omega_{jq},1_{m}}}A^{(1)}_{e_1}\cdots A^{(m)}_{e_m}\right),
\]
and
\[
T_2:=\prod\limits_{j=r_1+1}^{r}\left(\frac{1}{n^{\ell_jm_j}}\sum\limits_{\omega_{j1},\cdots,\omega_{jm_j}}\prod\limits_{q=1}^{m_j}\prod\limits_{e_1 \in E_{\omega_{jq},1_{1}}}\cdots\prod\limits_{e_m \in E_{\omega_{jq},1_{m}}}\prod\limits_{e_\ell \in E_{\omega_{jq},2}}A^{(1)}_{e_1}\cdots A^{(m)}_{e_m}B_{e_\ell}\right).
\]
Under $\bm H_0$,
\begin{align}
    \mathbb E_0[T_1T_2]&=\sum\limits_{\omega_{11},\cdots,\omega_{1m_1}}\cdots\sum\limits_{\omega_{r1},\cdots,\omega_{rm_r}}\frac{1}{n^{\sum_{j=r_1+1}^{r}\ell_jm_j}}\mathbb{E}_{\bm 0,0}\Bigg[\prod_{k=1}^{m}\prod_{t=1}^{r}\prod_{v=1}^{m_t}\prod_{e_{tk} \in E_{\omega_{tv,1_k}}}A^{(k)}_{e_{tk}}\prod_{t=r_1+1}^{r}\prod_{v=1}^{m_\ell}\prod\limits_{e_{t\ell} \in E_{\omega_{tv},2}}B_{e_{t\ell}}\Bigg]
\end{align}
By model assumptions,
\begin{align}
    &\sum\limits_{\omega_{11},\cdots,\omega_{1m_1}}\cdots\sum\limits_{\omega_{r1},\cdots,\omega_{rm_r}}\frac{1}{n^{\sum_{j=r_1+1}^{r}\ell_jm_j}}\mathbb{E}_{\bm 0,0}\Bigg[\prod_{k=1}^{m}\prod_{t=1}^{r}\prod_{v=1}^{m_t}\prod_{e_{tk} \in E_{\omega_{tv,1_k}}}A^{(k)}_{e_{tk}}\prod_{t=r_1+1}^{r}\prod_{v=1}^{m_\ell}\prod\limits_{e_{t\ell} \in E_{\omega_{tv},2}}B_{e_{t\ell}}\Bigg]\\
    =&\sum\limits_{\omega_{11},\cdots,\omega_{1m_1}}\cdots\sum\limits_{\omega_{r1},\cdots,\omega_{rm_r}}\frac{1}{n^{\sum_{j=r_1+1}^{r}\ell_jm_j}}\Bigg\{\mathbb{E}_{\bm 0,0}\Bigg[\prod_{k=1}^{m}\Bigg(\prod_{t=1}^{r}\prod_{v=1}^{m_t}\prod_{e_{tk} \in E_{\omega_{tv,1_k}}}A^{(k)}_{e_{tk}}\Bigg)\Bigg(\prod_{t=r_1+1}^{r}\prod_{v=1}^{m_t}\prod_{e_{tk} \in E_{\omega_{tv,1_k}}}A^{(k)}_{e_{tk}}\Bigg)\Bigg]\\
    &\times\mathbb{E}_{\bm 0,0}\Bigg[\prod_{t=r_1+1}^{r}\prod_{v=1}^{m_\ell}\prod\limits_{e_{t\ell} \in E_{\omega_{tv},2}}B_{e_{t\ell}}\Bigg]\Bigg\}.
\end{align}
Similar calculations also show that,
\begin{align}
 \mathbb{E}_{\bm 0,0}\left[\prod\limits_{j=1}^{r_1}Y^{m_j}_{n,k_{j_{1}},\cdots,k_{j_{m}},0}\right]\hspace{-0.01in}\mathbb{E}_{\bm 0,0}\left[\prod\limits_{j=r_1+1}^{r}Y^{m_j}_{n,k_{j_{1}},\cdots,k_{j_{m}},l_{j}}\right]&= \sum\limits_{\omega_{11},\cdots,\omega_{1m_1}}\cdots\sum\limits_{\omega_{r1},\cdots,\omega_{rm_r}}\frac{1}{n^{\sum_{j=r_1+1}^{r}\ell_jm_j}}\\
 &\times\Bigg\{\mathbb{E}_{\bm 0,0}\Bigg[\prod_{k=1}^{m}\prod_{t=1}^{r}\prod_{v=1}^{m_t}\prod_{e_{tk} \in E_{\omega_{tv,1_k}}}A^{(k)}_{e_{tk}}\Bigg]\\
&\times\mathbb{E}_{\bm 0,0}\Bigg[\prod_{t=r_1+1}^{r}\prod_{v=1}^{m_t}\prod_{e_{tk} \in E_{\omega_{tv,1_k}}}A^{(k)}_{e_{tk}}\Bigg]\\
&\times\mathbb{E}_{\bm 0,0}\Bigg[\prod_{t=r_1+1}^{r}\prod_{v=1}^{m_\ell}\prod\limits_{e_{t\ell} \in E_{\omega_{tv},2}}B_{e_{t\ell}}\Bigg]\Bigg\}.      
\end{align}
Let us observe that these two terms seem to be different because of the presence of overlapping cycles. For example, if two cycles overlap on $m$-edges we gain a factor of $O(n^m)$ in the first term. But, then they also overlap in at least $m+1$ vertices and thus we lose $O(n^{m+1})$ terms in choosing the cycles. Hence the dominant contribution comes from the non-overlapping cycles, implying part $(i)$ of the Lemma.

Under $\bm H_1$, we observe that conditioned on $\bm \sigma$, $\bm{A}_{k}$ for $k=1\cdots,m$ and $\bm{B}$ are independent. Hence, to compute $\mathbb{E}_{\bm \lambda, \mu}[\prod_{j=1}^{r}Y^{m_j}_{n,k_{j_{1}},\cdots,k_{j_{m}},l_{j}}]$, we first condition on $\bm \sigma$. So we get,
\begin{multline}
    \mathbb E_{\bm \lambda,\mu}[T_1T_2]
    \\=\mathbb{E}_{\bm \sigma}\Bigg[\sum\limits_{\omega_{11},\cdots,\omega_{1m_1}}\cdots\sum\limits_{\omega_{r1},\cdots,\omega_{rm_r}}\frac{1}{n^{\sum_{j=r_1+1}^{r}\ell_jm_j}}\mathbb{E}_{\bm \lambda,\mu}\Bigg[\prod_{k=1}^{m}\prod_{t=1}^{r}\prod_{v=1}^{m_t}\prod_{e_{tk} \in E_{\omega_{tv,1_k}}}A^{(k)}_{e_{tk}}\prod_{t=r_1+1}^{r}\prod_{v=1}^{m_\ell}\prod\limits_{e_{t\ell} \in E_{\omega_{tv},2}}B_{e_{t\ell}}|\bm \sigma\Bigg]\Bigg]
\end{multline}
Using model assumptions and calculations in previous part,
\begin{align}
    \sum\limits_{\omega_{11},\cdots,\omega_{1m_1}}\cdots\sum\limits_{\omega_{r1},\cdots,\omega_{rm_r}}\frac{1}{n^{\sum_{j=r_1+1}^{r}\ell_jm_j}}\mathbb{E}_{\bm \lambda,\mu}\Bigg[\prod_{k=1}^{m}\prod_{t=1}^{r}\prod_{v=1}^{m_t}\prod_{e_{tk} \in E_{\omega_{tv,1_k}}}A^{(k)}_{e_{tk}}\prod_{t=r_1+1}^{r}\prod_{v=1}^{m_\ell}\prod\limits_{e_{t\ell} \in E_{\omega_{tv},2}}B_{e_{t\ell}}|\bm \sigma\Bigg]\\
    =\sum\limits_{\omega_{11},\cdots,\omega_{1m_1}}\cdots\sum\limits_{\omega_{r1},\cdots,\omega_{rm_r}}\frac{1}{n^{\sum_{j=r_1+1}^{r}\ell_jm_j}}\Bigg\{\mathbb{E}_{\bm \lambda,\mu}\Bigg[\prod_{k=1}^{m}\Bigg(\prod_{t=1}^{r}\prod_{v=1}^{m_t}\prod_{e_{tk} \in E_{\omega_{tv,1_k}}}A^{(k)}_{e_{tk}}\Bigg)\Bigg(\prod_{t=r_1+1}^{r}\prod_{v=1}^{m_t}\prod_{e_{tk} \in E_{\omega_{tv,1_k}}}A^{(k)}_{e_{tk}}\Bigg)|\bm \sigma\Bigg]\\
    \times\mathbb{E}_{\bm \lambda,\mu}\Bigg[\prod_{t=r_1+1}^{r}\prod_{v=1}^{m_\ell}\prod\limits_{e_{t\ell} \in E_{\omega_{tv},2}}B_{e_{t\ell}}|\bm \sigma\Bigg]\Bigg\}.
\end{align}
Similarly,
\begin{align}
 \E_{\bm \sigma} & \Bigg[\mathbb{E}_{\bm \lambda,\mu}\left[\prod\limits_{j=1}^{r_1}Y^{m_j}_{n,k_{j_{1}},\cdots,k_{j_{m}},0}|\bm \sigma\right]\hspace{-0.01in}\mathbb{E}_{\bm \lambda,\mu}\left[\prod\limits_{j=r_1+1}^{r}Y^{m_j}_{n,k_{j_{1}},\cdots,k_{j_{m}},l_{j}}|\bm \sigma\right]\Bigg]\\ 
 &= \sum\limits_{\omega_{11},\cdots,\omega_{1m_1}}\cdots\sum\limits_{\omega_{r1},\cdots,\omega_{rm_r}}\frac{1}{n^{\sum_{j=r_1+1}^{r}\ell_jm_j}} \times\mathbb E_{\bm \sigma}\Bigg\{\mathbb{E}_{\bm \lambda,\mu}\Bigg[\prod_{k=1}^{m}\prod_{t=1}^{r}\prod_{v=1}^{m_t}\prod_{e_{tk} \in E_{\omega_{tv,1_k}}}A^{(k)}_{e_{tk}}|\bm \sigma\Bigg]\\
 &\qquad\qquad\qquad\times\mathbb{E}_{\bm \lambda,\mu}\Bigg[\prod_{t=r_1+1}^{r}\prod_{v=1}^{m_t}\prod_{e_{tk} \in E_{\omega_{tv,1_k}}}A^{(k)}_{e_{tk}}|\bm \sigma\Bigg] \times\mathbb{E}_{\bm \lambda,\mu}\Bigg[\prod_{t=r_1+1}^{r}\prod_{v=1}^{m_\ell}\prod\limits_{e_{t\ell} \in E_{\omega_{tv},2}}B_{e_{t\ell}}|\bm \sigma\Bigg]\Bigg\}.      
\end{align}
Again, as the dominant contributions come from the non-overlapping cycles by arguments of the previous part, the differences between the conditional expectations is $o(1)$. Hence, by Dominated Convergence Theorem, the second part of the lemma follows.\qed
\section{Proof of Theorem \ref{thm:contiguity}}\label{appendix:proofofthm3}
Consider the random variables,
\begin{align}
W^{(M)}&:=\prod_{K=1}^{M}\Bigg\{\prod_{k_1+\cdots+k_m=K}(1+\delta_{k_1,\cdots,k_m})^{\nu_{(k_1,\cdots,k_m)}}\exp\left(-\sum_{k_1+\cdots+k_m=K}\lambda_{k_1,\cdots,k_m}\delta_{k_1,\cdots,k_m}\right)\\
&\hspace{1.3in}\prod_{\substack{k_1+\cdots+k_m+\ell=K\\\ell \neq 0}}\exp\left(\frac{2\mu_{k_1,\cdots,k_m}Z_{(k_1,\cdots,k_m)}-\mu^2_{k_1,\cdots,k_m}}{2\sigma^2_{k_1,\cdots,k_m}}\right)\Bigg\},
\end{align}
for $M>0$ and,
\begin{align}
\label{eq:W_infty}
W^{(\infty)}&:=\prod_{K=1}^{\infty}\Bigg\{\prod_{k_1+\cdots+k_m=K}(1+\delta_{k_1,\cdots,k_m})^{\nu_{(k_1,\cdots,k_m)}}\exp\left(-\sum_{k_1+\cdots+k_m=K}\lambda_{k_1,\cdots,k_m}\delta_{k_1,\cdots,k_m}\right)\\
&\hspace{1.3in}\prod_{\substack{k_1+\cdots+k_m+\ell=K\\\ell \neq 0}}\exp\left(\frac{2\mu_{k_1,\cdots,k_m}Z_{(k_1,\cdots,k_m)}-\mu^2_{k_1,\cdots,k_m}}{2\sigma^2_{k_1,\cdots,k_m}}\right)\Bigg\}.
\end{align}
We prove it in two steps.
\paragraph{Step 1.}Let us observe that,
\begin{align}
\mathbb{E}_{\mathbb P_n}\Bigg[\prod_{k_1+\cdots+k_m=K}(1+\delta_{k_1,\cdots,k_m})^{\nu_{(k_1,\cdots,k_m)}}\exp\left(-\sum_{k_1+\cdots+k_m=K}\lambda_{k_1,\cdots,k_m}\delta_{k_1,\cdots,k_m}\right)\\
\prod_{\substack{k_1+\cdots+k_m+\ell=K\\\ell \neq 0}}\exp\left(\frac{2\mu_{k_1,\cdots,k_m}Z_{(k_1,\cdots,k_m)}-\mu^2_{k_1,\cdots,k_m}}{2\sigma^2_{k_1,\cdots,k_m}}\right)\Bigg]=1 \quad \forall i \in \mathbb{N}.
\end{align}
This implies $\{W^{(k)}\}$ is a martingale with respect to the natural filtration. Let us also observe that,
\begin{align}
\mathbb{E}_{\mathbb P_n}[(W^{(k)})^2]&=\prod_{K=1}^{\infty}\Bigg\{\prod_{k_1+\cdots+k_m=K}\exp\left(\lambda_{k_1,\cdots,k_m}\delta^2_{k_1,\cdots,k_m}+\lambda_{k_1,\cdots,k_m}\delta_{k_1,\cdots,k_m}\right)\\
&\hspace{1.4in}\prod_{\substack{k_1+\cdots+k_m+\ell=K\\\ell \neq 0}}\exp\left(\frac{\mu^2_{k_1,\cdots,k_m}}{\sigma^2_{k_1,\cdots,k_m}}\right)\Bigg\}<\infty \quad \forall k.
\end{align}
This implies $\{W^{(k)}\}$ is an $L_2$ bounded martingale. Hence there exists $W^{\infty}$ as defined in \eqref{eq:W_infty} in $L_2$ such that $W^{(k)}\overset{a.s/L_2}{\longrightarrow}W^{(\infty)}$. Now let us observe that,
\[
\log(W^{(k)})=Z_k+\widetilde{W}_k,
\]
where $Z_k$ come from Gaussian distribution. Hence,
\[
\log(W^{(\infty)})=Z+\widetilde{W},
\]
where $Z$ is Gaussian, implying $\mathbb P_n(W^{(\infty)}>0)=1$. That $W^{(\infty)}\in L_2$ immediately implies $\mathbb P_n(W^{(\infty)}<\infty)=1$. Hence, the term $W^{(\infty)}$ is well-defined.
\paragraph{Step 2.} It is enough to show that 
\[
T_n:=\frac{d\mathbb Q_n}{d\mathbb P_n} \overset{d}{\rightarrow} W^{(\infty)}.
\]
By the proof of Theorem \ref{thm:detection_threshold}, if $\lambda^2_1+\cdots+\lambda^2_m+\mu^2/\gamma\le1$, 
\[
\limsup\limits_{n \rightarrow \infty}\mathbb E_{\mathbb P_n}[T^2_n] < \infty.
\]
Hence the sequence $\{T_n\}$ is tight. By Prokhorov's Theorem, there exists a subsequence $\{n_k\}_{k=1}^{\infty}$ such that $T_{n_k}$ converges in distribution to $W(\{n_k\})$. We shall show that the limit $W(\{n_k\})$ does not depend on the subsequence and $W(\{n_k\})\overset{d}{=}W^{(\infty)}$. For $\varepsilon>0$, let us choose $M$ large such that,
\begin{align}
\label{eq:l_2diff}
   \Bigg|\prod_{K=1}^{M}\Bigg\{\prod_{k_1+\cdots+k_m=K}(1+\delta_{k_1,\cdots,k_m})^{\nu_{(k_1,\cdots,k_m)}}\exp\left(-\sum_{k_1+\cdots+k_m=K}\lambda_{k_1,\cdots,k_m}\delta_{k_1,\cdots,k_m}\right)\\
\prod_{\substack{k_1+\cdots+k_m+\ell=K\\\ell \neq 0}}\exp\left(\frac{2\mu_{k_1,\cdots,k_m}Z_{(k_1,\cdots,k_m)}-\mu^2_{k_1,\cdots,k_m}}{2\sigma^2_{k_1,\cdots,k_m}}\right)\Bigg\}-\\\prod_{K=1}^{\infty}\Bigg\{\prod_{k_1+\cdots+k_m=K}(1+\delta_{k_1,\cdots,k_m})^{\nu_{(k_1,\cdots,k_m)}}\exp\left(-\sum_{k_1+\cdots+k_m=K}\lambda_{k_1,\cdots,k_m}\delta_{k_1,\cdots,k_m}\right)\\
\hspace{1in}\prod_{\substack{k_1+\cdots+k_m+\ell=K\\\ell \neq 0}}\exp\left(\frac{2\mu_{k_1,\cdots,k_m}Z_{(k_1,\cdots,k_m)}-\mu^2_{k_1,\cdots,k_m}}{2\sigma^2_{k_1,\cdots,k_m}}\right)\Bigg\}\Bigg|<\varepsilon.
\end{align}
For this $M$, consider the joint distribution of $(T_{n_k},\{Y_{n_k,k_1,\cdots,k_m,\ell}\}$ where $(k_1,\cdots,k_m,\ell)\in\mathcal{K}$ for $\mathcal{K}:=\{(k_1,\cdots,k_m,\ell):k_1+\cdots+k_m+\ell\in\{1,\cdots,M\}\}$. By Theorem \ref{thm:cycle-stat-asymp-dist}, these random variables are tight with respect to $P_{n_k}$. So, there exists a further subsequence $\{n_{k_\ell}\}$ such that under $\mathbb P_{n_{k_\ell}}$,
\[
(T_{n_{k_\ell}},\{Y_{n_k,k_1,\cdots,k_m,\ell}\}_{(k_1,\cdots,k_m,\ell)\in \mathcal{K}}) \overset{d}{\rightarrow} (W(\{n_k\}),\bm \nu_{0,M}),
\]
where the set of random variables $\bm \nu_{0,M}$ is defined by Theorem \ref{thm:cycle-stat-asymp-dist}. Since, 
\[
\limsup\limits_{n \rightarrow \infty}\mathbb{E}_{\mathbb P_n}[T^2_n]<\infty,
\]
the sequence $\{Y_{n_{k_\ell}}\}$ is uniformly integrable, implying,
\[
\mathbb{E}[W(\{n_k\})]=\lim\limits_{\ell \rightarrow \infty}\mathbb{E}_{\mathbb P_{n_{k_\ell}}}[T_{n_{k_\ell}}]=1.
\]
Now for any bounded positive continuous function $f$, by Fatou's Lemma,
\begin{equation}
\label{eq:fatou_1}
\liminf\limits_{\ell \rightarrow \infty}\mathbb{E}_{\mathbb{P}_{n_{k_\ell}}}\left[f(\{Y_{n_k,k_1,\cdots,k_m,\ell}\}_{(k_1,\cdots,k_m,\ell)\in \mathcal{K}}))T_{n_{k_\ell}}\right]\ge \mathbb{E}\left[f(\bm \nu_{M,0}))W(\{n_k\})\right]
\end{equation}
For any constant $\xi$, by uniform integrability, $\xi=\xi\mathbb E[T_{n_{k_\ell}}]\rightarrow \xi\mathbb E[W(\{n_k\})]=\xi$. So, \eqref{eq:fatou_1} holds for any bounded continuous function. Replacing $f$ by $-f$, we get for any bounded continuous function,
\begin{equation}
\label{eq:fatou_2}
\lim\limits_{\ell \rightarrow \infty}\mathbb{E}_{\mathbb{P}_{n_{k_\ell}}}\left[f(\{Y_{n_k,k_1,\cdots,k_m,\ell}\}_{(k_1,\cdots,k_m,\ell)\in \mathcal{K}}))T_{n_{k_\ell}}\right]= \mathbb{E}\left[f(\bm \nu_{0,M}))W(\{n_k\})\right]
\end{equation}
By Theorem \ref{thm:cycle-stat-asymp-dist}, 
\begin{align}
\int f(\{Y_{n_k,k_1,\cdots,k_m,\ell}\}_{(k_1,\cdots,k_m,\ell)\in \mathcal{K}}))T_{n_{k_\ell}}d\mathbb{P}_{n_{k_\ell}}
&=\int f(\{Y_{n_k,k_1,\cdots,k_m,\ell}\}_{(k_1,\cdots,k_m,\ell)\in \mathcal{K}}))d\mathbb{Q}_{n_{k_\ell}}\\
&\rightarrow \int f(\bm \nu_{M,1})d\mathbb{Q},
\end{align}
where the set of random variables $\bm \nu_{M,1}$ is defined by Theorem \ref{thm:cycle-stat-asymp-dist} and $Q$ is the measure induced by $\bm \nu_{M,1}$. By definition of $W^{(M)}$,
\[
\int f(\bm \nu_{M,1})d\mathbb{Q}=\mathbb E[f(\bm \nu_{M,0})W^{(M)}],
\]
for any bounded continuous function $f$, and so $\int_AdQ=\mathbb{E}[1_AW^{(m)}]$ for any $A \in \sigma(\bm \nu_{M,0})$. This implies,
\[
W^{(m)}=\mathbb{E}[W(\{n_k\})|\sigma(\bm \nu_{M,0})].
\]
From Fatou's Lemma and definition of $W(\{n_k\})$,
\[
\mathbb E[W(\{n_k\})^2] \le \liminf\limits_{n \rightarrow \infty}\mathbb{E}_{\mathbb P_n}[T^2_n].
\]
As a consequence, by \eqref{eq:l_2diff},
\[
0 \le \mathbb{E}|W(\{n_k\})-W^{(M)}|^2=\mathbb{E}[W^2(\{n_k\})]-\mathbb{E}[(W^{(M)})^2]<\varepsilon.
\]
So, $L_2(W(\{n_k\}),W^{(M)})<\sqrt{\varepsilon}$. This implies $W^{(n)}\overset{L_2/d}{\rightarrow}W(\{n_k\})$ as $n \rightarrow \infty$. But we have also shown $W^{(n)}\overset{L_2/d}{\rightarrow}W^{(\infty)}$. This implies $W(\{n_k\})\overset{d}{=}W^{(\infty)}$, which completes the proof of the weak convergence statement.

The proof of the final statement follows exactly as the proof of equation (6) in \cite{Banerjee2018AsymptoticNA}.
\qed
\section{Proof of Theorem~\ref{thm:weak_recovery_threshold}}\label{appendix:proofofthm4}
Weak recovery above the threshold follows via a proposed estimator defined using self-avoiding walks in Section~\ref{subsec: weak_recovery}: it only remains to prove Lemma~\ref{lem:psd}, which is done in the next section. In this section, we prove the impossibility of reconstruction under the threshold. We start with a proposition.

\begin{prop}
\label{prop: conditional_TV_convergence}
Let $\sum_i \lambda_i + \mu^2/\gamma < 1$. Then, for any fixed $r$ with $1 \leq r \leq n$ and any $(\sigma_1, \dots, \sigma_r), (\tau_1, \dots, \tau_r) \in \{\pm 1\}^r$, we have
\[
\|\Pro_{\bm \lambda, \mu}(\cdot|\sigma_{1:r}) - \Pro_{\bm \lambda, \mu}(\cdot|\tau_{1:r})\|_{\mathrm{TV}} \to 0.
\]
\end{prop}
\begin{proof}
This is an adaptation of Proposition 3 of \cite{lu2020contextual} to our setup and the proof proceeds along the same lines. Since $r$ is fixed, let $L_{\bm \sigma_{1:r},n}:= \E_{\bm \sigma_{-r}, \bm u}[\Pro_{\bm \lambda, \mu}(A_1, \dots, A_m, B | \bm \sigma, \bm u) /\Pro_{0,0} (A_1, \dots, A_m, B)]$ and $L_{\bm \tau_{1:r}, n}$ (analogously defined) denote the conditional log-likelihoods, and let $\mathcal{S} = \{\|\bm u\| \leq 2\sqrt{p} \}$. Then, $\Pro(\mathcal S) \to 1$. We define the following truncations of the conditional log-likelihood to $\mathcal S$:
\[
\tilde L_{\bm \sigma_{1:r}, n} := \frac{\E_{\bm \sigma_{-r}, \bm u} \left [\Pro_{\bm \lambda, \mu}(A_1, \dots, A_m, B | \bm \sigma, \bm u)\bm 1_{[\bm u \in \mathcal{S}]} \right ]}{\Pro_{0,0} (A_1, \dots, A_m, B)},
\]
\[
\tilde L_{\bm \tau_{1:r}, n} := \frac{\E_{\bm \tau_{-r}, \bm u} \left [\Pro_{\bm \lambda, \mu}(A_1, \dots, A_m, B | \bm \sigma, \bm u)\bm 1_{[\bm u \in \mathcal{S}]} \right ]}{\Pro_{0,0} (A_1, \dots, A_m, B)}.
\]
Denote by $Q_{\bm \sigma_{1:r}, n}$ and $Q_{\bm \tau_{1:r}, n}$ the probability measures induced by $\tilde L_{\bm \sigma_{1:r}, n}$ and $\tilde L_{\bm \tau{1:r}, n}$, respectively, after normalization. Now, as long as, $\E_{0,0} \left [ L_{\bm \sigma_{1:r}, n}^2 \right], \E_{0,0} \left [ L_{\bm \tau_{1:r}, n}^2 \right] < \infty$, by Cauchy-Schwarz inequality, we have $\|\Pro_{\bm \lambda, \mu}(\cdot | \sigma_{1:r}) - Q_{\bm \sigma_{1:r}, n}(\cdot) \|_{\mathrm{TV}} \to 0$ and $\|\Pro_{\bm \lambda, \mu}(\cdot | \tau_{1:r}) - Q_{\bm \tau_{1:r}, n}(\cdot) \|_{\mathrm{TV}} \to 0$. From the proof of Theorem 2.1, this holds under $\sum_i \lambda_i + \mu^2/\gamma < 1$. Thus, it remains to show $\|Q_{\bm \tau_{1:r}, n}(\cdot) - Q_{\bm \tau_{1:r}, n}(\cdot)\|_{\mathrm{TV}} \to 0$, which amounts to showing $\frac{1}{\Pro(\mathcal S)}\E_{0,0}\left[ \big | \tilde L_{\bm \sigma_{1:r}, n} - \tilde L_{\bm \sigma_{1:r}, n} \big | \right] \to 0$. Again, $\E_{0,0}\left[ \big | \tilde L_{\bm \sigma_{1:r}, n} - \tilde L_{\bm \sigma_{1:r}, n} \big | \right] \leq \E_{0,0}\left[ \big ( \tilde L_{\bm \sigma_{1:r}, n} - \tilde L_{\bm \sigma_{1:r}, n} \big )^2 \right]$, so that it is enough to show
\[
\frac{1}{\Pro(S)}\E_{0,0}\left[ \big ( \tilde L_{\bm \sigma_{1:r}, n} - \tilde L_{\bm \sigma_{1:r}, n} \big )^2 \right] \to 0.
\]
Expanding the square, we get
\[
\E_{0,0}\left[ \big ( \tilde L_{\bm \sigma_{1:r}, n} - \tilde L_{\bm \sigma_{1:r}, n} \big )^2 \right] = \E_{0,0} \tilde L_{\bm \sigma_{1:r}, n}^2 + \E_{0,0} \tilde L_{\bm \tau_{1:r}, n}^2 - 2 \E_{0,0} \tilde L_{\bm \sigma_{1:r}, n} \tilde L_{\bm \tau_{1:r}, n},
\]
where 
\begin{align*}
    \tilde L_{\bm \sigma_{1:r}, n}^2 &= \left ( \E_{\bm \sigma_{-r}, \bm u} \left [  \frac{d\Pro_{\bm \lambda, \mu}(A_1, \dots, A_m, B | \bm \sigma, \bm u)}{d \Pro_{0,0} (A_1, \dots, A_m, B)} \mathbbm{1}\{\bm u \in \mathcal{S}\} \right ] \right )^2\\
    &= \E_{\bm \sigma_{-r}, \bm \tau_{-r}, \bm u, \bm v} \left [\frac{d\Pro_{\bm \lambda, \mu}(A_1, \dots, A_m, B | \bm \sigma, \bm u)}{d \Pro_{0,0} (A_1, \dots, A_m, B)}  \frac{d\Pro_{\bm \lambda, \mu}(A_1, \dots, A_m, B | \bm \sigma_{1:r}, \bm \tau_{-r}, \bm v )}{d \Pro_{0,0} (A_1, \dots, A_m, B)} \mathbbm{1}\{\bm u , v \in \mathcal{S}\}\right ].
\end{align*}
Using similar expansions for the other two terms it is enough to show that
\[
 \E_{0,0} \left [ \frac{d\Pro_{\bm \lambda, \mu}(A_1, \dots, A_m, B | \bm \sigma, \bm u)}{d \Pro_{0,0} (A_1, \dots, A_m, B)}  \frac{d\Pro_{\bm \lambda, \mu}(A_1, \dots, A_m, B | \bm \tau, \bm v)}{d \Pro_{0,0} (A_1, \dots, A_m, B)} \mathbbm{1}\{\bm u , v \in \mathcal{S}\}\right ]
\]
converges to a limit. This was shown in the proof of Theorem~\ref{thm:detection_threshold}, which completes the proof.
\end{proof}

We next translate the convergence of conditional data distributions in Proposition \ref{prop: conditional_TV_convergence} to a convergence of posterior distributions of the cluster assignments $\sigma_i$.

\begin{prop}
\label{prop: posterior_TV_convergence}
Let $\sum_i \lambda_i + \mu^2/\gamma < 1$, $S \subset [n]$, $|S| = r$ where $r$ is finite and fixed, $u \in [n]$ any fixed index such that $u \not\in S$. Then, as $n \to \infty$,
\[
\E_{\bm \lambda, \mu} \left [ \| \Pro_{\bm \lambda, \mu} ( \sigma_u | A_1, \dots, A_m, B, \sigma_S) - \Pro_{\bm 0, 0}(\sigma_u)\|_{\mathrm{TV}} \right ] \to 0.
\]
\end{prop}
\begin{proof}
\begin{align*}
    \E_{\bm \lambda, \mu} &\left [ \| \Pro_{\bm \lambda, \mu} ( \sigma_u | A_1, \dots, A_m, B, \sigma_S) - \Pro_{\bm 0, 0}(\sigma_u)\|_{\mathrm{TV}} \right ] \\
    &= \Pro_{\bm 0, 0}(\sigma_u) \E_{\bm \lambda, \mu} \left [ \left\| \frac{\Pro_{\bm \lambda, \mu} ( \sigma_u | A_1, \dots, A_m, B, \sigma_S)}{\Pro_{\bm 0, 0}(\sigma_u)} - 1 \right \|_{\mathrm{TV}} \right ]\\
    &= \Pro_{\bm 0, 0}(\sigma_u) \E_{\bm \lambda, \mu} \left [ \left\| \frac{\Pro_{\bm \lambda, \mu} ( \sigma_u, A_1, \dots, A_m, B, \sigma_S)}{\Pro_{\bm 0, 0}(\sigma_u) \Pro_{\bm \lambda, \mu} ( A_1, \dots, A_m, B, \sigma_S)} - 1 \right \|_{\mathrm{TV}} \right ]\\
    &= \Pro_{\bm 0, 0}(\sigma_u) \E_{\bm \lambda, \mu} \left [ \left\| \frac{\Pro_{\bm 0, 0} ( \sigma_u, \sigma_S) \Pro_{\bm \lambda, \mu} (A_1, \dots, A_m, B | \sigma_u, \sigma_S) }{\Pro_{\bm 0, 0}(\sigma_u, \sigma_S) \Pro_{\bm \lambda, \mu} ( A_1, \dots, A_m, B | \sigma_S)} - 1 \right \|_{\mathrm{TV}} \right ]\\
    &= \frac{1}{2} \E_{\bm \lambda, \mu} \left [ \E_{\bm \lambda, \mu} \left [ \left\| \frac{ \Pro_{\bm \lambda, \mu} (A_1, \dots, A_m, B | \sigma_u, \sigma_S) }{ \Pro_{\bm \lambda, \mu} ( A_1, \dots, A_m, B | \sigma_S)} - 1 \right \|_{\mathrm{TV}} \middle | \sigma_S \right ] \right ]\\
    &= \frac{1}{2} \E_{\bm \lambda, \mu} \Bigg [ \Pro_{\bm \lambda, \mu} ( A_1, \dots, A_m, B | \sigma_S) \times \\
    &\qquad\qquad\quad\E_{\bm \lambda, \mu} \left [ \left\| \Pro_{\bm \lambda, \mu} (A_1, \dots, A_m, B | \sigma_u, \sigma_S) - \Pro_{\bm \lambda, \mu} ( A_1, \dots, A_m, B | \sigma_S) \right \|_{\mathrm{TV}} \middle | \sigma_S \right ] \Bigg ]\\
    &= \frac{1}{4} \E_{\bm \lambda, \mu} \Bigg [ \Pro_{\bm \lambda, \mu} ( A_1, \dots, A_m, B | \sigma_S) \times \\
    &\qquad\qquad\quad \left\| \Pro_{\bm \lambda, \mu} (A_1, \dots, A_m, B | \sigma_u = +1, \sigma_S) - \Pro_{\bm \lambda, \mu} ( A_1, \dots, A_m, B | \sigma_u = -1, \sigma_S) \right \|_{\mathrm{TV}} \Bigg ]\\
\end{align*}
The term inside the expectation converges to 0 in probability by Proposition~\ref{prop: conditional_TV_convergence}, so that the outer expectation also converges to zero by the Dominated Convergence Theorem.
\end{proof}

It remains to prove the impossibility of weak recovery using Propositions \ref{prop: conditional_TV_convergence} and \ref{prop: posterior_TV_convergence}. This follows using the technique of the proof of Theorem 2.3 of \cite{banerjee2018contiguity}.
\section{Proof of Lemma~\ref{lem:psd}}
It is enough to show, for a constant delta as specified in Lemma~\ref{lem:psd} and for each $i_1, i_2 \in [n]$, that
\[
\E \left [ \widehat\Sigma_{i_1 i_2} \cdot \sigma_{i_1} \sigma_{i_2} \right ] \geq \delta \sqrt{\E \Big [ \widehat\Sigma_{i_1 i_2}^2 \Big]}.
\]
For the left hand side, we have
\begin{align*}
    \E \left [ \widehat\Sigma_{i_1 i_2} \cdot \sigma_{i_1} \sigma_{i_2} \right ] &= \frac{1}{|\mathcal{W}(i_1, i_2, k_1, \dots, k_m, \ell)|}\sum_{\alpha \in \mathcal{W}(i_1, i_2, k_1, \dots, k_m, \ell)} \E\left[ p_\alpha \sigma_{i_1} \sigma_{i_2} \right]\\
    &= \frac{1}{|\mathcal{W}(i_1, i_2, k_1, \dots, k_m, \ell)|}\sum_{\alpha \in \mathcal{W}(i_1, i_2, k_1, \dots, k_m, \ell)} \E\left[ \E[ p_\alpha | \bm{\sigma} ] \sigma_{i_1} \sigma_{i_2} \right]\\
    &= 1,
\end{align*}
while, for the right hand side,
\[
\E \Big [ \widehat\Sigma_{i_1 i_2}^2 \Big] = \frac{1}{|\mathcal{W}(i_1, i_2, k_1, \dots, k_m, \ell)|^2}\sum_{\alpha, \beta \in \mathcal{W}(i_1, i_2, k_1, \dots, k_m, l)} \E[ p_\alpha p_\beta ]
\]
For paths $\alpha, \beta \in \mathcal{W}(i_1, i_2, k_1, \dots, k_m, \ell)$ with no wedges in common, we have
\[
\E[p_\alpha p_\beta] = \E[\E[p_\alpha p_\beta | \bm{\sigma}]] = 1.
\]
Next, we turn to controlling $\E[p_\alpha p_\beta]$ for paths $\alpha$ and $\beta$ with common wedges. Let $\alpha$ and $\beta$ coincide on $s_i$ type $\bm{A}_{i}$ wedges, for $i = 1, \dots, m$, and $t$ type $\bm{B}$ wedges. For these terms, we have
\[
\E[p_\alpha p_\beta] = O(1) \left [ \Big ( \frac{n}{\lambda_1^2} \Big )^{s_1} \dots \Big ( \frac{n}{\lambda_m^2} \Big )^{s_m} \Big ( \frac{np}{\mu^2/\gamma} \Big)^t \right ].
\]
Further, for fixed $s_1, \dots, s_m, t$, the minimum possible number of common factor and variable nodes is between $\alpha$ and $\beta$ as described above is $t + \sum_i s_i$ variable nodes, $s_i$ type $\bm{A}_{i}$ factor nodes and $b$ type $\bm{B}$ factor nodes. This happens when two paths have the first $\sum_i s_i$ wedges common and of type $\mathbb{A}$, and the last $t$ wedges common and of type $\bm{B}$. Set $\bm{s} = (s_1, \dots, s_m)$ and $\bm{k} = (k_1, \dots, k_m)$. The number of path pairs having exactly this number of common nodes is then $n^{2(k - 1) - \sum_i s_i - t} p^{2\ell - t} g( \bm s, \bm k)$, where $g( \bm s, \bm k)$ is the number of valid choices of the type $\mathbb A$ factor nodes for a pair of paths $\alpha, \beta$ as described above that are determined except for this choice. In fact, it is not difficult to observe that
\[
g(\bm s, \bm k) \le \frac{(\sum_{i=1}^{m} s_i)!}{\prod_{i=1}^{m} (s_i)!} \left ( \frac{\Big(\sum_{i=1}^{m} (k_i - s_i)\Big)!}{ \prod_{i=1}^{m} (k_i - s_i)!} \right )^2.
\]
For terms with the minimum number of vertices common, the contribution to $\sum \E[p_\alpha p_\beta]$ is then
\[
O(1) n^{2(k - \ell - 1)-\sum_{i=1}^{m} s_i - t} p^{2\ell - t}  \prod_{i=1}^{m} \left ( \frac{n}{\lambda_i^2} \right )^{s_i} \left ( \frac{n}{\mu} \right )^t g( \bm s, \bm k) = O(1) n^{2(k - 1)} p^{2\ell} \prod_{i=1}^{m} \lambda_i^{-2s_i} \left ( \frac{\mu^2}{\gamma} \right)^{t} g( \bm s, \bm k),
\]

A similar calculation will show that the leading order term (in $n$) in the number of path pairs as described above is contributed by the pairs with minimum possible number of vertices common as the number of paths reduces when the paths intersect in a larger number of vertices. Then, summing over all possible number of intersecting wedges, we have, since $g(\bm s, \bm k) \leq 1$,
\begin{align*}
&\sum_{\alpha, \beta \in \mathcal{W}(i_1, i_2, k_1, \dots, k_m, \ell)} \E[p_\alpha p_\beta] \\
&\qquad\leq O(1) n^{2(k - 1)} p^{2\ell}  \sum_{i=1}^{m} \sum_{0 \leq  s_i \leq  k_i} \sum_{0 \leq t \leq \ell}  \left ( \prod_{i=1}^{m} \lambda_i^{-2s_i} \left ( \frac{\mu^2}{\gamma} \right)^{t} g( \bm s, \bm k) \right ).\\
&\qquad\leq O(1) n^{2(k-1)}p^{2\ell}\prod_{i=1}^{m} \lambda_i^{-2k_i} \left ( \frac{\mu^2}{\gamma} \right )^{-\ell}.
\end{align*}
Finally, to establish Lemma~\ref{lem:psd}, it remains to show that 
\[|\mathcal{W}(i_1, i_2, k_1, \dots, k_m,\ell)|^2 = \frac{k!}{k_1! \dots k_m! \ell!} n^{2(k-1)}p^{2l} \geq n^{2(k-1)}p^{2\ell} \prod_{i=1}^{m}\lambda_i^{-2k_i} \left ( \frac{\mu^2}{\gamma} \right )^{-\ell}
\]
for some choice of $\bm k$ and $\ell$. Choose
\[
\frac{k_i}{k} = \frac{\lambda^2_i}{\sum_{i=1}^{m} \lambda^2_i + \mu^2/\gamma},\quad  \frac{\ell}{k} = \frac{\mu^2/\gamma}{\sum_{i=1}^{m} \lambda^2_i + \mu^2/\gamma}.
\]
Since $\sum_{i=1}^{m} \lambda_i^2 + \mu^2/\gamma > 1$, we have $\lambda_i^2 > k_i / k$ and $\mu^2/\gamma > \ell/k$. Then,
\[
\frac{k!}{k_1! \dots k_m! \ell!} \geq \exp \left ( -2l \log \frac{\ell}{k} - 2 \sum_i \log \frac{k_i}{k} + o(1) \right ) \gtrsim \exp \left ( -\ell \log \frac{\mu^2}{\gamma} - k_i \log \lambda^2 \right ) = \prod_i \lambda^{-2k_i} \left ( \frac{\mu^2}{\gamma} \right )^{-\ell}.
\]
This completes the proof.
\section{Proofs of results in Section \ref{Belief_Propagation}}
\revsag{Let us compute the \emph{Belief Propagation Messages} in the factor graph described by Figure \ref{fig:factor_graph_new}.
\begin{figure}
    \centering
    \includegraphics[scale = 0.8]{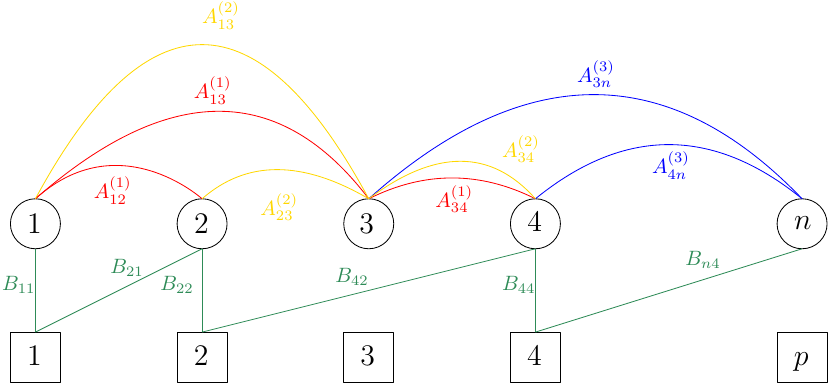}
    \caption{Combined Multi-Layered Factor graph for running the Belief Propagation Algorithm.}
    \label{fig:factor_graph_new}
\end{figure}
In this algorithm we compute messages $\nu^{t}_{i \rightarrow j;k}(\sigma_i)$ which are the marginal distributions of the variable $\sigma_i$ in the posterior distributions $\mathbb P_{i \rightarrow j;k}(\bm \sigma,\bm v|\bm G_1,\ldots,\bm G_m,\bm B / \sqrt{p})$ when all the connections to vertex $\sigma_j$ corresponding to the $k$-th network are removed from the factor graph. Similarly, we compute the messages $\nu^{t}_{i \rightarrow q}(\sigma_i)$ which are the marginal distributions of $\sigma_i$ in $\mathbb P_{i \rightarrow q}(\bm \sigma,\bm u|\bm G_1,\ldots,\bm G_m,\bm B / \sqrt{p})$ when the vertex $u_q$ is removed from the factor graph with all associated connections and the messages  $\nu^{t}_{q \rightarrow i}(u_q)$ which are the marginal distributions of $u_q$ in $\mathbb P_{q \rightarrow i}(\bm \sigma,\bm u|\bm G_1,\ldots,\bm G_m,\bm B / \sqrt{p})$ when the vertex $\sigma_i$ is removed from the factor graph with all associated connections. Let us denote $\widetilde{\bm B} = \bm B / \sqrt{p}$ and $\widetilde{u}_q=u_q/\sqrt{p}$. Then, the messages $\nu^{t}_{i \rightarrow j;k}(\sigma_i)$ from the variable nodes $i$ to $j$ along the edge $(i,j)$ in the graph $\bm G_k$, for $i,j \in [n]$, $k \in [m]$, are given by,
\begin{align*}
\nu^{t+1}_{i \rightarrow j;k}(\sigma_i) &\approx \prod_{q \in [p]}\mathbb E^t_{q \rightarrow i}\left[\exp\left(\sqrt{\frac{\mu p^2}{n}}\widetilde{B}_{qi}\widetilde{u}_q\sigma_i\right)\right]\prod_{\ell \neq k, \ell=1}^{m}\Bigg\{\prod_{c \in \partial_\ell i}\mathbb{E}^t_{c \rightarrow i;\ell}\left[\frac{d_\ell+\lambda_\ell\sqrt{d_\ell}\sigma_i\sigma_c}{n}\right]\nonumber\\
& \quad\quad\quad\prod_{c \in (\partial_\ell i)^c}\mathbb{E}^t_{c \rightarrow i;\ell}\left[1-\frac{d_\ell+\lambda_\ell\sqrt{d_\ell}\sigma_i\sigma_c}{n}\right]\Bigg\}\nonumber\\
& \quad\quad\Bigg\{\prod_{c \in \partial_k i\setminus\{j\}}\mathbb{E}^t_{c \rightarrow i;k}\left[\frac{d_k+\lambda_k\sqrt{d_k}\sigma_i\sigma_c}{n}\right]\prod_{c \in (\partial_k i)^c\setminus\{j\}}\mathbb{E}^t_{c \rightarrow i;k}\left[1-\frac{d_k+\lambda_k\sqrt{d_k}\sigma_i\sigma_c}{n}\right]\Bigg\}\nonumber.
\end{align*}
Here, the marginal distributions are provided modulo a constant of proportionality. A similar convention is followed throughout the rest of the proof. Further, $\mathbb{E}^t_{c \rightarrow i;\ell}$, and $\mathbb E^t_{q \rightarrow i}$ represent the expectations of the corresponding variables with respect to the measures $\nu^{t}_{i \rightarrow j;k}$ and $\nu^{t}_{i \rightarrow q}$. Similar notations are also adopted throughout the rest of the proof. Next, the messages from the variable node $i$ to factor node $q$ for $i \in [n]$ and $q \in [p]$ are given by,
\begin{align*}
\nu^{t+1}_{i \rightarrow q}(\sigma_i) &\approx \prod_{r \in [p]\setminus\{q\}}\mathbb E^t_{r \rightarrow i}\left[\exp\left(\sqrt{\frac{\mu p^2}{n}}\widetilde{B}_{ri}\widetilde{u}_r\sigma_i\right)\right]\prod_{\ell=1}^{m}\Bigg\{\prod_{k \in \partial_\ell i}\mathbb{E}^t_{k \rightarrow i;\ell}\left[\frac{d_\ell+\lambda_\ell\sqrt{d_\ell}\sigma_i\sigma_k}{n}\right]\nonumber\\
& \quad\quad\quad\prod_{k \in (\partial_\ell i)^c}\mathbb{E}^t_{k \rightarrow i;\ell}\left[1-\frac{d_\ell+\lambda_\ell\sqrt{d_\ell}\sigma_i\sigma_k}{n}\right]\Bigg\}.\nonumber
\end{align*}
Finally, the messages from factor nodes $q \in [p]$ to $i \in [n]$ are given by,
\begin{align*}
\nu^{t+1}_{q \rightarrow i}(\tilde{u}_q) &\approx \exp\left(-\frac{p(1+\mu)\widetilde{u}^2_q}{2}\right)\prod_{j \neq i}\mathbb{E}^t_{j \rightarrow q}\left[\exp\left(\sqrt{\frac{\mu p^2}{n}}\widetilde{B}_{qj}\widetilde{u}_q\sigma_j\right)\right].\nonumber
\end{align*}
Now, to characterize the distributions $\nu^{t+1}_{i \rightarrow j;k}$'s and $\nu^{t+1}_{i \rightarrow q}$'s we define the following parameters characterizing the log odds ratio,
\begin{equation}
\label{eq:f_new}
\eta^t_{i \rightarrow j;k}:=\frac{1}{2}\log\frac{\nu^{t+1}_{i \rightarrow j;k}(+1)}{\nu^{t+1}_{i \rightarrow j;k}(-1)},
\end{equation}
and
\begin{equation}
\label{eq:rho_new}
\eta^t_{i \rightarrow q}:=\frac{1}{2}\log\frac{\nu^{t+1}_{i \rightarrow q}(+1)}{\nu^{t+1}_{i \rightarrow q}(-1)}.
\end{equation}
For the messages $\nu^t_{q \rightarrow i}$, we use the \emph{Gaussian ansatz}, i.e., 
\[
\nu^t_{q \rightarrow i}:=\mathsf{N}\left(\frac{m^t_{q \rightarrow i}}{\sqrt{p}},\frac{\tau^t_{q \rightarrow i}}{p}\right).
\]
Let us observe that,
\begin{align*}
\mathbb E^t_{q \rightarrow i}\left[\exp\left(\sqrt{\frac{\mu p^2}{n}}\widetilde{B}_{qi}\tilde{u}_q\sigma_i\right)\right]&=\exp\left(\sqrt{\frac{\mu p}{n}}\widetilde{B}_{qi}\sigma_im^t_{q \rightarrow i}+\frac{\mu p}{2n}\widetilde{B}^2_{qi}\tau^t_{q\rightarrow i}\right),\nonumber\\
\mathbb{E}^t_{i \rightarrow j;\ell}\left[\frac{d_\ell+\lambda_\ell\sqrt{d_\ell}\sigma_i\sigma_j}{n}\right]&=\frac{d_\ell}{n}\left(1+\frac{\lambda_\ell \sigma_j}{\sqrt{d_\ell}}\tanh(\eta^t_{i \rightarrow j; \ell})\right),\nonumber\\
\mathbb E^t_{i \rightarrow q}\left[\exp\left(\sqrt{\frac{\mu p^2}{n}}\widetilde{B}_{qi}\tilde{u}_q\sigma_i\right)\right]&=\frac{\cosh(\eta^t_{i \rightarrow q}+\sqrt{\mu p^2/n}\widetilde{B}_{qi}\tilde{u}_q)}{\cosh(\eta^t_{i \rightarrow q})}.
\end{align*}
Using \eqref{eq:f_new} and \eqref{eq:rho_new}, we get,
\begin{align}
\label{eq:def_eta_i_j}
\eta^{t+1}_{i \rightarrow j;k}&=\sqrt{\frac{\mu}{\gamma p}}\sum\limits_{q \in [p]}B_{qi}m^t_{q \rightarrow i}+\sum\limits_{\ell=1;\ell \neq k}^m\Bigg[\sum_{c \in \partial_\ell i}f(\eta^t_{c \rightarrow i;\ell};\rho_\ell)-\sum_{c \in (\partial_\ell i)^c}f(\eta^t_{c \rightarrow i;\ell};\rho_{n,\ell})\Bigg]\nonumber\\
&\quad\quad+\sum_{c \in \partial_\ell i\setminus \{j\}}f(\eta^t_{c \rightarrow i;k};\rho_\ell)-\sum_{c \in (\partial_\ell i)^c\setminus \{j\}}f(\eta^t_{c \rightarrow i;k};\rho_{n,\ell}),\nonumber\\
\eta^{t+1}_{i \rightarrow q}&=\sqrt{\frac{\mu}{\gamma p}}\sum\limits_{r \in [p]\setminus\{q\}}B_{ri}m^t_{r \rightarrow i}+\sum\limits_{\ell=1}^m\Bigg[\sum_{c \in \partial_\ell i}f(\eta^t_{c \rightarrow i;\ell};\rho_\ell)-\sum_{c \in (\partial_\ell i)^c}f(\eta^t_{c \rightarrow i;\ell};\rho_{n,\ell})\Bigg]\nonumber\\
\end{align}
Also, define,
\begin{equation}
\label{eq:def_eta}
    \eta^{t+1}_i=\sqrt{\frac{\mu}{\gamma p}}\sum\limits_{q \in [p]}B_{qi}m^t_{q \rightarrow i}+\sum\limits_{\ell=1}^m\Bigg[\sum_{c \in \partial_\ell i}f(\eta^t_{c \rightarrow i;\ell};\rho_\ell)-\sum_{c \in (\partial_\ell i)^c}f(\eta^t_{c \rightarrow i;\ell};\rho_{n,\ell})\Bigg].
\end{equation}
By Taylor Expansion,
\begin{align*}
\log \nu^{t+1}_{q \rightarrow i}(u_q)&=\mbox{const}-\frac{p(1+\mu)}{2}u^2_q+\sum\limits_{j \in [n]\setminus\{i\}}\log \cosh\left(\eta^t_{j \rightarrow q}+p\sqrt{\frac{\mu}{n}}\widetilde{B}_{qj}u_q\right)\nonumber\\
&=\mbox{const}-\frac{p(1+\mu)}{2}u^2_q+\left(p\sqrt{\frac{\mu}{n}}\sum\limits_{j \in [n]\setminus \{i\}}\widetilde{B}_{qj}\tanh(\eta^t_{j \rightarrow q})\right)u_q\nonumber\\
&\quad\quad +\left(\frac{p^2\mu^2}{2n}\sum\limits_{j \in [n]\setminus\{i\}}\widetilde{B}^2_{qj}\sech^2(\eta^t_{j \rightarrow q})\right)u^2_q + O(n^{-1/2}).\nonumber
\end{align*}
This implies,
\begin{align*}
\tau^{t+1}_{q \rightarrow i}&=\left(1+\mu-\frac{\mu}{\gamma}\sum\limits_{j \in [n]\setminus\{i\}}\frac{B^2_{qj}}{p}\sech^2(\eta^t_{j \rightarrow q})\right)^{-1}\nonumber\\
m^{t+1}_{q \rightarrow i}&=\tau^{t+1}_{q \rightarrow i}\sqrt{\frac{\mu}{\gamma}}\sum\limits_{j \in [n]\setminus\{i\}}\frac{B_{qj}}{\sqrt{p}}\tanh(\eta^t_{j \rightarrow q}).\nonumber
\end{align*}
Next, we try to approximate the parameters $\eta^t_{i\rightarrow j;\ell}$ using techniques similar to \cite{ContBlockMod,Decelle_2011}. As $\sup_z f(z,\rho) \le \rho$, if $i$ and $j$ have no edge between them in $\bm G_\ell$, then,
\begin{align*}
\eta^t_{i \rightarrow j;\ell}&=\eta^t_{i}-f(\eta^{t-1}_{j \rightarrow i;\ell};\rho_{n,\ell})\\
&=\eta^t_{i}+o(\rho_{n,\ell})=\eta^t_{i;\ell}+O(n^{-1}).
\end{align*}
Using Taylor expansion,
\begin{align*}
f(\eta^t_{i \rightarrow j;\ell};\rho_{n,l})&=f(\eta^t_{i};\rho_{n,l})+O\left(\frac{\tanh(\rho_{n,\ell})}{n}\right)\\
&=f(\eta^t_{i};\rho_{n,l})+O(n^{-2}).
\end{align*}
Hence, from \eqref{eq:def_eta_i_j} and \eqref{eq:def_eta}, we have the following approximations
\begin{align}
\label{eq:equations_edge_message_1}
\eta_{i\rightarrow j,\ell}^{t+1} = 
    & \sqrt{\frac{\mu}{p\gamma}}\sum_{r = 1}^{p}B_{ri}m^t_{r \rightarrow i}+ \sum_{\substack{r=1\\ r\neq \ell}}^{m}\left\{\sum_{k\in \partial_{r}i}f(\eta^t_{k \rightarrow i; r};\rho_r)-\sum\limits_{k=1}^{n}f(\eta^t_k;\rho_{n,r})\right\}\\
    & + \sum_{k\in\partial_{\ell}i\setminus\{j\}}f(\eta^t_{k \rightarrow i; \ell};\rho_\ell) - \sum\limits_{k=1}^{n}f(\eta^t_k;\rho_{n,\ell})+O(n^{-1}),\\
    \label{eq:equations_edge_message_2}
    \eta^t_i&=\sqrt{\frac{\mu}{p\gamma}}\sum_{r = 1}^{p}B_{ri}m^t_{r \rightarrow i}+ \sum_{r=1}^{m}\left\{\sum_{k\in \partial_{r}i}f(\eta^t_{k \rightarrow i; r};\rho_r)-\sum\limits_{k=1}^{n}f(\eta^t_k;\rho_{n,r})\right\}
\end{align}
Furthermore, we also have
\begin{align*}
    \eta^{t+1}_{i \rightarrow q}&=\sqrt{\frac{\mu}{\gamma p}}\sum\limits_{r \in [p]\setminus\{q\}}B_{ri}m^t_{r \rightarrow i}+\sum\limits_{\ell=1}^m\Bigg[\sum_{c \in \partial_\ell i}f(\eta^t_{c \rightarrow i;\ell};\rho_\ell)-\sum_{k \in [n]}f(\eta^t_{k};\rho_{n,\ell})\Bigg]
\end{align*}
Now we use the following ansatz,
\begin{align}
\label{eq:ansatz}
\eta^t_{i \rightarrow q}&=\eta^t_i+\delta\eta^t_{i \rightarrow q},\\
m^t_{q \rightarrow i}&=m^t_q+\delta m^t_{q \rightarrow i},\\
\tau^t_{q \rightarrow i}&=\tau^t_q+\delta\tau^t_{q \rightarrow i},
\end{align}
where $\delta\eta^t_{i \rightarrow q},\delta m^t_{q \rightarrow i},\delta\tau^t_{q \rightarrow i}$ are each $O(n^{-1/2})$. Now, observing that,
\begin{align*}
\eta^{t+1}_{i \rightarrow q}&=\sqrt{\frac{\mu}{p\gamma}}\sum_{r = 1}^{p}B_{ri}(m_r^{t}+\delta m^t_{r \rightarrow i})+\sum\limits_{\ell=1}^{m}\Bigg\{\sum\limits_{k \in \partial_\ell i}f(\eta^t_{k \rightarrow i;\ell};\rho_\ell)-\sum_{k \in [n]}f(\eta^t_{k;\ell};\rho_{n,\ell})\Bigg\}\\
&=\sqrt{\frac{\mu}{p\gamma}}\sum_{r = 1}^{p}B_{ri}(m_r^{t}+\delta m^t_{r \rightarrow i})+\sum\limits_{\ell=1}^{m}\Bigg\{\sum\limits_{k \in \partial_\ell i}f(\eta^t_{k \rightarrow i;\ell};\rho_\ell)-\sum_{k \in [n]}f(\eta^t_{k;\ell};\rho_{n,\ell})\Bigg\}\\
&\hskip 3em -\sqrt{\frac{\mu}{\gamma p}}(B_{qi}m^t_q+B_{qi}m^t_{q \rightarrow i}\delta).
\end{align*}
Since, the term $B_{qi}m^t_{q \rightarrow i}\delta/\sqrt{p}=O(n^{-1})$, we can ignore it. Hence, we have the following approximation,
\begin{align}
\label{eq:eta_update}
\eta^{t+1}_i&=\sqrt{\frac{\mu}{p\gamma}}\sum_{r = 1}^{p}B_{ri}(m_r^{t}+\delta m^t_{r \rightarrow i})+\sum\limits_{\ell=1}^{m}\Bigg\{\sum\limits_{k \in \partial_\ell i}f(\eta^t_{k \rightarrow i;\ell};\rho_\ell)-\sum_{k \in [n]}f(\eta^t_{k;\ell};\rho_{n,\ell})\Bigg\}\nonumber\\
\delta\eta^{t+1}_{i \rightarrow q}&\approx-\sqrt{\frac{\mu}{\gamma p}}B_{qi}m^t_q\nonumber
\end{align}
Using \eqref{eq:ansatz} and following the techniques described in (128)-(130) of \cite{ContBlockMod}, we get the following approximation,
\begin{align}
\tau^{t+1}_q&=\left(1+\mu-\frac{\mu}{p\gamma}\sum\limits_{j \in [n]}B^2_{qj}\sech^2(\eta^{t}_j)\right)^{-1}\nonumber\\
\delta\tau^{t+1}_q&\approx 0.
\end{align}
Similarly, using approximations similar to (133)-(136) of \cite{ContBlockMod}, we get,
\begin{align}
\label{eq:m_t_update}
m^{t+1}_q&=\tau^{t+1}_q\sqrt{\frac{\mu}{\gamma}}\sum\limits_{j \in [n]}\frac{B_{qj}}{\sqrt{p}}\tanh(\eta^t_j)-\frac{\mu\tau^{t+1}_q}{\gamma}\left[\sum\limits_{j \in [n]}\frac{B^2_{qj}}{p}\sech^2(\eta^t_j)\right]m^{t-1}_q\nonumber\\
\delta m^{t+1}_q&\approx -\tau^{t+1}_q\sqrt{\frac{\mu}{p\gamma}}B_{qi}\tanh(\eta^{t}_i).
\end{align}
Plugging in \eqref{eq:m_t_update} and \eqref{eq:eta_update} in \eqref{eq:equations_edge_message_1} and \eqref{eq:equations_edge_message_2}, we get the Belief Propagation updates \eqref{eq:bp_updates}-\eqref{eq:bp_updates_end}.}
\bibliographystyle{abbrvnat}
\bibliography{reference}

\end{document}